\documentclass[11pt]{amsart}
\usepackage{amsmath,amssymb}
\usepackage{a4wide}
\usepackage[all]{xy}

\newtheorem{theorem}{Theorem}[section]
\newtheorem{problem}[theorem]{Problem}

\newtheorem{proposition}[theorem]{Proposition}
\newtheorem{lemma}[theorem]{Lemma}
\newtheorem{claim}[theorem]{Claim}
\newtheorem{corollary}[theorem]{Corollary}
\theoremstyle{definition}
\newtheorem{remark}[theorem]{Remark}
\newtheorem{example}[theorem]{Example}

\newtheorem{definition}[theorem]{Definition}

\newcommand{\odim}{\overline{\dim}}
\newcommand{\udim}{\underline{\dim}}

\newcommand{\Set}{\mathbf{Set}}
\newcommand{\Mon}{\mathbf{Mon}}

\newcommand{\Metr}{\mathbf{Metr}}

\newcommand{\DMetr}{\mathbf{Dist}}
\newcommand{\Dist}{\mathbf{Dist}}

\newcommand{\w}{\omega}
\newcommand{\supp}{\mathrm{supp}}

\newcommand{\IN}{\mathbb N}
\newcommand{\IZ}{\mathbb Z}

\newcommand{\IE}{\mathbb E}

\newcommand{\Lip}{\mathrm{Lip}}

\newcommand{\e}{\varepsilon}
\newcommand{\dom}{\mathrm{dom}}
\newcommand{\IR}{\mathbb R}
\newcommand{\diam}{\mathrm{diam}}
\newcommand{\id}{\mathrm{id}}
\newcommand{\ji}{{\ddot\imath}}
\newcommand{\length}{\ell}
\newcommand{\ud}{\underline{d}}
\newcommand{\Ra}{\Rightarrow}
\newcommand{\U}{\,\mathcal U}
\newcommand{\V}{\mathcal V}
\newcommand{\W}{\mathcal W}

\newcommand{\Dim[1]}{\dim_{#1}}
\newcommand{\cov}{\mathrm{cov}}
\newcommand{\F}{\mathcal F}
\newcommand{\mesh}{\mathrm{mesh}}

\begin{document}

\title{The $\ell^p$-metrization of functors with finite supports}
\author{T.~Banakh,  V.~Brydun, L.~Karchevska, M.~Zarichnyi}

\address{Ivan Franko National University of Lviv  and 
Jan Kochanowski University in Kielce}
\email{t.o.banakh@gmail.com}

\address{Ivan Franko National University of Lviv, 
Universytetska 1, Lviv, 79000, Ukraine}
\email{v\_frider@yahoo.com}

\address{Ivan Franko National University of Lviv,
Universytetska 1, Lviv, 79000, Ukraine}
\email{lesia.karchevska@gmail.com}

\address{Ivan Franko National University of Lviv  and 
Uniwersytet Rzeszowski}
\email{zarichnyi@yahoo.com}

\begin{abstract}
 Let $p\in[1,\infty]$ and $F:\Set\to\Set$ be a functor with finite supports in the category $\Set$ of sets. Given a non-empty metric space $(X,d_X)$, we introduce the distance $d^p_{FX}$ on the functor-space $FX$ as the largest distance such that for every $n\in\IN$ and $a\in Fn$ the map $X^n\to FX$, $f\mapsto Ff(a)$, is non-expanding with respect to the $\ell^p$-metric $d^p_{X^n}$ on $X^n$. We prove that the distance $d^p_{FX}$ is a pseudometric if and only if the functor $F$ preserves singletons; $d^p_{FX}$ is a metric if $F$ preserves singletons and one of the following conditions holds: (1) the metric space $(X,d_X)$ is Lipschitz disconnected, (2) $p=1$, (3) the functor $F$ has finite degree, (4) $F$ preserves supports. We prove that for any Lipschitz map $f:(X,d_X)\to (Y,d_Y)$ between metric spaces the map $Ff:(FX,d^p_{FX})\to (FY,d^p_{FY})$ is Lipschitz with Lipschitz constant $\Lip(Ff)\le \Lip(f)$. If the functor $F$ is finitary, has finite degree (and preserves supports), then $F$ preserves uniformly continuous function, coarse functions, coarse equivalences, asymptotically Lipschitz functions, quasi-isometries (and continuous functions). For many dimension functions  we  prove the formula $\dim F^pX\le\deg(F)\cdot\dim X$. Using injective envelopes, we introduce a modification $\check d^p_{FX}$ of the distance $d^p_{FX}$ and prove that the functor $\check F^p:\DMetr\to\DMetr$, $\check F^p:(X,d_X)\mapsto (FX,\check d^p_{FX})$, in the category $\DMetr$ of distance spaces preserves Lipschitz maps and isometries between metric spaces.
\end{abstract}
\maketitle

\newpage 
\tableofcontents
\newpage

{
\section{Introduction}
\baselineskip14pt

Let $\Set$ be the category whose objects are sets and morphisms are functions between sets, and $\Metr$ be the category whose objects are metric spaces and morphisms are arbitrary maps between metric spaces. The category $\Metr$ is a subcategory of the category $\DMetr$ whose objects are distance spaces and morphisms are arbitrary maps between distance spaces.

By a {\em distance space} we understand a pair $(X,d_X)$ consisting of a set $X$ and a {\em distance} $d_X$ on $X$, which is a function $d_X:X\times X\to[0,\infty]$ satisfying three well-known axioms:
\begin{itemize}
\item $d_X(x,x)=0$,
\item $d_X(x,y)=d_X(y,x)$,
\item $d_X(x,z)\le d_X(x,y)+d_X(y,z)$,
\end{itemize}
holding for any points $x,y,z\in X$. Here we assume that $\infty+x=\infty=x+\infty$ for any $x\in[0,\infty]$.

An alternative term for a distance is an $\infty$-pseudometric. A distance $d_X$ on a set $X$ is called an {\em $\infty$-metric} if $d_X(x,y)>0$ for any distinct points $x,y\in X$.
A distance $d_X$ is a {\em pseudometric} if $d_X(x,y)<\infty$ for any points $x,y\in X$.

In this paper, for every functor $F:\Set\to\Set$ with finite supports, every distance space $(X,d_X)$ and every $p\in[1,\infty]$, we introduce a distance $d^p_{FX}$ on the functor-space $FX$ and study its properties. The distance $d^p_{FX}$ is defined as the largest distance on $FX$ such that for every $n\in\IN$ and $a\in Fn$ the map $\xi^a_X:X^n\to FX$, $\xi^a_X:f\mapsto Ff(a)$, is non-expanding with respect to the  $\ell^p$-distance on the finite power  $X^n$ of $X$. 
The correspondence $F^p:\DMetr\to\DMetr$, $F^p:(X,d_X)\mapsto(FX,d^p_{FX})$, determines a functor such that $U\circ F^p=F\circ U$ where $U:\DMetr\to\Set$ is the (forgetful) functor assigning to each distance space $(X,d)$ its underlying set $X$.
So, $F^p$ is a lifting of the functor $F$ to the category $\DMetr$, called the {\em $\ell^p$-metrization} of the functor $F$. In the paper we establish some properties of the functor $F^p$; in particular, we prove that this functor preserves Lipschitz functions between distance spaces. In Theorem~\ref{t:main} we prove that for any metric space $(X,d_X)$ the distance $d^p_{FX}$ is a metric if the functor $F$ preserves singletons and one of the following conditions holds: (1) the distance space $(X,d_X)$ is Lipschitz disconnected, (2) $p=1$, (3) $F$ has finite degree, (4) $F$ preserves supports. If the functor $F$ is finitary, has finite degree (and preserves supports), then $F$ preserves uniformly continuous functions, coarse functions, coarse equivalences, asymptotically Lipschitz functions, quas-isometries (and continuous functions). For many dimension functions $\dim$ we  prove the formula $\dim F^pX\le\deg(F)\cdot\dim X$. Using the injective envelopes, we also define a modification $\check d^p_{FX}$ of the distance $d^p_{FX}$, which has some additional properties, for example, preserves isometries between distance spaces. The properties of the functors $F^p$ and $\check F^p$ are summed up in Theorems~\ref{t:main}, \ref{t:main-fdeg}, \ref{t:mainU}.

In the last three sections we will study the $\ell^p$-metrizations of some concrete functors with finite supports: of $n$th power, of free abelian (Boolean) group, and of free semilattice.

The results obtained in this paper can be considered as a development of the results of Shukel \cite{Shukel} and Radul \cite{RS} who studied the $\ell^p$-metrizations of normal functors for $p=\infty$. An axiomatic approach to metrizability of functors was developed by Fedorchuk \cite{Fed}, see \cite[\S3.9]{TZ}. 
}

\section{Some notations and conventions}

We denote by $\w=\{0,1,2,\dots\}$ the set of all finite ordinals and by $\mathbb N=\w\setminus\{0\}$ the set of positive integer numbers. 
Each finite ordinal $n\in\w$ is identified with the set $n:=\{0,\dots,n-1\}$ of smaller ordinals. Cardinals are identified with the smallest ordinals of given cardinality.

For a set $X$ and a cardinal $\kappa$ we use the following notations:
\begin{itemize}
\item $\kappa^+$ is the smallest cardinal, which is larger than $\kappa$;
\item $|X|$ is the cardinality of $X$;
\item $[X]^\kappa:=\{A\subseteq X:|A|=\kappa\}$;
\item $[X]^{<\kappa}:=\{A\subseteq X:|A|<\kappa\}$;
\item $X^\kappa$ is the family of functions from $\kappa$ to $X$;
\item $X^{<\w}=\bigcup_{n\in\w}X^n$ the family of all finite sequences of elements of $X$. 
\end{itemize} 

For a function $f:X\to Y$ between sets and subsets $A\subseteq X$, $B\subseteq Y$ by $f{\restriction}_A$ we denote the restriction of $f$ to $A$, by $f[A]:=\{f(a):a\in A\}\subseteq Y$ the image of the set $A$ under the function $f$, and by $f^{-1}[B]=\{x\in X:f(x)\in B\}$ the preimage of $B$. Such notations are convenient for handling functions on natural numbers $n:=\{0,\dots,n-1\}$. In this case for a function $f:n\to Y$ and element $i\in n$ we have that $f(i)$ is an element of $Y$ but $f[i]$ is the subset $\{f(x):x\in i\}$ of $Y$.  For two sets $A,B$ we denote by $A\triangle B$ their symetric difference $(A\cup B)\setminus (A\cap B)$.

We extend the basic arithmetic operations  to the extended half line $[0,\infty]$ letting $x+\infty=\infty+x=\infty$ and $\infty\cdot 0=0=x-x$ for all $x\in[0,\infty]$, and $$x\cdot\infty=\infty\cdot x=\infty^x=\infty-x=|x-\infty|=|\infty-x|=\infty$$ for any $x\in(0,\infty]$.

\section{Distance spaces}\label{s:distance}

In this section we present some basic definitions related to distance spaces.

We recall that a {\em distance space} is a pair $(X,d_X)$ consisting of a set $X$ and a {\em distance} $d_X$ on $X$, which is a function $d_X:X\times X\to[0,\infty]$ such that
\begin{itemize}
\item $d_X(x,x)=0$,
\item $d_X(x,y)=d_X(y,x)$,
\item $d_X(x,z)\le d_X(x,y)+d_X(y,z)$,
\end{itemize}
for any points $x,y,z\in X$. 

 A distance $d_X$ on a set $X$ is called 
 \begin{itemize}
 \item an {\em $\infty$-metric} if $d_X(x,y)>0$ for any distinct points $x,y\in X$;
 \item  a {\em pseudometric} if $d_X(x,y)<\infty$ for any points $x,y\in X$.
\end{itemize}

Each set $X$ carries the {\em $\{0,\infty\}$-valued $\infty$-metric} $d:X\times X\to\{0,\infty\}$ defined by
$$d(x,y)=\begin{cases}0,&\mbox{$x=y$};\\
\infty,&\mbox{$x\ne y$}.
\end{cases}
$$

An important example of a distance space is the extended real line $\bar\IR=[-\infty,+\infty]$, endowed with the $\infty$-metric
$$d_{\bar\IR}(x,y)=\begin{cases}
|x-y|,&\mbox{if $x,y<\infty$};\\
0,&\mbox{if $x=y$};\\
\infty,&\mbox{otherwise}.
\end{cases}
$$

For every set $X$ the distance $d_{\bar\IR}$ induces the $\sup$-distance 
$$d_{\bar\IR^X}(f,g)=\sup_{x\in X}d_{\bar\IR}(f(x),g(x))$$on the set $\bar\IR^X$ of all functions from $X$ to $\bar \IR$. 




Every distance $d_X$ on a set $X$ induces the {\em Hausdorff distance}  
$$d_{HX}(A,B)=
\begin{cases}
\max\{\sup\limits_{a\in A}\inf\limits_{b\in B}d_X(a,b),\sup\limits_{b\in B}\inf\limits_{a\in A}d_X(a,b)\},&\mbox{if $A\ne \emptyset\ne B$},\\
0,&\mbox{if $A=B$},\\
\infty,&\mbox{otherwise},
\end{cases}
$$on the set $HX$ of all subsets of $X$. 
It is a standard exercise to check that $d_{HX}$ satisfies the three axioms of distance.

For a point $x\in X$ of a distance space $(X,d_X)$ and a number $r\in[0,\infty]$ let
$$O(x;r):=\{y\in X:d_X(x,y)<r\}\mbox{ and }O[x;r]:=\{y\in X:d_X(x,y)\le r\}$$be the open and closed balls of radius $r$ around $x$, respectively. 

For a subset $A\subseteq X$ let 
$$O[A;r)=\bigcup_{a\in A}O(a,r)$$be the {\em $\e$-neighborhood} of $A$ in $(X,d_X)$.

The family $\{O(x;\infty):x\in X\}$ is the partition of the distance space $(X,d_X)$ into {\em pseudometric components} (observe that the restriction of the distance $d_X$ to any pseudometric component $O(x;\infty)$ is a pseudometric). 

For a subset $A$ of a distance space $(X,d_X)$, the numbers
$$
\begin{aligned}
\diam(A)&:=\sup(\{0\}\cup\{d_X(x,y):x,y\in A\}),\\
\overline{d}_X(A)&:=\sup(\{0\}\cup\{d_X(x,y):x,y\in A,\;d_X(x,y)<\infty\})\mbox{ and }\\
\ud_X(A)&:=\inf(\{\infty\}\cup\{d_X(x,y):x,y\in A,\;x\ne y\})
\end{aligned}
$$
are called the {\em diameter}, the {\em real diameter} and the {\em separatedness number} of $A$.

A subset $A$ of a distance space $X$ is {\em bounded} if $\diam(A)<\infty$.

A distance space $(X,d_X)$ is called {\em separated} if $\ud_X(X)>0$.
\smallskip

Also for two subsets $A,B\subseteq X$ we put $$\ud_X(A,B)=\inf(\{\infty\}\cup \{d_X(a,b):a\in A,\;b\in B\}).$$
It is easy to see that $\ud_X(A,B)\le d_{HX}(A,B)$ for any nonempty sets $A,B\subseteq X$.

For a family $\U$ of subsets of a distance space $X$ we put 
$$
\begin{aligned}
&\mesh(\U):=\sup\big(\{0\}\cup\{\diam(U):U\in\U\}\big)\mbox{ \ and \ }\\
&\ud_X(\U)=\inf\big(\{\infty\}\cup\{\ud_X(U,V):U,V\in\U,\;\;U\ne V\}\big).
\end{aligned}
$$

\section{Functions between distance spaces}

In this section we introduce some classes of functions bewteen distance spaces.


A function $f:X\to Y$ between distance spaces $(X,d_X),(Y,d_Y)$ is called
\begin{itemize}
\item {\em an isometry} if $d_Y(f(x),f(y))=f_X(x,y)$ for every $x,y\in X$;
\item {\em non-expanding} if $d_Y(f(x),f(y))\le d_X(x,y)$ for every $x,y\in X$;
\item {\em Lipschitz} if there exists a real number $L$ such that $d_Y(f(x),f(y))\le L\cdot d_X(x,y)$ for all $x,y\in X$; the smallest such number $L$ is denoted by $\Lip(f)$ and called the {\em Lipschitz constant} of $f$;
\item {\em asymptotically Lipschitz} if there exist positive real numbers $L,C$ such that\newline $d_Y(f(x),f(y))\le L\cdot d_X(x,y)+C$ for any $x,y\in X$;
\item {\em continuous} if for any $x\in X$ and $\e\in(0,\infty]$ there exists $\delta\in(0,\infty]$ such that for any $y\in X$ with $d_X(x,y)\le\delta$ we have  $d_X(f(x),f(y))\le\e$;
\item {\em uniformly continuous} (or else {\em microform}) if for any $\e\in(0,\infty]$ there exists $\delta\in(0,\infty]$ such that for any $x,y\in X$ with $d_X(x,y)\le\delta$ we have  $d_X(f(x),f(y))\le\e$;
\item {\em coarse} (or else {\em macroform}) if for every $\delta\in[0,\infty)$ there exists $\e\in[0,\infty)$ such that for any $x,y\in X$ with $d_X(x,y)\le\delta$ we have  $d_X(f(x),f(y))\le\e$;
\item {\em duoform} if $f$ is both microform and macroform;
\item {\em bornologous} if for any bounded set $B$ in $X$ the image $f[B]$ is bounded in $Y$;
\item a {\em coarse equivalence} if $f$ is coarse  and there exists a coarse map $g:Y\to X$ such that $\sup_{x\in X}d_X(x,g\circ f(x))<\infty$ and $\sup_{y\in Y}d_Y(y,f\circ g(y))<\infty$;
\item a {\em quasi-isometry} if $f$ is an asymptotically Lipschitz map and there exists an asymptotically Lipschitz map $g:Y\to X$ such that $\sup_{x\in X}d_X(x,f\circ g(x))<\infty$ and $\sup_{y\in Y}d_Y(y,f\circ g(y))<\infty$.
\end{itemize}

The implications between these properties are described in the following diagram.
{
$$
\xymatrix
{
&\mbox{uniformly}\atop\mbox{continuous}\ar@{=}[r]&\mbox{microform}\ar@{=>}[r]&\mbox{continuous}\\
\mbox{non-expanding}\ar@{=>}[r]&\mbox{Lipschitz}\ar@{=>}[r]\ar@{=>}[d]&\mbox{duoform}\ar@{=>}[u]\ar@{=>}[d]\\
\mbox{isometry}\ar@{=>}[u]\ar@{=>}[r]&\mbox{asymptotically}\atop\mbox{Lipschitz}\ar@{=>}[r]&\mbox{macroform}\ar@{=}[r]&\mbox{coarse}\ar@{=>}[d]\\
\mbox{bijective isometry}\ar@{=>}[r]\ar@{=>}[u]&\mbox{quasi-isometry}\ar@{=>}[r]\ar@{=>}[u]&\mbox{coarse}\atop\mbox{equivalence}\ar@{=>}[u]\ar@{=>}[r]&\mbox{bornologous}
}
$$
}

\begin{remark} An isometry between distance spaces needs not be  injective. On the other hand, an isometry between $\infty$-metric spaces is necessarily injective.
\end{remark}


For a function $f:X\to Y$ between distance spaces, its {\em continuity modulus} is the function $\w_f:(0,\infty]\to[0,\infty]$ defined by $$\w_f(\e):=\sup\{d_Y(f(x),f(y)):x,y\in X\;\wedge\;d_X(x,y)<\e\} \mbox{  \ for $\e\in(0,\infty]$}.$$ 
The definition implies that the continuity modulus $\w_f$ is a non-decreasing function from $(0,\infty]$ to $[0,\infty]$.

Many properties of functions can be expressed via the continuity modulus.

In particular, a function $f:X\to Y$ between distance spaces is
\begin{itemize} 
\item non-expanding iff $\w_f(\e)\le\e$ for all $\e\in(0,\infty)$;
\item Lipschitz iff there exits a real number $L$ such that\newline $\w_f(\e)\le L\cdot\e$ for all $\e\in(0,\infty)$;
\item asymptotically Lipschitz iff there exist real numbers $L,C$ such that\newline $\w_f(\w)\le L\cdot\e+C$ for all $\e\in(0,\infty)$;
\item microform iff for any $\e\in(0,\infty)$ there exists $\delta\in(0,\infty)$ such that $\w_f(\delta)\le\e$;
\item macroform iff for any $\delta\in(0,\infty)$ there exists $\e\in(0,\infty)$ such that $\w_f(\delta)\le\e$.
\end{itemize}

\smallskip

The following two propositions are well-known for metric spaces, see \cite[Lemma 1.1]{BL}.

\begin{proposition}\label{p:emb} For any distance space $(X,d_X)$, the map 
$$i:X\to \bar\IR^X,\;\;i:x\mapsto d_X(x,\cdot),$$
is an isometry.
\end{proposition}

\begin{proof} Given any point $x\in X$ let $d_x:X\to\bar\IR$ be the function defined by $d_x(y)=d_X(x,y)$ for $y\in X$. For any $x,y,z\in X$, the triangle inequality for the distance $d_X$ implies $d_{\bar\IR}(d_x(z),d_y(z))\le d_X(x,y)$ and then $d_{\bar\IR^X}(d_x,d_y)\le d_X(x,y)$. On the other hand, $d_{\bar\IR^X}(d_x,d_y)\ge d_{\bar\IR}(d_x(x),d_x(y))=d_X(x,y)$. Therefore, $d_{\bar \IR^X}(d_x,d_y)=d_X(x,y)$, and the map $i:X\to\bar\IR^X$, $i:x\mapsto d_x$, is an isometry.
\end{proof}

\begin{proposition}\label{p:Lip} For any isometry $i:X\to Y$ between distance spaces and any Lipschitz function $f:X\to \bar\IR$, there exists a Lipschitz function  $\bar f:Y\to\bar\IR$ such that $\bar f\circ i=f$ and $\Lip(\bar f)=\Lip(f)$.
\end{proposition} 

\begin{proof} For every $y\in Y$, consider the set $O(y;\infty)=\{z\in Y:d_Y(z,y)<\infty\}$ and observe that $\mathcal O=\{O(y;\infty):y\in Y\}$ is a partition of $Y$ into pseudometric subspaces (called the {\em pseudometric components} of $Y$). For every $B\in\mathcal O$ we will define a Lipschitz map $\bar f_B:B\to\bar\IR$ with Lipschitz constant $\Lip(f_B)\le \Lip(f)$ such that $\bar f_B\circ i{\restriction}_B=f{\restriction}_B$.

The Lipschitz property of $f$ implies that $f[i^{-1}[B]]$ is contained in one of the sets: $\IR$, $\{{-}\infty\}$ or $\{{+}\infty\}$. If $f[i^{-1}[B]]\subseteq\{{-}\infty,{+}\infty\}$, then let $f_B:B\to\{{-}\infty,{+}\infty\}$ be any constant map. It remains to consider the case $\emptyset\ne f[i^{-1}[B]]\subseteq\IR$. In this case we can define a map $\bar f_B:B\to\IR$ by the explicit formula $$\bar f_B(y)=\inf_{x\in i^{-1}[B]}\big(f(x)+\Lip(f)\cdot d_Y(i(x),y)\big)\mbox{ \ for any $y\in B$}.$$
It is easy to see that $\bar f_B(i(x))=f(x)$ for any $x\in i^{-1}[B]$. Repeating the argument of the proof of Lemma 1.1 in \cite{BL}, we can show that the function $\bar f_B:B\to\IR$ is well-defined, Lipschitz, and has Lipschitz constant $\Lip(\bar f_B)\le\Lip(f)$. Finally, define the function $\bar f:Y\to\bar\IR$ letting $\bar f{\restriction}_B=\bar f_B$ for every $B\in\mathcal O$, and observe that $\bar f$ has the desired property.
\end{proof}



\section{The injective envelope of a distance space}\label{s:E(X)}

In this section we extend the classical Isbell construction \cite{Isbell} (see also \cite{Dress}, \cite{Her}, \cite[\S3]{Lang}) of the injective envelope of a metric space  to the category of distance spaces.  Injective envelopes will be applied in Section~\ref{s:stable} for defining the $\check\ell^p$-metrization of functors.



For a distance space $X$ with distance $d_X$, let $\Delta X$ be the set of all functions $f:X\to\bar\IR$ such that $$f(x)+f(y)\ge d_X(x,y)\mbox{ \  for all \ $x,y\in X$}.
$$ For $x=y$ we have $0=d_X(x,x)\le 2{\cdot}f(x)$, which means that functions $f\in\Delta X$ take their values in the extended half-line $\bar \IR_+=[0,\infty]$.

 The space $\Delta X$ is endowed with the $\infty$-metric $d_{\Delta X}$  inherited from $\bar\IR^X$:
 $$d_{\Delta X}(f,g)=\sup_{x\in X}d_{\bar\IR}(f(x),g(x))\mbox{ \  for \ }f,g\in\Delta X.$$
The set $\Delta X$ also carries a natural partial order $\le$ defined by $f\le g$ iff $f(x)\le g(x)$ for all $x\in X$. For two function $f,g\in\Delta X$ we write $f<g$ if $f\le g$ and $f\ne g$.

A function $f\in\Delta X$ is called {\em extremal} if $f$ is a minimal element of the partially ordered set $(\Delta X,\le)$. Let $EX$ be the distance subspace of $\Delta X$ consisting of extremal functions. By $d_{E X}$ we denote the $\sup$-distance of the space $E X$. The distance $d_{EX}$ is an $\infty$-metric, being a restriction of the $\infty$-metric $d_{\bar\IR^X}$.

The distance space $E X$ is called the {\em injective envelope} of the distance space $X$. Now we establish some important properties of injective envelopes, generalizing well-known properties of injective envelopes of metric spaces.

\begin{lemma}\label{l:extr1} For every $z\in X$, the function $$d_z:X\to[0,\infty],\;\;d_z:x\mapsto d_X(x,z),$$ is extremal.
\end{lemma}

\begin{proof}  The triangle inequality for the distance $d_X$ guarantees that the function $d_z$ belongs to $\Delta X$. Assuming that $d_z$ is not extremal, we can find a function $f\le d_z$ in $\Delta X$ such that $f(x)<d_z(x)$ for some $x\in X$.  Then the number $f(x)$ is finite, $f(z)\le d_z(z)=0$ and hence $$d_z(x)=d_X(x,z)\le f(x)+f(z)=f(x)+0=f(x)<d_z(x),$$
which is a desired contradiction.
\end{proof}

\begin{lemma}\label{l:embE} The function $\mathsf e_X:X\to E X$, $\mathsf e_X:x\mapsto d_x$, is a well-defined isometry between the distance spaces $X$ and $E X$.
\end{lemma}

\begin{proof} By Lemma~\ref{l:extr1}, the function $\mathsf e_X$ is well-defined. For any points  $x,y,z\in X$, by the triangle inequality, we have $$d_x(z)=d_X(x,z)\le d_X(x,y)+d_X(y,z)=d_X(x,y)+d_y(z),$$ which implies that $d_{\bar\IR}(d_x(z),d_y(z))\le d_X(x,y)$. Taking the supremum by $z\in X$, we obtain
$$d_{E X}(d_x,d_y)=\sup_{z\in X}d_{\bar\IR}(d_x(z),d_y(z))\le d_X(x,y),$$
which means that the map $\mathsf e_X:X\to E X$ is non-expanding.

On the other hand, $$d_{X}(x,y)=d_{\bar\IR}(d_x(x),d_y(x))\le \sup_{z\in X}d_{\bar\IR}(d_x(z),d_y(z))=d_{E X}(d_x,d_y)$$witnessing that the map $\mathsf e_X:X\to E X$, is an isometry.
\end{proof} 



\begin{lemma}\label{l:main}  There exists a non-expanding map $p:\Delta X\to E X$ such that $p(f)\le f$ for every $f\in\Delta X$.
\end{lemma}

\begin{proof} A subset $B\subseteq X$ is called {\em $\infty$-bounded} if $d_X(a,b)<\infty$ for any points $a,b\in B$. Denote by $\mathcal B_\infty(X)$ the family of all $\infty$-bounded sets in $X$.

For any $\infty$-bounded set $A\in \mathcal B_\infty(X)$, consider the subset $$\Delta_AX=\big\{f\in\Delta X:\infty\notin f[A],\;\;f[X\setminus A]\subseteq\{\infty\}\big\}.$$ 

\begin{claim} $\Delta X=\bigcup\limits_{A\in\mathcal B_\infty(X)}\Delta_AX$.
\end{claim}

\begin{proof} Given any function $f\in\Delta X$, consider the set $A=f^{-1}[\IR]$. We claim that $A\in\mathcal B_\infty(X)$. In the opposite case, we would find two points $a,b\in A$ with $d_X(a,b)=\infty$ and obtain the contradiction:
$$\infty>f(a)+f(b)\ge d_X(a,b)=\infty.$$
The definition of the set $A$ ensures that $f\in\Delta_AX$.
\end{proof}

It is easy to see that for any distinct sets $A,B\in \mathcal B_\infty(X)$ and any functions $f\in \Delta_AX$, $g\in \Delta_BX$, we have $d_{\Delta X}(f,g)=\infty$. Therefore, it suffices for any set $A\in\mathcal B_\infty(X)$ to construct a non-expanding function $p_A:\Delta_AX\to E X$ such that $p_A(f)\le f$ for every $f\in\Delta_AX$.

If the set $A$ is empty, then $\Delta_AX$ is a singleton and any map $p_A:\Delta_AX\to EX$ is non-expanding and has $p_A(f)\le f$ for the unique element $f\in \{\infty\}^X=\Delta_AX$.

So, we assume that $A$ is not empty. In this case we can fix a point $a\in A$ and conclude that $A\subseteq B$ where $B=\{x\in X:d_X(x,a)<\infty\}$.

\begin{claim} \label{cl:qA}
There exists a non-expanding map $q_A:\Delta_AX\to\Delta_BX$ such that $q_A(f)\le f$ for any $f\in \Delta_AX$.
\end{claim}

\begin{proof}  Consider the map $q_A:\Delta_AX\to \Delta_BX$ assigning to each function $f\in\Delta_AX$ the function $g:X\to \bar\IR$ defined by 
$$g(x)=\begin{cases}
f(a)+d_X(a,x)&\mbox{if $x\in B\setminus A$};\\
f(x)&\mbox{otherwise}.
\end{cases}
$$
We claim that $g\in \Delta X$. Given any points $x,y\in X$, we should check that $g(x)+g(y)\ge d_X(x,y)$. If $x$ or $y$ do not belong to $B$, then $g(x)+g(y)=\infty\ge d_X(x,y)$ and we are done. So, assume that $x,y\in B$. 
If $x,y\notin B\setminus A$, then $g(x)+g(y)=f(x)+f(y)\ge d_X(x,y)$. If $x,y\in B\setminus A$, then 
$g(x)+g(y)=2f(a)+d_X(a,x)+d_X(a,y)\ge d_X(x,y)$. 
If $x\in B\setminus A$ and $y\in A$, then $g(x)+g(y)=d_X(a,x)+f(a)+f(y)\ge d_X(a,x)+d_X(a,y)\ge d_X(x,y)$. By analogy we can treat the case $x\in A$ and $y\in B\setminus A$. Therefore, $g\in\Delta X$. The definition of $g$ guarantees that $g\in\Delta_BX$ and $g\le f$.

It remains to show that the map $q_A$ is non-expanding. Choose arbitrary functions $f,f'\in\Delta_AX$, and consider the functions $g=q_A(f)$ and $g'=q_A(f')$.
If $x\notin B\setminus A$, then $d_{\bar \IR}(g(x),g'(x))=d_{\bar \IR}(f(x),f'(x))\le d_{\Delta X}(f,f')$. If $x\in B\setminus A$, 
then $$d_{\bar\IR}(g(x),g'(x))=d_{\bar\IR}(f(a)+d_X(a,x),f'(a)+d_X(a,x))= d_{\bar\IR}(f(a),f'(a))\le d_{\Delta X}(f,f').$$
\end{proof}

 Let $E_BX:=E X\cap\Delta_BX$.
 
 \begin{claim}\label{cl:EB} A function $f:X\to \bar\IR$ belongs to $E_BX$ if and only if $f[X\setminus B]\subseteq\{\infty\}$ and $f(x)=\sup_{y\in B}(d_X(x,y)-f(y))<\infty$
 for every $x\in B$.
 \end{claim}
 
 \begin{proof} To prove the ``only if'' part, assume that $f\in E_BX$. Then $f\in \Delta_BX$ and hence $f[X\setminus B]\subseteq\{\infty\}$ and 
 $\infty>f(x)\ge \sup_{y\in B}(d_X(x,y)-f(y))$ for all $x\in B$. Assuming that $f(x)\ne 
  \sup_{y\in B}(d_X(x,y)-f(y))$ for some $x\in B$, we conclude that $f(x)>s$ where $s=\sup_{y\in B}(d_X(x,y)-f(y))$.  Since $s+f(y)\ge d_X(x,y)$ for every $y\in X$, the function $g: X \to \bar\IR$ given by $g{\restriction}_{X{\setminus}\{x\}}=f{\restriction}_{X{\setminus}\{x\}}$ and $g(x)=\frac12(f(x) +  s)$ belongs to $\Delta_BX$. Since $g < f$, the function $f$ is not extremal, which contradicts the choice of $f$. This contradiction shows that  $f(x)=\sup_{y\in B}(d_X(x,y)-f(y))$ for all $x\in B$. 
 \smallskip
  
Next, we prove the ``if'' part. Choose any function $f:X\to\bar\IR$ such that $f[X\setminus B]\subseteq\{\infty\}$ and $f(x)=\sup_{y\in B}(d_X(x,y)-f(y))<\infty$ for all $x\in B$. Then $f(x)+f(y)\ge d_X(x,y)$ for all $x,y\in X$ and hence $f\in\Delta_BX$. Assuming that $f\notin E X$, we can find a function $g\le f$ in $\Delta X$ such that $g(x)<f(x)$ for some $x\in X$. Since $\infty>f(a)+g(x)\ge g(a)+g(x)\ge d_X(a,x)$, the point $x$ belongs to the pseudometric component $B$ of $X$. The inclusion $g\in\Delta X$ and the inequality $g\le f$ imply
 $$\sup_{y\in B}(d_X(x,y)-f(y))\le\sup_{y\in B}(d_X(x,y)-g(y))\le g(x)<f(x),$$which contradicts our assumption. This contradiction shows that $f\in E_BX$.
 \end{proof}

\begin{claim}\label{cl:pB} There exists a non-expanding map $p_B:\Delta_BX\to E_BX$ such that $p_B(f)\le f$ for any $f\in \Delta_BX$.
\end{claim}

\begin{proof} We modify the argument of Dress \cite[\S1.9]{Dress}, following the lines of the proof of Proposition 2.1 in \cite{Lang}. For every $f\in\Delta_BX$, define the function $f^*:X\to\bar\IR$ by letting $f^*(x)=\sup_{z\in B}(d_X(x,z)-f(z))$ for $x\in B$ and $f^*(x)=\infty$ for  $x\in X\setminus B$. The inclusion $f\in\Delta_BX$ implies $f^*\le f$.

For every pair of points $x,y\in X$, the definition of $f^*$ implies $f(x)+f^*(y)\ge d_X(x,y)$ and $f^*(x)+f(y)\ge d_X(x,y)$. Then the function $q(f)=\frac12(f+f^*)$ belongs to $\Delta_BX$ and $q(f)\le f$.

For any functions $f,g\in\Delta_BX$ and $x\in B$, we have 
$$g^*(x)=\sup_{z\in B}(d_X(x,z)-f(z)+f(z)-g(z))\le f^*(x)+d_{\Delta X}(f,g)$$and hence $d_{\bar\IR^X}(f^*,g^*)\le d_{\Delta X}(f,g)$ and finally
$$d_{\Delta X}(q(f),q(g))\le \tfrac12d_{\Delta X}(f,g)+\tfrac12d_{\Delta X}(f^*,g^*)\le d_{\Delta X}(f,g).$$

Iterating the map $q$, we obtain for every $f\in\Delta_BX$ a decreasing function sequence $g(f)\ge q^2(f)\ge\dots $ in $\Delta_BX$, then define $p_B(f)$ as the pointwise limit of the function sequence $\big(q^n(f)\big){}_{n\in\IN}$.  Clearly, $p_B(f)\in\Delta_BX$, $p_B(f)\le f$ and the map $p_B:\Delta_BX\to\Delta_BX$ is non-expanding. For all $n\ge 1$, $p_B(f)\le q^n(f)$ and hence $p_B(f)^*\ge q^n(f)^*$, so 
$$0\le p_B(f)-p_B(f)^*\le q^n(f)-q^n(f)^*=2(q^n(f)-q^{n+1}(f)).$$
As $n\to\infty$, the last term converges pointwise to $0$, thus $p_B(f)^*=p(f)$ and therefore, $p_B(f)\in E_B(X)$ by Claim~\ref{cl:EB}.
\end{proof}
 
Finally, define a non-expanding map $p_A:\Delta_AX\to E_BX$ as $p_A=p_B\circ q_A$, where $q_A$, $p_B$ are the non-expanding maps from Claims~\ref{cl:qA} and \ref{cl:pB}, respectively. The properties of these maps ensure that $p_A(f)=p_B(q_A(f))\le q_A(f)\le f$ for every $f\in \Delta_AX$.
\end{proof}

An isometry $i:X\to Y$ between distance spaces is called {\em Isbell-injective} if for any isometry $j:X\to Z$ to a distance space $Z$, there exists a non-expanding map $f:Z\to Y$ such that $i=f\circ j$.

\begin{lemma}\label{l:eD} For any distance space $X$ the isometry  $\mathsf e_X:X\to \Delta X$ is Isbell-injective.
\end{lemma}

\begin{proof} Given any isometry $j:X\to Z$ to a distance space $Z$, consider the map $f:Z\to \Delta X$ assigning to each point $z\in Z$ the function $f_z:X\to\overline\IR$ defined by $f_z(x)=d_{Z}(z,j(x))$ for $x\in X$. Observe that for every $x,y\in X$ we have
$$f_z(x)+f_z(y)=d_{Z}(z,j(x))+d_{Z}(z,j(y))\ge d_{Z}(j(x),j(y))=d_{X}(x,y),$$ and hence $f_z\in\Delta X$.

For any $x\in X$, the function $f_{j(x)}$ coincides with the function $d_x$. Indeed, for any $y\in X$ we have $f_{j(x)}(y)=d_{Z}(j(x),j(y))=d_{X}(x,y)=d_x(y)$.
Consequently, $f\circ j=\mathsf e_X:X\to EX$.

For any points $y,z\in Z$ and any $x\in X$ the triangle inequality for the distance $d_{Z}$ ensures that
$$d_{\bar\IR}(f_y(x),f_z(x))=|d_{Z}(y,j(x))-d_{Z}(z,j(x))|\le d_{Z}(y,z),$$ and hence 
$$d_{\bar\IR^X}(f_y,f_z)=\sup_{x\in X}d_{\bar \IR}(f_y(x),f_z(x))\le d_{Z}(y,z),$$which means that the map $f:Z\to\Delta X$ is non-expanding.
\end{proof}

\begin{lemma}\label{l:eE} For any distance space $X$, the isometry  $\mathsf e_X:X\to E X$ is Isbell-injective.
\end{lemma}

\begin{proof} By Lemma~\ref{l:main}, there exists a non-expanding map $p:\Delta X\to E X$ such that $p(f)=f$ for any $f\in E X$. Consequently, $p\circ \mathsf e_X=\mathsf e_X$.

 To show that the isometry $\mathsf e_X:X\to E X$ is Isbell-injective, take any isometry $j:X\to Z$ to a distance space $Z$.
By Lemma~\ref{l:eD}, there exists a non-expanding map $f:Z\to \Delta X$ such that $f\circ j=\mathsf e_X$. Then $g=p\circ f:Z\to E X$ is a non-expanding map such that $g\circ j=p\circ f\circ j=p\circ \mathsf e_X=\mathsf e_X$.
\end{proof}

\begin{lemma}\label{l:id} Let $X$ be a distance space and $f:E X\to E X$ be a non-expanding map such that $f\circ \mathsf e_X=\mathsf e_X$. Then $f$ is the identity map of $E X$.
\end{lemma}

\begin{proof} Given any function $\alpha\in E X$ consider the function $\beta=f(\alpha)\in E X$. Consider the partition $\{O(x;\infty):x\in X\}$ of $X$ into pseudometric components. Claim~\ref{cl:qA} implies that there are pseudometric components $A,B$ of $X$ such that $\infty\notin \alpha[A]\cup\beta[B]$ and $\alpha[X\setminus A]=\beta[X\setminus B]=\{\infty\}$. We claim that $A=B$. Take any point $a\in A$. By Claim~\ref{cl:EB}, 
$$\infty>\alpha(a)=\sup_{x\in A}(d_a(x)-\alpha(x))=d_{E X}(d_a,\alpha)\ge d_{E X}(f(d_a),f(\alpha))=d_{E X}(d_a,\beta)\ge \beta(a)$$
and hence $a\in B$ and $A\subseteq B$. By analogy, we can prove that $B\subseteq A$ and hence $A=B$. Then $\alpha{\restriction}_{X\setminus A}=\beta{\restriction}_{X\setminus A}$. On the other hand, for every $a\in A$, by Claim~\ref{cl:EB}, 
$$\beta(a)=\sup_{x\in A}(d_a(x)-\beta(x))=d_{E X}(d_a,\beta)=d_{E X}(f(d_a),f(\alpha))\le d_{E X}(d_a,\alpha)\le \alpha(a),$$so $\beta=\alpha$ by the minimality of $\alpha\in E X$.
\end{proof}

An isometry $i:X\to Y$ between distance spaces is called {\em tight} if any non-expanding map $f:Y\to Z$ to a distance space $Z$ is an isometry iff $f\circ i$ is an isometry.

\begin{lemma}\label{l:t} For any distance space $X$, the isometry  $\mathsf e_X:X\to E X$ is tight.
\end{lemma}

\begin{proof} Let $f:E X\to Z$ be a non-expanding map to a distance space $Z$ such that $f\circ\mathsf e_X$ is an isometry. By Lemma~\ref{l:eE}, there exists a non-expanding map $g:Z\to E X$ such that $g\circ f\circ \mathsf e_X=\mathsf e_X$.
By Lemma~\ref{l:id}, $g\circ f$ is the identity map of $E X$. Then for any $\alpha,\beta\in E X$ we have
$$d_{E X}(\alpha,\beta)=d_{E X}\big(g(f(\alpha)),g(f(\beta))\big)\le d_Z(f(\alpha),f(\beta))\le d_{E X}(\alpha,\beta)$$and hence
$d_Z(f(\alpha),f(\beta))=d_{E X}(\alpha,\beta)$, which means that $f$ is an isometry.
\end{proof}

An isometry $i:X\to Y$ between two distance spaces is called an {\em injective envelope} of $X$ if it is tight and Isbell-injective. 

Lemmas~\ref{l:eE} and \ref{l:t} imply the following main result of this section.

\begin{theorem} For any distance space $X$ the map $\mathsf e_X:X\to E X$ is an injective envelope of $X$. 
\end{theorem}

The following proposition says that the injective envelope is unique up to an isometry.

\begin{proposition} For any injective envelope $i:X\to Y$ of a distance space $X$,  there exists a surjective isometry $f:Y\to E X$ such that $f\circ i=\mathsf e_X$. 
\end{proposition}

\begin{proof} Since both isometries $\mathsf e_X:X\to E X$ and $i:X\to Y$ are Isbell-injective, there are nonexpanding maps $f:Y\to E X$ and $g:E X\to Y$ such that 
$f\circ i=\mathsf e_X$ and $g\circ \mathsf e_X=i$. Since the isometries $i$ and $\mathsf e_X$ are tight, the non-expanding maps $f,g$ are isometries. Observe that $f\circ g\circ \mathsf e_X=f\circ i=\mathsf e_X$. By Lemma~\ref{l:id}, $f\circ g$ is the identity map of $E X$. Then the isometry map $f$ is surjective and $g$ is injective.
\end{proof}

\begin{proposition}\label{p:isoext} For any isometry $i:X\to Y$ between distance spaces, there exists an isometry $f:E X\to EY$ such that $f\circ\mathsf e_X=\mathsf e_Y\circ i$.
\end{proposition}

\begin{proof}  By Lemma~\ref{l:eE}, there exists a non-expanding map $f:E X\to EY$ such that  $f\circ \mathsf e_X=\mathsf e_Y\circ i$. By Lemma~\ref{l:eE}, there exists a non-expanding map $g:EY\to E X$ such that $g\circ \mathsf e_Y\circ i=\mathsf e_X$. Since $g\circ f\circ\mathsf e_X=g\circ \mathsf e_Y\circ i=\mathsf e_X$, we can apply Lemma~\ref{l:id} and conclude that $g\circ f$ is the identity map of $E X$. Then the non-expanding map $f$ is an isometry.
\end{proof}

A distance space $X$ is defined to be {\em injective} if for any isometry $i:X\to Y$ to a distance space $Y$ there exists a non-expanding map $f:Y\to X$ such that $f\circ i$ is the identity map of $X$. 

\begin{proposition}\label{p:pi} A distance space $X$ is injective if and only if there exists a non-expanding map $r:E X\to X$ such that $r\circ\mathsf e_X$ is the identity map of $X$.
\end{proposition}

\begin{proof} The ``only if'' part follows from the definition of an injective distance space and Lemma~\ref{l:embE}.

To prove the ``if'' part, assume that $X$ admits a non-expanding map $r:E X\to X$ such that $r\circ \mathsf e_X$ is the identity map of $X$. Given any isometry $i:X\to Y$, apply Lemma~\ref{l:eE} and find a non-expanding map $g:Y\to E X$ such that $g\circ i=\mathsf e_X$. Then the map $f=r\circ g:Y\to X$ is non-expanding and for every $x\in X$ 
$$
f\circ i(x)=r\circ g\circ i(x)=r\circ \mathsf e_X(x)=x.
$$
\end{proof}

\begin{corollary}\label{c:Einj} For any distance space $X$ the $\infty$-metric space $EX$ is injective.
\end{corollary}

\begin{proof} By Lemma~\ref{l:embE}, the maps $\mathsf e_X:X\to EX$ and $\mathsf e_{EX}:EX\to EEX$ are isometries. By Lemma~\ref{l:eE}, there exists a non-expanding map $f:EEX\to EX$ such that $f\circ \mathsf e_{EX}\circ\mathsf e_X=\mathsf e_X$. By Lemma~\ref{l:id}, the composition $f\circ \mathsf e_{EX}$ is the identity map of $EX$.  By Proposition~\ref{p:pi}, the distance space $EX$ is injective.
\end{proof}

\begin{proposition}\label{p:ELip} For any Lipschitz map $f:X\to Y$ between distance spaces, there exists a Lipschitz map $\bar f:EX\to EY$ between their injective envelopes such that $\bar f\circ \mathsf e_X=\mathsf e_Y\circ f$ and $\Lip(\bar f)\le\Lip(f)$.
\end{proposition}

\begin{proof} By Proposition~\ref{p:emb}, there exists an isometry $i:Y\to\bar\IR^Y$. By Proposition~\ref{p:Lip}, for the Lipschitz map $i\circ f:X\to\bar\IR^Y$ and the isometry $\mathsf e_X:X\to EX$, there exists a Lipschitz map $g:EX\to\bar\IR^Y$ such that $g\circ \mathsf e_X=i\circ f$ and $\Lip(g)\le \Lip(i\circ f)=\Lip(f)$. By Lemma~\ref{l:eE}, for the isometry $i:Y\to\bar\IR^Y$, there exists a non-expanding map $h:\bar\IR^Y\to EY$ such that $h\circ i=\mathsf e_Y$. Then $\bar f=h\circ g:EX\to EY$ is a Lipschitz map with Lipschitz constant $\Lip(\bar f)\le \Lip(h)\cdot\Lip(g)\le\Lip(g)\le\Lip(f)$ such that $$\bar f\circ \mathsf e_X=h\circ g\circ\mathsf e_X=h\circ i\circ f=\mathsf e_Y\circ f.$$
\end{proof}

\section{Preliminaries on functors}\label{s:defs}

All functors considered in this paper are covariant. 

We say that a functor $F:\Set\to\Set$ has {\em finite supports} (or else, $F$ is {\em finitely supported\/}) if for any set $X$ and element $a\in FX$ there exists a finite subset $A\subseteq X$ such that $a\in F[A;X]$ where $F[A;X]:=Fi_{A,X}[FA]\subseteq FX$ and $i_{A,X}:A\to X$ is the identity embedding.


For a set $X$ and an element $a\in FX$ its {\em support} is defined as
$$\supp(a)=\bigcap\{A\subseteq X:a\in F[A;X]\}.$$

The following helpful result is proved in \cite{BMZ}.

\begin{proposition}\label{p:BMZ} If $\supp(a)$ is not empty, then $a\in F[\supp(a);X]$. If $\supp(a)$ is empty, then $a\in F[A;X]$ for any non-empty subset $A\subseteq X$.
\end{proposition}

\begin{corollary}\label{c:supp} Let $f:X\to Y$ be a function between sets and $a\in FX$. If $|f[X]|>1$ \textup{(}and $f{\restriction}_{\supp(a)}$ is injective\textup{)}, then $\supp(Ff(a))\subseteq f[\supp(a)]$ \textup{(}and $\supp(Ff(a))= f[\supp(a)]$\textup{)}. 
\end{corollary}

\begin{proof} First assume that $\supp(a)=\emptyset$. Since $|f[X]|>1$, we can find two non-empty sets $A,B\subseteq X$ with $f[A]\cap f[B]=\emptyset$. By Proposition~\ref{p:BMZ}, $a\in F[A;X]\cap F[B;X]$, which implies $Ff(a)\in F[f[A];Y]\cap F[f[B];Y]$. Then $\supp(Ff(a))\subseteq f[A]\cap f[B]=\emptyset=f[\supp(a)]$ and hence $\supp(Ff(a))=f[\supp(a)]$.

Now assume that the set $S:=\supp(a)$ is not empty. By Proposition~\ref{p:BMZ}, $a\in F[S;X]$ and there exists $a'\in FS$ such that $a=Fi_{S,X}(a')$. Then $Ff(a)=Ff\circ Fi_{S,X}(a')=F(f\circ i_{S,X})(a')\in F[f[S];X]$ and hence $\supp(Ff(a))\subseteq f[S]=f[\supp(a)]$.

If $f{\restriction}_S$ is injective, then we can find a function $g:Y\to X$ such that $g(y)\in f^{-1}(y)$ for every $y\in f[X]$ and $g\circ f(x)=x$ for every $x\in S$. Then $g\circ f\circ i_{S,X}=i_{S,X}$ and hence
$$Fg(Ff(a))=Fg\circ Ff\circ Fi_{S,X}(a')=Fi_{S,X}(a')=a.$$ Taking into account that $|g[Y]|=|f[X]|>1$, we conclude that 
$$\supp(a)=\supp(Fg(Ff(a)))\subseteq g[\supp(Ff(a))]\subseteq g[f[\supp(a)]]=\supp(a)$$and hence $g[\supp(Ff(a))]=\supp(a)$. Taking into account that $f\circ g{\restriction}_{f[X]}$ is the identity map of $f[X]$ and $\supp(Ff(a))\subseteq f[\supp(a)]\subseteq f[X]$, we conclude that $$\supp(Ff(a))=f\circ g[\supp(Ff(a))]=f[\supp(a)].$$
\end{proof}

For a set $X$ and a cardinal $n\le\w$, put $F_nX=\{a\in FX:|\supp(a)|\le n\}$. If $FX=F_nX$ for any set $X$, then we say that the functor $F$ has degree $\le n$. The smallest number $n$ with this property is called {\em the degree of the functor $F$} and is denoted by $\deg(F)$. If the cardinal $\deg(F)$ is finite, then $F$ is called a {\em functor of finite degree}. 
 More information on functors (of finite degree) can be found in \cite{TZ}.
 
\begin{proposition}\label{p:BMZ2} If $F:\Set\to\Set$ is a functor with finite support $n=\deg(F)$, then for any set $X$ and element $a\in FX$ with nonempty support, there exist a function $f:n\to X$ and element $b\in Fn$ such that $a=Ff(b)$ and $f[n]=\supp(a)$.
\end{proposition}
 
\begin{proof} By Proposition~\ref{p:BMZ}, $a\in F[S;X]=Fi_{S,X}[FS]$, where $S=\supp(a)$ and $i_{S,X}:S\to X$ is the identity inclusion. Fix any bijective function $h:|S|\to S$ and observe that the function $Fh:F|S|\to FS$ is bijective, too. The definition of the degree $n=\deg(F)$ ensures that $|S|=|\supp(a)|\le n$ and hence $|S|\subseteq n$. Let $f:n\to X$ be any function such that $f[n]=S$ and $f(k)=h(k)$ for every $k\in|S|$. Then $f\circ i_{|S|,n}=i_{S,X}\circ h$ and hence $Ff\circ Fi_{|S|,n}=Fi_{S,X}\circ Fh$. Since $a\in F[S;X]=Fi_{S,X}[FS]=Fi_{S,X}[Fh[F(|S|)]]$, there exists an element $c\in F(|S|)$ such that $a=F{i_{S,X}}\circ Fh(c)$. Let $b=Fi_{|S|,n}(c)\in Fn$ and observe that
$$Ff(b)=Ff\circ Fi_{|S|,n}(c)=Fi_{S,X}\circ Fh(c)=a.$$
\end{proof}

We will say that a functor $F:\Set\to\Set$
\begin{itemize}
\item {\em preserves supports} if for any function $f:X\to Y$ and any element $a\in FX$ we get $\supp(Ff(a))=f[\supp(a)]$;
\item is {\em finitary} if for any finite set $X$ the set $FX$ is finite;
\item {\em preserves singletons} if for the singleton $1=\{0\}$ the set $F(1)$ is a singleton.
\end{itemize}
\smallskip

\section{Introducing the distance $d^p_{FX}$}

Let $F:\Set\to\Set$ be a functor with finite supports.
Fix a number $p\in[1,\infty]$.

For each distance $d_X$ on a set $X$ we define the distance $d^p_{FX}$ on $FX$ as follows. Let $X^{<\w}=\bigcup_{n\in\w}X^n$. The elements of $X^{<\w}$ are functions $f:n\to X$ defined on finite ordinals $n=\{0,\dots,n-1\}$. On each finite power $X^n$ we consider the $\ell^p$-distance $d^p_{X^{\!n}}$ defined by the formula
$$d^p_{X^{\!n}}(f,g)=\begin{cases}
\big(\sum_{i\in n}d_X(f(i),g(i))^p\big)^{\frac1p}&\mbox{if $p<\infty$;}\\
\;\;\max_{i\in n}d_X(f(i),g(i))&\mbox{if $p=\infty$}.
\end{cases}
$$

The distances $d^p_{X^{\!n}}$, $n\in\w$, determine the distance $d^p_{X^{\!<\!\w}}$ on the set $X^{<\w}=\bigcup_{n\in\w}X^n$, defined by the formula
$$d^p_{X^{\!<\!\w}}(f,g)=\begin{cases}
d^p_{X^n}(f,g)&\mbox{if $f,g\in X^n$ for some $n\in\w$},\\
\;\infty,&\mbox{otherwise}.
\end{cases}
$$
\smallskip

For two elements $a,b\in FX$ let
$L_{FX}(a,b)$ be the set of all finite sequences\newline $\big((a_0,f_0,g_0),\dots,(a_l,f_l,g_l)\big)$ of triples $$(a_0,f_0,g_0),\dots,(a_l,f_l,g_l)\in \bigcup_{n\in\w} Fn\times X^n\times X^n$$ such that $a=Ff_0(a_0)$, $b=Fg_l(a_l)$ and $Fg_i(a_i)=Ff_{i+1}(a_{i+1})$ for all $0\le i<l$. Elements of the set $L_{FX}(a,b)$ are called {\em $(a,b)$-linking chains}.

For an $(a,b)$-linking chain $s=\big((a_i,f_i,g_i)\big){}_{i=0}^l\in L_{FX}(a,b)$ put
$$\Sigma d^p_{X}(s):=\sum_{i=0}^ld^p_{X^{\!<\!\w}}(f_i,g_i).$$

Define the distance $d^p_{FX}:FX\times FX\to[0,\infty]$ letting $$d^p_{FX}(a,b)=\inf\big(\{\infty\}\cup\{\Sigma d^p_X(s):s\in L_{FX}(a,b)\}\big).$$



The correspondence $F^p:(X,d_X)\mapsto (FX,d^p_{FX})$ determines a functor $F^p:\DMetr\to\DMetr$ in the category $\DMetr$ of distance spaces, which is a lifting of the functor $F$ with respect to the forgetful functor $\DMetr\to\Set$. This lifting will be called the {\em $\ell^p$-metrization} of the functor $F$. In the remaining part of this paper we study the properties of the functor $F^p$ and its modification $\check F^p$.

From now on, we assume that $F:\Set\to\Set$ is a functor with finite supports, $\mathsf X=(X,d_X)$ is a distance space, and $p\in[1,\infty]$.

\section{The distance $d^p_{FX}$ between elements with small supports}

In this section we evaluate the distance between elements of $FX$ supported by the empty set.

\begin{theorem}\label{t:small}  For any elements $a,b\in FX$ with $|\supp(a)\cup\supp(b)|\le 1$, we get
$$d^p_{FX}(a,b)=\begin{cases}0,&\mbox{if $a=b$},\\
\infty,&\mbox{if $a\ne b$}.
\end{cases}
$$
\end{theorem}

\begin{proof} It is clear that $d^p_{FX}(a,b)=0$ if $a=b$. So, we assume that $a\ne b$. In this case we prove that $d^p_{FX}(a,b)=\infty$. Separately we consider the cases of empty and non-empty set $X$.

First we consider the case of $X=\emptyset$. Assuming that $d^p_{FX}(a,b)<\infty$, we can find an $(a,b)$-linking chain $s=\big((a_i,f_i,g_i)\big){}_{i=0}^l\in L_{FX}(a,b)$ such that $\Sigma d^p_X(s)<\infty$. For every $i\in\{0,\dots,l\}$ find a finite ordinal $n_i$ such that $f_i,g_i\in X^{n_i}$. Then $X=\emptyset\ne X^{n_i}$ implies that $n_i=0$ for all $i\in\{0,\dots,k\}$ and hence $f_i=g_i$ (as the set $X^0$ is a singleton containing the unique function $\emptyset\to \emptyset=X$).
Then $$a=Ff_0(a_0)=Fg_0(a_0)=Ff_1(a_1)=Fg_1(a_1)=\dots=Ff_l(a_l)=Fg_l(a_l)=b,$$which is a desired contradiction.
\smallskip

Now we assume that the set $X$ is not empty. In this case we can find a singleton $S\subseteq X$ containing the union $\supp(a)\cup\supp(b)$.
By Proposition~\ref{p:BMZ}, $a,b\in F[S;X]$. So, we can find (distinct) elements $a',b'\in FS$ such that $a=Fi_{S,X}(a')$ and $b=Fi_{S,X}(b')$. Here by $i_{S,X}:S\to X$ we denote the identity embedding of $S$ into $X$.
Let $f:X\to S$ be the unique constant map and observe that $f\circ i_{S,X}$ is the identity map of $S$. Consequently, $a'=Ff\circ Fi_{S,X}(a')=Ff(a)$ and $b'=Ff(b)$.

Assuming that $d^p_{FX}(a,b)<\infty$, we can find an $(a,b)$-linking chain $s{=}\big((a_i,f_i,g_i)\big){\!\,}_{i=0}^l\in L_{FX}(a,b)$ such that $\Sigma d^p_X(s)<\infty$. Applying to this chain the map $f:X\to S$, we obtain the $(Ff(a),Ff(b))$-linking chain $\big((a_i,f\circ f_i,f\circ g_i)\big){}_{i=0}^l\in L_{FS}(Ff(a),Ff(b))$. Since $S$ is a singleton, $f\circ f_i=f\circ g_i$ for all $i\in\{0,\dots,l\}$ and hence
$$a'=Ff(a)=Ff\circ f_0(a_0)=Ff\circ g_0(a_0)=Ff\circ f_1(a_1)=\cdots=Ff\circ g_l(a_l)=Ff(b)=b'.$$
Then $a=Fi_{S,X}(a')=Fi_{S,X}(b')=b$, which contradicts our assumption.
\end{proof}

\begin{corollary}\label{c:small} If $|X|\le 1$, then for any elements $a,b\in FX$ we get
$$d^p_{FX}(a,b)=\begin{cases}0,&\mbox{if \ $a=b$},\\
\infty,&\mbox{if \ $a\ne b$}.
\end{cases}
$$
\end{corollary}

\section{Short linking chains}

Let $a,b\in FX$ be two elements. An $(a,b)$-linking chain $\big((a_i,f_i,g_i)\big){}_{i=0}^l\in L_{FX}(a,b)$ is called 
\begin{itemize}
\item {\em short} if for any distinct numbers $i,j\in\{0,\dots,l\}$ with $\dom(f_i)=\dom(f_j)$ we have $a_i\ne a_j$;
\item {\em very short} if it is short and $\dom(f_i)=\dom(g_i)=\max\{1,\deg(F)\}$ for every $i\in\{0,\dots,l\}$.
\end{itemize}
Observe that the length of any very short $(a,b)$-linking sequence does not exceed $|Fn|$ where $n=\max\{1,\deg(F)\}$.

By $L^s_{FX}(a,b)$ (resp. $L^{vs}_{FX}(a,b)$) we  denote the subset of $L_{FX}(a,b)$ consisting of (very) short $(a,b)$-linking chains.

\begin{theorem}\label{t:short} For any $a,b\in FX$ the following equality holds:
$$d^p_{FX}(a,b)=\inf\big(\{\infty\}\cup \{\Sigma d^p_X(w):w\in  L^s_{FX}(a,b)\}\big).$$
\end{theorem}

\begin{proof} Let $D:=\inf\big(\{\infty\}\cup\{\Sigma d^p_X(w):w\in L^s_{FX}(a,b)\}\big)$. The inclusion $L^s_{FX}(a,b)\subseteq L_{FX}(a,b)$ implies the inequality $d^p_{FX}(a,b)\le D$. To derive a contradiction, assume that this inequality is strict. Then there exists an $(a,b)$-linking chain $w=\big((a_i,f_i,g_i)\big){}_{i=0}^l\in L_{FX}(a,b)$ such that $\Sigma d^p_X(w)<D$. We can assume that the length $(l+1)$ of this chain is the smallest possible.

 We claim that in this case the chain $w$ is short. Indeed, if for some $i<j$ we have $\dom(f_i)=\dom(f_j)$ and $a_i=a_j$, then the chain $w$ can be replaced  by the shorter $(a,b)$-linking chain 
$$v:=\big((a_0,f_0,g_0),\dots,(a_i,f_i,g_j),(a_{j+1},f_{j+1},g_{j+1}),\dots,(a_l,f_l,g_l)\big),$$which has 
$$
\begin{aligned}
\Sigma d^p_X(v)&=\sum_{\ji=0}^{i-1}d^p_{X^{\!<\!\w}}(f_\ji,g_\ji)+d^p_{X^{\!<\!\w}}(f_i,g_j)+\sum_{\ji=j+1}^ld^p_{X^{\!<\!\w}}(f_\ji,g_\ji)\le\\
&\le\sum_{\ji=0}^{i-1}d^p_{X^{\!<\!\w}}(f_\ji,g_\ji)+\sum_{\ji=i}^jd^p_{X^{\!<\!\w}}(f_{\ji},g_\ji)+\sum_{\ji=j+1}^ld^p_{X^{\!<\!\w}}(f_\ji,g_\ji)=\Sigma d^p_X(w)<D,
\end{aligned}
$$
but this contradicts the choice of $w$ as the shortest chain with $\Sigma d^p_X(w)<D$.
\end{proof}

\begin{theorem}\label{t:very-short} If the functor $F$ has finite degree, then for any $a,b\in FX$
$$d^p_{FX}(a,b)=\inf\big(\{\infty\}\cup \{\Sigma d^p_X(w):w\in  L^{vs}_{FX}(a,b)\}\big).$$
\end{theorem}

\begin{proof} Let $D:=\inf\big(\{\infty\}\cup\{\Sigma d^p_X(w):w\in L^{vs}_{FX}(a,b)\}\big)$. The inclusion $L^{vs}_{FX}(a,b)\subseteq L_{FX}(a,b)$ implies the inequality $d^p_{FX}(a,b)\le D$. To derive a contradiction, assume that this inequality is strict.

Then $a\ne b$ and there exists an $(a,b)$-linking chain $w=\big((a_i,f_i,g_i)\big){}_{i=0}^l\in L_{FX}(a,b)$ such that $\Sigma d^p_X(w)<D$.  We can assume that the length $(l+1)$ of this chain is the smallest possible. The minimality of $l$ implies that $f_i\ne g_i$ and hence $n_i\ne 0$ for all $i\in\{0,\dots,l\}$.

By our assumption, the cardinal $n=\deg(F)$ is finite. Taking into account that $a\ne b$ and $d^p_{FX}(a,b)<\infty$, we can apply Theorem~\ref{t:small} and conclude that $n=\deg(F)>0$. For every $j \in\{0,\dots,l\}$ consider the cardinal $n_j=\dom(f_j)=\dom(g_j)>0$.  By Proposition~\ref{p:BMZ}, there exists a non-empty set $S_j\subseteq n_j$ of cardinality $|S_j|\le n$ such that $a_j\in F[S_j;n_j]$. Then $a_j=Fj_{S_j,n_j}(a_j')$ for some element $a_j'\in FS_j$. Let $h_j:|S_j|\to S_j$ be any bijection. Then $Fh_j$ is a bijection too. Hence, there exists an element $a_j''\in F|S_j|$ such that $a_j'=Fh_j(a''_j)$.

By definition, the cardinal $|S_j|$ coincides with the subset $\{0,\dots,|S_j|-1\}$ of the cardinal $n=\{0,\dots,n-1\}$. For every $j\in\{0,\dots,l\}$, let $b_j=Fi_{|S_j|,n}(a_j'')\in Fn$ and choose functions $f_j',g_j':n\to X$ such that $f_j'\circ i_{|S_j|,n}=f_j\circ i_{S_j,n_j}\circ h_j$, $g_j'\circ i_{|S_j|,n}=g_j\circ i_{S_j,n_j}\circ h_j$, and $f'_j{\restriction}_{n\setminus |S_j|}=g'_j{\restriction}_{n\setminus|S_j|}$. It is easy to see that $d^p_{X^{<\w}}(f'_j, g'_j)=d^p_{X^{<\w}}(f_j,g_j)$.

Observe that $$Ff_j'(b_j)=Ff_j\circ Fi_{|S_j|,n}(a_j'')=Ff_j\circ Fi_{S_j,n_j}\circ Fh_j(a_j'')=Ff_j\circ Fi_{S_j,n_j}(a_j')=Ff_j(a_j)$$and similarly $Fg_j'(b_j)=Fg_j(a_j)$. This implies that $w'=\big((b_j,f'_j,g_j')\big){}_{j=0}^l$ is an $(a,b)$-linking chain with $\sum d^p_X(w')=\sum d^p_X(w)<D$. The minimality of $l$ ensures that this chain is very short, which contradicts the definition of $D$.
\end{proof}

\begin{lemma}\label{l:2short} If the functor $F$ is finitary and has finite degree, then for any $\delta\in(0,\infty)$, set $A\subseteq X$ and  elements $a,b\in FX$ with $d^\infty_{FX}(a,b)<\delta$, there exists a function $r:X\to X$ such that $r\circ r=r$,  $Fr(a)=Fr(b)$, $r[A]\subseteq A$ and $$\sup_{x\in X}d_X(x,r(x))< n\cdot |Fn|\cdot \delta,$$
where $n=\max\{1,\deg(F)\}$.
\end{lemma}

\begin{proof} By Theorem~\ref{t:very-short}, there exist $l\le|Fn|$ and a very short $(a,b)$-linking chain $w=\big((a_i,f_i,g_i)\big)_{i\in l}\in (Fn\times X^n\times X^n)^l$ such that $$\e:=\Sigma d^\infty_{X}(w)<\delta.$$ Let $J=\{(f_i(k),g_i(k)):i\in l,\;k\in n\}$ and $\tilde J$ be the smallest equivalence relation on $X$ containing the set $J$. 
For a point $x\in X$ let $\tilde J(x):=\{y\in X:(x,y)\in \tilde J\}$ be its equivalence class. If $y\in \tilde J(x)\setminus\{x\}$, then we can find a sequence of pairwise distinct points $x_0,\dots,x_m$ in $X$ such that $x_0=x$, $x_m=y$ and $(x_{i-1},x_i)\in J\cup J^{-1}$ for every $i\in \{1,\dots,m\}$. Then $d_{X}(x,y)\le m\cdot \e\le |J|\cdot \e\le nl\e$ and hence $\diam(\tilde J(x))\le nl\e$ for every $x\in X$. 
Choose a set $S\subseteq X$ such that  
\begin{itemize}
\item $S\cap \tilde J(x)$ is a singleton for every $x\in X$;
\item $S\cap\tilde J(x)\subseteq A$ for every $x\in A$.
\end{itemize}
Let $r:X\to S$ be the function assigning to each point $x\in X$ the unique point of the singleton $S\cap \tilde J(x)$. Observe that for every $x\in X$ we have $d_X(x,r(x))\le\diam \tilde J(x)\le nl\e<n\cdot|Fn|\cdot\delta$. The definition of $r$ implies $r\circ r=r$.

We claim that $Fr(a)=Fr(b)$. For this observe that $F(a)=Ff_0(a_0)$ and $F(b)=Fg_l(a_l)$. For every $i\in l$ and $k\in n$ we have $(f_i(k),g_i(k))\in J\subseteq\tilde J$ and hence $r\circ f_i(k)=r\circ g_i(k)$. Then $r\circ f_i=r\circ g_i$ and $Fr\circ Ff_i(a_i)=Fr\circ Fg_i(a_i)$. Finally,
\begin{multline*}
Fr(a)=Fr\circ Ff_0(a_0)=Fr\circ Fg_0(a_0)=Fr\circ Ff_1(a_1)=Fr\circ Fg_1(a_1)=\dots\\
=Fr\circ Ff_l(a_l)=Fr\circ Fg_l(a_l)=Fr(b).
\end{multline*}
\end{proof}

\section{Equivalent definitions of the distance $d_{FX}^p$}

The distance $d_{FX}^p$ can be equivalently defined as follows.

\begin{theorem}\label{t:alt} The distance $d^p_{FX}$ is equal to the largest distance on $FX$ such that for every $n\in\IN$ and $a\in Fn$ the map $\xi^a_X:X^n\to FX$, $\xi^a_X:f\mapsto Ff(a)$, is non-expanding with respect to the $\ell^p$-distance $d^p_{X^{\!n}}$ on $X^n$.
\end{theorem}

\begin{proof} The definition of the distance $d^p_{FX}$ implies that for every $n\in\IN$ and $a\in Fn$ the map $\xi^a_X:X^n\to FX$, $\xi^a_X:f\mapsto Ff(a)$, is non-expanding.

Next, we show that $d^p_{FX}\ge\rho$ for any distance $\rho$ on $FX$ such that for every $n\in\IN$ and $a\in Fn$ the map $\xi^a_X:(X^n,d^p_{X^n})\to (FX,\rho)$ is non-expanding. Assuming that $d^p_{FX}\not\ge\rho$, we can find two elements $a,b\in FX$ such that $d^p_{FX}(a,b)<\rho(a,b)$. The definition of the distance $d^p_{FX}(a,b)$ yields an $(a,b)$-linking chain $\big((a_i,f_i,g_i)\big){}_{i=0}^l\in L_{FX}(a,b)$ such that  $\sum_{i=0}^ld^p_{X^{\!<\!\w}}(f_i,g_i)<\rho(a,b)$.

For every $i\in\{0,\dots,l\}$ find a number $n_i\in\w$ such that $f_i,g_i\in X^{n_i}$. The non-expanding property of the map $\xi^{a_i}_X:X^{n_i}\to FX$ implies that $$\rho\big(Ff_i(a_i),Fg_i(a_i)\big)=\rho\big(\xi^{a_i}_X(f_i),\xi^{a_i}_X(g_i)\big)\le d^p_{X^{\!n_{\!i}}}(f_i,g_i)=d^p_{X^{\!<\!\w}}(f_i,g_i).$$ Applying the triangle inequality, we conclude that $$\rho(a,b)=\rho\big(Ff_0(a_0),Fg_l(a_l)\big)\le \sum_{i=0}^l\rho\big(Ff_i(a_i),Fg_i(a_i)\big)\le \sum_{i=0}^l d^p_{X^{\!<\!\w}}(f_i,g_i)<\rho(a,b),$$
which is a desired contradiction proving the inequality $\rho\le d^p_{FX}$.
\end{proof}

If the functor $F$ has finite degree, then the distance $d^p_{FX}$ can be equivalently defined as follows.

\begin{theorem}\label{t:altm} If the functor $F$ has finite degree $n=\deg(F)$, then the distance $d^p_{FX}$ is equal to the largest distance on $FX$ such that for every $a\in Fn$ the map $\xi^a_X:X^n\to FX$, $\xi^a_X:f\mapsto Ff(a)$, is non-expanding with respect to the $\ell^p$-distance $d^p_{X^{\!n}}$ on $X^n$.
\end{theorem}

\begin{proof} The definition of the distance $d^p_{FX}$ implies that for every $a\in Fn$ the map $\xi^a_X:X^n\to FX$, $\xi^a_X:f\mapsto Ff(a)$, is non-expanding.

Next, we show that $d^p_{FX}\ge\rho$ for any distance $\rho$ on $FX$ such that for every $a\in Fn$ the map $\xi^a_X:(X^n,d^p_{X^n})\to (FX,\rho)$ is non-expanding. Assuming that $d^p_{FX}\not\ge\rho$, we can find two elements $a,b\in FX$ such that $d^p_{FX}(a,b)<\rho(a,b)$. Then $a\ne b$ and $d^p_{FX}(a,b)<\infty$.
Let us show that $n>0$. Assuming that $n=0$, we conclude that $\supp(a)=\supp(b)=\emptyset$. By Theorem~\ref{t:small}, $d^p_{FX}(a,b)=\infty$, which contradicts $d^p_{FX}(a,b)<\rho(a,b)$.
So, $n\in\IN$ and $n=\max\{1,\deg(F)\}$.

By Theorem~\ref{t:very-short}, there exists a very short $(a,b)$-linking chain $$\big((a_i,f_i,g_i)\big){}_{i=0}^l\in L_{FX}(a,b)\cap(Fn\times X^n\times X^n)^{<\w}$$ such that $\sum_{i=0}^ld^p_{X^{\!n}}(f_i,g_i)<\rho(a,b)$. 
For every $i\le l$, the non-expanding property of the map $\xi^{a_i}_X:X^{n}\to FX$ implies that $\rho\big(Ff_i(a_i),Fg_i(a_i)\big)=\rho\big(\xi^{a_i}_X(f_i),\xi^{a_i}_X(g_i)\big)\le d^p_{X^{\!n}}(f_i,g_i)$. Applying the triangle inequality, we conclude that $$\rho(a,b)=\rho\big(Ff_0(a_0),Fg_l(a_l)\big)\le \sum_{i=0}^l\rho\big(Ff_i(a_i),Fg_i(a_i)\big)\le \sum_{i=0}^l d^p_{X^n}(f_i,g_i)<\rho(a,b),$$
which is a desired contradiction proving the inequality $\rho\le d^p_{FX}$.
\end{proof}

\begin{problem}\label{prob:algo} Is there an efficient algorithm for calculating the distance $d^p_{FX}$ (with arbitrary precision)?
\end{problem}

\section{The relation between the distances $d^p_{FX}$ and $d^q_{FX}$}

In this section we establish the relation between the distances  $d^p_{FX}$ for various numbers $p\in[1,\infty]$. For this we will use the following known relation between the distances $d^p_{X^n}$ and $d^q_{X^n}$ on the set $X^n$.

\begin{lemma}\label{l:ineq} For any $n\in\w$ and any numbers $p,q\in[1,\infty]$ with $p\le q$ the following inequalities hold:
$$d^q_{X^n}\le d^p_{X^n}\le n^{\frac{q-p}{qp}}\cdot d^q_{X^n}.$$
\end{lemma}

\begin{proof} If $n=0$, then $X^0$ is a singleton and 
$$d^q_{X^n}=d^p_{X^n}=n^{\frac{q-p}{qp}}\cdot d^q_{X^n}=0.$$
So, we assume that $n>0$. Choose any distinct points $x,y\in X^n$ and consider the function $t:n\to [0,\infty]$, $t:i\mapsto d_X(x(i),y(i))$. If $t(i)=\infty$ for some $i\in n$, then  $$d^q_{X^n}(x,y)=\infty=d^p_{X^n}(x,y)=n^{\frac{q-p}{qp}}\cdot d^q_{X^n}(x,y)$$and we are done. So, assume that $t(i)<\infty$ for all $i\in n$.

In this case $d^p_{X^n}(x,y)=\|t\|_p$ where $\|t\|_p=\big(\sum_{i\in n}t(i)^p\big)^{\frac1p}$. If $p\le q$, then the well-known inequality \cite[Ex.1.8]{FA} $\|t\|_q\le \|t\|_p$ implies that $d^q_{X^n}(x,y)\le d^p_{X^n}(x,y)$.

To prove that $d^p_{X^n}(x,y)\le n^{\frac{q-p}{qp}}d^q_{X^n}(x,y)$, we will apply the H\"older inequality \cite[1.5]{FA} to the functions $t^p:n\to\IR$, $t^p:i\mapsto t(i)^p$ and $\mathbf 1:n\to\{1\}\subset\IR$:
$$\sum_{i\in n}t(i)^{p}\cdot 1\le \|t^p\|_{\frac{q}p}\cdot\|\mathbf 1\|_{\frac1{1-\frac{p}q}}=\big(\sum_{i\in n}t^{p\frac{q}p}\big)^{\frac{p}q}\cdot n^{1-\frac{p}{q}}=(\|t\|_q)^p\cdot n^{\frac{q-p}{q}}.$$
Taking this inequality to the power $\frac1p$, we obtain the required inequality
$$d^p_{X^n}(x,y)=\|t\|_p\le \|t\|_q\cdot n^{\frac{q-p}{pq}}=d^q_{X^n}(x,y)\cdot n^{\frac{q-p}{qp}}.$$
\end{proof}

Lemma~\ref{l:ineq} and Theorems~\ref{t:altm}, \ref{t:alt} imply

\begin{theorem}\label{t:ineq} For any numbers $p,q\in[1,\infty]$ the inequality $p\le q$ implies $$d^\infty_{FX}\le d^q_{FX}\le d^p_{FX}\le d^1_{FX}\mbox{ and }d^p_{FX}\le n^{\frac{q-p}{qp}}\cdot d^q_{FX}$$where $n=\deg(F)$.
\end{theorem}

\begin{proof} Assume that $1\le p\le q\le\infty$. By Theorem~\ref{t:alt}, for every $k\in\IN$ and every $a\in Fk$ the function $\xi^a_X:(X^k,d^q_{X^k})\to (FX,d^q_{FX})$ is non-expanding. By Lemma~\ref{l:ineq}, the identity function $(X^k,d^p_{X^k})\to (X^k,d^q_{X^k})$ is non-expanding. Then the map $\xi^a_X:(X^k,d^p_{X^k})\to (FX,d^q_{FX})$ is non-expanding (as a composition of two non-expanding maps). Applying Theorem~\ref{t:alt}, we conclude that $d^q_{FX}\le d^p_{FX}$.
\smallskip

Now assume that the functor $F$ has finite degree $n=\deg(F)$. By Lemma~\ref{l:ineq}, $d^p_{X^n}\le n^{\frac{q-p}{qp}}d^q_{X^n}$. Then for the distance $\rho_X:=n^{\frac{q-p}{qp}}\cdot d_X$ on $X$, the identity map $(X^n,\rho^q_{X^n})\to (X^n, d^p_{X^n})$ is non-expanding. Looking at the definitions of the distances $d^q_{FX}$ and $\rho^q_{FX}$, we can see that $\rho^q_{FX}=n^{\frac{q-p}{qp}}\cdot d^q_{FX}$.

By Theorem~\ref{t:alt}, for every $a\in Fn$ the map $\xi^a_X:(X^n,d^p_{X^n})\to (FX,d^p_{FX})$ is non-expanding and so is the map $\xi^a_X:(X^n,\rho^q_{X^n})\to (FX,d^p_{FX})$. Applying Theorem~\ref{t:altm}, we conclude that $d^p_{FX}\le \rho^q_{FX}=n^{\frac{q-p}{qp}}\cdot d^q_{FX}$.
\end{proof}

\section{Preserving Lipschitz functions by the functor $F^p$}

In this section we show that the $\ell^p$-metrization $F^p$ of the functor $F$ preserves Lipschitz functions between distance spaces and does not increase their Lipschitz constant.


\begin{theorem}\label{t:Lip} For any Lipschitz function $f:(X,d_X)\to (Y,d_Y)$ between distance spaces, the function $Ff:(FX,d^p_{FX})\to (FY,d^p_{FY})$ is Lipschitz with Lipschitz constant $\Lip(Ff)\le \Lip(f)$.
\end{theorem}

\begin{proof} We need to check that $d^p_{FY}(Ff(a),Ff(b))\le \Lip(f)\cdot d^p_{FX}(a,b)$ for any points $a,b\in FX$. This inequality is trivial if $d^p_{FX}(a,b)=\infty$. So, we assume that $d^p_{FX}(a,b)<\infty$. In this case it suffices to prove that $d^p_{FY}(Ff(a),Ff(b))\le  \Lip(f)\cdot (d^p_{FX}(a,b)+\e)$ for every $\e>0$. By the definition of the distance $d^p_{FX}(a,b)$, there exists an $(a,b)$-linking chain $\big((a_i,f_i,g_i)\big){}_{i=0}^l\in L_{FX}(a,b)$ such that $\sum_{i=0}^ld^p_{X^{\!<\!\w}}(f_i,g_i)<d^p_{FX}(a,b)+\e$.
Applying to this chain the function $f:X\to Y$, we obtain an $(Ff(a),Ff(b))$-linking chain $\big(a_i,f{\circ} f_i,f{\circ} g_i)\big){}_{i=0}^l$ such that
$$
\begin{aligned}
&d^p_{FY}(Ff(a),Ff(b))\le\sum_{i=0}^ld^p_{Y^{\!<\!\w}}(f{\circ} f_i,f{\circ} g_i)=\sum_{i=0}^l\Big(\sum_{x\in n_i}d_Y(f{\circ} f_i(x),f{\circ} g_i(x))^p\Big)^{\frac1p}\le\\
&\le \sum_{i=0}^l\Big(\sum_{x\in n_i}\big(\Lip(f)\cdot d_X(f_i(x),g_i(x))\big)^p\Big)^{\frac1p}=\Lip(f)\cdot
\sum_{i=0}^l\Big(\sum_{x\in n_i}d_X(f_i(x),g_i(x))^p\Big)^{\frac1p}=\\
&=\Lip(f)\cdot \sum_{i=0}^ld^p_{X^{\!<\!\w}}(f_i,g_i)\le\Lip(f)\cdot(d^p_{FX}(a,b)+\e),
\end{aligned}
$$
where $n_i:=\dom(f_i)=\dom(g_i)$ for $i\in\{0,\dots,l\}$.
\end{proof}

\section{Preserving isometries by the functor $F^p$}

\begin{lemma}\label{l:isom1} If $f:X\to Y$ is a surjective isometry of distance spaces, then the map $F^p\!f:F^p\!X\to F^pY$ is an isometry.
\end{lemma}

\begin{proof} By the surjectivity of $f$, there exists a function $g:Y\to X$ such that $f\circ g$ is the identity map of $Y$. Then $Ff\circ Fg$ is the identity map of $FY$, which implies that $Ff$ is surjective.

Taking into account that $f$ is an isometry, we conclude that  for any $y,y'\in Y$,
$$d_X(g(y),g(y'))=d_Y(f\circ g(y),f\circ g(y'))=d_Y(y,y'),$$which means that $g$ is an isometry. Consider the map $r=g\circ f:X\to X$ and observe that for every $x\in X$ we have
$$d_X(x,r(x))=d_X(x,g\circ f(x))=d_Y(f(x),f\circ g\circ f(x))=d_Y(f(x),f(x))=0.$$

We claim that for any $a\in FX$ the distance $d^p_{FX}(a,Fr(a))=0$. Since the functor $F$ has finite supports, there exist $n_0\in\w$, map $f_0:n_0\to X$ and $a_0\in Fn_0$ such that $a=Ff_0(a_0)$. Put $g_0=r\circ f_0$ and observe that $\big((a_0,f_0,g_0)\big)$ is an $(a,Fr(a))$-linking chain witnessing that
$$d^p_{FX}(a,Fr(a))\le d^p_{X^{\!<\!\w}}(f_0,g_0)=\Big(\sum_{i\in n_0}d_X(f_0(i),r\circ f_0(i))^p\Big)^{\frac1p}=0.$$

Now take any elements $a,b\in FX$ and applying Theorem~\ref{t:Lip}, conclude that
\begin{multline*}
d^p_{FX}(a,b)\le d^p_{FX}(a,Fr(a))+d^p_{FX}(Fr(a),Fr(b))+d^p_{FX}(Fr(b),b)=\\
=0+d^p_{FX}(Fg{\circ} Ff(a),Fg{\circ} Ff(b))+0\le d^p_{FY}(Ff(a),Ff(b))\le d^p_{FX}(a,b),
\end{multline*}
which implies that all inequalities in this chain are equailites. In particular, $$d^p_{FY}(Ff(a),Ff(b))=d^p_{FX}(a,b),$$ which means that the map $Ff$ is an isometry.
\end{proof}

We will say that a subset $A\subseteq X$ of a distance space $(X,d_X)$ is {\em dense} in $X$ if for any $\e>0$ and $x\in X$ there exists $a\in A$ such that $d_X(a,x)<\e$.

\begin{lemma}\label{l:isom2} If $f:X\to Y$ is an injective isometry of distance spaces and the set $f[X]$ is dense in $Y$, then the map $F^pf:F^pX\to F^pY$ is an injective isometry and the set $F^pf[F^pX]=Ff[FX]$ is dense in $F^pY$.
\end{lemma}

\begin{proof} If $X=\emptyset$, then $Y=\emptyset$ by the density of the set $f[X]=\emptyset$ in $Y$. In this case the map $f:X\to Y$ is the identity map of $\emptyset$ and $F^pf$ is the identity map (and hence a  bijective isometry) of the distance space $F^pX=F^pY=F^p\emptyset$.

So, we assume that the set $X$ is not empty. Then $Y$ is not empty, too. By Theorem~\ref{t:Lip}, the map $F^pf:F^pX\to F^pY$ is non-expanding. Assuming that $F^pf$ is not an isometry, we can find two elements $a,b\in F^pX$ such that $d^p_{FY}(Ff(a),Ff(b))<d^p_{FX}(a,b)$. By the definition of the distance $d^p_{FY}$, there exists an $(Ff(a),Ff(b))$-linking chain $\big((a_i,f_i,g_i)\big){}_{i=0}^l\in L_{FY}(Ff(a),Ff(b))$ such that $\sum_{i=0}^ld^p_{Y^{\!<\!\w}}(f_i,g_i)<d^p_{FX}(a,b)$. For every $i\in\{0,\dots,l\}$ let $n_i:=\dom(f_i)=\dom(g_i)$.
Choose $\e>0$ so small that $$\sum_{i=0}^l\big(2\e\cdot\sqrt[p]{n_i}+d^p_{Y^{\!<\!\w}}(f_i,g_i)\big)<
d^p_{FX}(a,b).$$ Using the injectivity of the function $f$ and the density of the set $f[X]$ in $Y$, choose a function $r:Y\to X$ such that  $r\circ f(x)=x$ for all $x\in X$, and $d_Y(y,f{\circ} r(y))<\e$ for all $y\in Y\setminus f[X]$.
Applying the functor $F$ to the identity map $r\circ f$ of $X$, we conclude that $Fr\circ Ff$ is the identity map of $FX$, which implies that the map $Ff:FX\to FY$ is injective.

Consider the maps $f_i'=r{\circ} f_i:n_i\to X$ and $g_i'=r{\circ} g_i:n_i\to X$. Taking into account that the map $f$ is an isometry, we conclude that
$$
\begin{aligned}
d^p_{X^{\!<\!\w}}&(f_i',g_i')=d^p_{Y^{\!<\!\w}}(f{\circ} f_i',f{\circ} g_i')=d^p_{Y^{\!<\!\w}}(f{\circ} r{\circ} f_i,f{\circ} r{\circ} g_i)\le\\
&\le
d^p_{Y^{\!<\!\w}}(f{\circ} r{\circ} f_i,f_i)+
d^p_{Y^{\!<\!\w}}(f_i,g_i)+
d^p_{Y^{\!<\!\w}}(g_i,f{\circ} r{\circ} g_i)<\e\sqrt[p]{n_i}+
d^p_{Y^{\!<\!\w}}(f_i,g_i)+\e\sqrt[p]{n_i}.
\end{aligned}
$$

We claim that $Ff_0'(a_0)=a$ and $Fg_l'(a_l)=b$.
It follows that $\supp(Ff(a))\subseteq f[X]$ so we can find a finite set $S\subseteq f[X]$ of cardinality $|S|=\max\{1,|\supp(Ff(a))|\}$ containing the support $\supp(Ff(a))$ of $Ff(a)$. By Proposition~\ref{p:BMZ}, $Ff(a)\subseteq F[S;Y]$. So, we can find an element $a'\in FS$ such that $Ff(a)=Fi_{S,Y}(a')$. The definition of $r$ and the inclusion $S\subseteq f[X]$ imply $f\circ r\circ i_{S,Y}=i_{S,Y}$. Then
\begin{multline*}
Ff(a)=Fi_{S,Y}(a')=Ff\circ Fr\circ Fi_{S,Y}(a')=Ff\circ Fr\circ Ff(a)=\\=Ff\circ Fr\circ Ff_0(a_0)=Ff\circ Ff_0'(a_0)
\end{multline*}
and hence
$a=Ff_0'(a_0)$ by the injectivity of the map $Ff$.

By analogy, we can prove that $Fg_l'(a_l)=b$.
Consequently, $\big((a_0,f_0',g_0'),\dots,(a_l,f_l',g_l')\big)$ is an $(a,b)$-linking chain and
$$d^p_{FX}(a,b)\le\sum_{i=0}^ld^p_{X^{\!<\!\w}}(f_i',g_i')<
\sum_{i=0}^l\big(\e\sqrt[p]{n_i}+d^p_{X^{\!<\!\w}}(f_i,g_i)+\e\sqrt[p]{n_i}\big)<d^p_{FX}(a,b),$$
which is a desired contradiction proving that $d^p_{FY}(Ff(a),Ff(b))= d^p_{FX}(a,b)$.
\smallskip

Finally, we prove that the image $Ff[FX]$ is dense in $F^pY$. Given any $b\in FY$ and $\e>0$, choose a finite set $S\subseteq Y$ of cardinality $|S|=\max\{1,|\supp(b)|\}$ containing $\supp(b)$.
 By Proposition~\ref{p:BMZ}, $b\in F[S;Y]$.

 Let $n_0=|S|$, $\beta:n_0\to S$ be any bijection, and $f_0=i_{S,Y}\circ \beta:n_0\to Y$. Then $F\beta:Fn_0\to FS$ is a bijection and hence $b\in F[S;Y]=Fi_{S,Y}[FS]=Fi_{S,Y}[F\beta[Fn_0]]=Ff_0[Fn_0]$. Consequently, there exists an element $b'\in Fn_0$ such that $b=Ff_0(b')$.

Since the set $f[X]$ is dense in $Y$, there exists a function $g_0:n_0\to f[X]\subseteq Y$ such that $d^p_{Y^{\!<\!\w}}(f_0,g_0)<\e$. It follows that the map $\alpha=f^{-1}\circ g_0:n_i\to X$ is well-defined and $f\circ\alpha =g_0$. Let $a:=F\alpha(b')\in FX$ and observe that $Ff(a)=Ff\circ F\alpha(b')=Fg_0(b')$ and
$$d^p_{FY}(b,Ff(a))=d^p_{FX}(Ff_0(b'),Fg_0(b'))\le d^p_{Y^{\!<\!\w}}(f_0,g_0)<\e,$$which means that the set $Ff[FX]$ is dense in $F^pY$.
\end{proof}

Lemmas~\ref{l:isom1} and \ref{l:isom2} imply

\begin{theorem}\label{t:isom} If $f:X\to Y$ is an isometry of distance spaces and the set $f[X]$ is dense in $Y$, then the map $F^pf:F^pX\to F^pY$ is an  isometry and the set $F^pf[F^pX]=Ff[FX]$ is dense in $F^pY$.
\end{theorem}

\begin{proof} The isometry $f$ can be written as the composition $f=h\circ g$ of a surjective isometry $g:X\to Z$ and an injective isometry $h:Z\to Y$. For example, we can take $Z=f[X]$, $g=f$ and $h=i_{Z,Y}$. By Lemma~\ref{l:isom1}, the map $Fg:F^pX\to F^pZ$ is a surjective isometry. By the density of the set  $f[X]=h[g[X]]=h[Z]$ in $Y$ and Lemma~\ref{l:isom2}, the map $Fh:F^pZ\to F^pY$ is an isometry with dense image $Fh[FZ]$ in $F^pY$. Then the function $Ff=Fh\circ Fg$ is an isometry as the composition of two isometries. Moreover, the image  $Ff[FX]=Fh[Fg[FX]]=Fh[FZ]$ is dense in $F^pY$ by the density of $Fh[FZ]$ in $F^pY$.
\end{proof}

\begin{example}\label{ex1} In general, the functor $F^p$ does not preserve isometries. To construct a counterexample, consider the functor $F:\Set\to\Set$ assigning to each set $X$ the family $FX$ of all subsets $A\subseteq X$ of cardinality $0<|A|\le 2$. To every map $f:X\to Y$ between sets the functor $F$ assigns the map $Ff:FX\to FY$, $Ff:a\mapsto f[a]$. The functor $F$ has finite degree $2$, preserves supports and singletons (see Section~\ref{s:hyperf} for more information on this functor).

Now consider the following finite subsets $X\subseteq Y$ of the complex plane $\mathbb C$:
$$X:=\{5+\sqrt{3}i,5-\sqrt{3}i,-5+\sqrt{3}i,-5-\sqrt{3}i\}\quad\mbox{and}\quad Y=X\cup\{-4,4\}.$$In the set $FX\subseteq FY$ consider the elements $a=\{-5+\sqrt{3}i,-5-\sqrt{3}i\}$ and $b=\{5+\sqrt{3}i,5-\sqrt{3}i\}$.

It can be shown that for every $p\in[1,\infty]$
$$d^p_{FX}(a,b)=10+4\sqrt{3}>8+4\sqrt[p]{2}=d^p_{FY}(a,b).$$
\end{example}

Yet, we have the following partial positive  result.

\begin{proposition}\label{p:weakiso} If the functor $F$ is finitary and has finite degree, then for any isometry $f:X\to Y$ between nonempty distance spaces and any elements $a,b\in F^pX$, we have
$$d^\infty_{FX}(a,b)\le 2n\cdot|Fn|\cdot d^\infty_{FY}(Ff(a),Ff(b)),$$
where $n=\max\{1,\deg(F)\}$.
\end{proposition}

\begin{proof} 
To derive a contradiction, assume that $d^\infty_{FX}(a,b)>2n\cdot|Fn|\cdot d^p_{FY}(Ff(a),Ff(b))$. Using Lemma~\ref{l:2short}, find a function $r:Y\to Y$ such that $Fr(Ff(a))=Fr(Ff(b))$, $r[f[X]]\subseteq f[X]$ and $$\sup_{y\in Y}d_Y(y,r(y))<n\cdot|Fn|\cdot\frac1{2n\cdot|Fn|}\cdot d^\infty_{FX}(a,b)=\frac12\, d^\infty_{FX}(a,b).$$ Let $s:f[X]\to X$ be any function such that $s(y)\in f^{-1}(y)$ and hence $f\circ s(y)=y$ for every $y\in f[X]$. 
Using Proposition~\ref{p:BMZ}, find two functions $f_0,g_1\in X^n$ and elements $a_0,a_1\in Fn$ such that $a=Ff_0(a_0)$ and $b=Fg_1(a_1)$. Now consider the chain
$$w=\big((a_0,f_0,g_0),(a_1,f_1,g_1)\big)\in (Fn\times X^n\times X^n)^2,$$
where $g_0=s\circ r\circ f\circ f_0$ and $f_1=s\circ r\circ f\circ g_1$. 
Since 
\begin{multline*}Fg_0(a_0)=Fs\circ Fr\circ Ff\circ Ff_0(a_0)=Fs\circ Fr\circ Ff(a)=\\
Fs\circ Fr\circ Ff(b)=Fs\circ Fr\circ Ff\circ Fg_1(a_1)=Ff_1(a_1),
\end{multline*}
the chain $w$ is $(a,b)$-linking.

Taking into account that $f$ is an isometry, $f\circ s(y)=y$ for $y\in f[X]$, and $r[f[X]]\subseteq f[X]$ we conclude that for every $k\in n$,
\begin{multline*}
d_X(f_0(k),g_0(k))=d_X(f_0(k),s\circ r\circ f\circ f_0(k))=\\=d_Y(f\circ f_0(k),f\circ s\circ r\circ f\circ f_0(k))=
 d_X(f\circ f_0(k),r\circ f\circ f_0(k))<\tfrac12d^\infty_{FX}(a,b)
\end{multline*}
and hence $d^\infty_{X^n}(f_0,g_0)<\frac12d^\infty_{FX}(a,b)$.
By analogy we can prove that $d^\infty_{X^n}(f_1,g_1)<\frac12d^\infty_{FX}(a,b)$.
Then
$$d^\infty_{FX}(a,b)\le d^\infty_X(w)=d^\infty_X(f_0,g_0)+d^\infty_X(f_1,g_1)<\tfrac12d^\infty_{FX}(a,b)+\tfrac12d^\infty_{FX}(a,b)=d^\infty_{FX}(a,b),$$
which is a desired contradiction.
\end{proof}

\section{The functor $F^p\!f$ and the continuity modulus of functions}

Let us recall that for a function $f:X\to Y$ between distance spaces, its continuity modulus $\w_f:(0,\infty]\to[0,\infty]$ is defined by the formula 
$$\w_f(\e)=\sup\{d_Y(f(x),f(y)):x,y\in X\;\wedge\;d_X(x,y)<\e\}\mbox{ \ for  \ } \e\in(0,\infty].$$
 
\begin{theorem}\label{t:modulus} If the functor $F$ has finite degree $n=\deg(F)$, then for any function $f:X\to Y$ between distance spaces, the function $F^p\!f:F^pX\to F^pY$ has continuity modulus $\w_{F^p\!f}\le |Fn|\cdot n^{\frac1p}\cdot \w_f$.
\end{theorem} 

\begin{proof} Assuming that  $\w_{F^p\!f}\not\le |Fn|\cdot n^{\frac1p}\cdot \w_f$, we can find $\e\in(0,\infty]$ and elements $a,b\in F^pX$ such that $d^p_{FX}(a,b)<\e$ but $d^p_{FY}(Ff(a),Ff(b))>|Fn|\cdot n^{\frac1p}\cdot \w_f(\e)$. Then $a\ne b$ and we can apply Theorem~\ref{t:small} to conclude that $|\supp(a)\cup\supp(b)|>1$ and hence  $n=\deg(F)>0$. By Theorem~\ref{t:very-short}, there exists a very short $(a,b)$-linking chain $w=\big((a_i,f_i,g_i)\big){}_{i\in l}\in (Fn\times X^n\times X^n)^{l}$ such that $\sum d^p_X(w)=\sum_{i\in l}d^p_{X^n}(f_i,g_i)<\e$. Since the chain $w$ is short, the elements $a_i\in Fn$ are pairwise distinct. Then $l\le |Fn|$. Observe that $\big((a_i,f\circ f_i,f\circ g_i)\big){}_{i\in l}$ is an $(Ff(a),Ff(b))$-linking chain and hence
\begin{multline*}
|Fn|\cdot n^{\frac1p}\cdot\w_f(\e)<d^p_{FY}(Ff(a),Ff(b))\le \sum_{i=0}^l d^p_{Y^n}(f\circ f_i,f\circ g_i)=\\
\sum_{i\in l}\Big(\sum_{x\in n}d_Y\big(f(f_i(x)),f(g_i(x))\big)^p\Big)^{\frac1p}\le
 \sum_{i\in l}\Big(\sum_{x\in n}\w_f\big(d_X(f_i(x),g_i(x))\big)^p\Big)^{\frac1p}\le\\ 
 \sum_{i\in l}\Big(\sum_{x\in n}\w_f(\e)^p\Big)^{\frac1p}=
 l\cdot n^{\frac1p}\cdot\w_f(\e)\le |Fn|\cdot n^{\frac1p}\cdot\w_f(\e),
 \end{multline*}
 which is a desired contradiction completing the proof of the inequality $\w_{F^p\!f}\le |Fn|\cdot n^{\frac1p}\cdot \w_f$.
 \end{proof}
 
 \begin{corollary}\label{c:uniform}  Assume that the functor $F$ is finitary and has finite degree. 
If a function $f:X\to Y$ between distance spaces is asymptotically Lipschitz (resp. microform, macroform, duoform), then so is the function $F^p\!f:F^pX\to F^p Y$.
\end{corollary}

\begin{lemma}\label{l:dist} For any nonempty set $Z$ and functions $f,g:Z\to X$ we have
$$\sup_{a\in FZ}d^p_{FX}(Ff(a),Fg(a))\le n^{\frac1p}\cdot d_{X^Z}(f,g)$$
where $n=\max\{1,\deg(F)\}$.
\end{lemma}

\begin{proof} To derive a contradiction, assume that $d^p_{FX}(Ff(a),Fg(a))>n^{\frac1p}\cdot d_{X^Z}(f,g)$ for some $a\in FZ$. 

Given any $a\in FZ$, use Proposition~\ref{p:BMZ} and find a nonempty finite subset $S\subseteq Z$ of cardinality $|S|\le n$ and an element $a'\in FS$ such that $a=Fi_{S,Z}(a')$.
Let $k=|S|$ and $h:k\to S$ be any bijection. Then $Fh$ is a bijection, too, and we can find an element $a_0\in Fk$ such that $a'=Fh(a_0)$. Consider the maps $f_0=f\circ i_{S,Z}\circ h:k\to Z$ and $g_0=g\circ i_{S,Z}\circ h:k\to X$ and observe that $$Ff_0(a_0)=Ff\circ Fi_{S,Z}\circ Fh(a_0)=Ff\circ Fi_{S,Z}(a')=Ff(a)$$ and similarly, $Fg_0(a_0)=Fg(a)$. Then $\big((a_0,f_0,g_0)\big)$ is an $(Ff(a),Fg(a))$-linking chain and hence $$d^p_{FX}(Ff(a),Fg(a))\le \Big(\sum_{i\in k}d_X(f_0(i),g_0(i))^p\Big)^{\frac1p}\le k^{\frac1p}\cdot d_{X^Z}(f,g)\le n^{\frac1p}\cdot d_{X^Z}(f,g),
$$ which contradicts the choice of $a$. 
\end{proof}  

\begin{corollary}\label{c:coarse-equiv}  Assume that the functor $F$ is finitary and has finite degree. 
If a function $f:X\to Y$ between distance spaces is a coarse equivalence (resp. quasi-isometry), then so is the function $F^p\!f:F^pX\to F^p Y$.
\end{corollary}

\begin{proof} If $f$ is a coarse equivalence, then $f$ is macroform and admits a macroform map $g:Y\to X$ such that $$d_{X^X}(1_X,g\circ f)=\sup_{x\in X}d_X(x,g\circ f(x))<\infty$$ and $$d_{Y^Y}(1_Y,f\circ g)=\sup_{y\in Y}d_Y(y,f\circ g(y))<\infty.$$ By Corollary~\ref{c:uniform}, the maps $Ff:F^p X\to F^pY$ and $Fg:F^pY\to F^pX$ are macroform and by Lemma~\ref{l:dist}, $$\sup_{a\in F^pX}d^p_{FX}(a,Fg\circ Ff(a))\le \max\{1,\deg(F)\}^{\frac1p}\cdot d_{X^X}(1_X,g\circ f)<\infty$$ and $$\sup_{b\in F^pY}d^p_{FY}(b,Ff\circ Fg(b))\le \max\{1,\deg(F)\}^{\frac1p}\cdot d_{Y^Y}(1_Y,f\circ g)<\infty.$$ This means that the map $F^p\!f$ is a coarse equivalence.
\smallskip

By analogy we can prove that the map $F^p f$ is a quasi-isometry if so is the map $f$.
\end{proof}

\section{Upper bounds for $d^p_{FX}$}\label{s:up}

In this section we prove some upper bounds for the distance $d^p_{FX}$. In particular, we establish conditions implying that $d^p_{FX}$ is a pseudometric.
First we show that the functor $F^p$ preserves the $\{0,\infty\}$-valuedness of distances.

A distance $d$ on a set $Y$ is called {\em $\{0,\infty\}$-valued} if $d[Y\times Y]=\{0,\infty\}$.

\begin{lemma}\label{l:ivalued} If the distance $d_X$ is $\{0,\infty\}$-valued (and $\infty$-metric), then the distance $d^p_{FX}$ on $FX$ also is $\{0,\infty\}$-valued (and $\infty$-metric).
\end{lemma}

\begin{proof} To show that $d^p_{FX}$ is $\{0,\infty\}$-valued (and $\infty$-metric), it suffices to check that $d^p_{FX}(a,b)=0$ (and $a=b$) for any elements $a,b\in FX$ with $d^p_{FX}(a,b)<\infty$.
Using the definition of the distance $d^p_{FX}$, find an $(a,b)$-linking chain 
$s=\big((a_i,f_i,g_i)\big){}_{i=0}^l\in L_{FX}(a,b)$ such that $\sum_{i=0}^l d^p_{X^{\!<\!\w}}(f_i,g_i)<\infty$. If $d_X$ is $\{0,\infty\}$-valued, then so is the  distance $d^p_{X^{\!<\!\w}}$. Then for every $i\le l$ the inequality $d^p_{X^{\!<\!\w}}(f_i,g_i)<\infty$ is equivalent to $d^p_{X^{\!<\!\w}}(f_i,g_i)=0$, and finally
$$d^p_{FX}(a,b)\le \sum_{i=0}^ld^p_{X^{\!<\!\w}}(f_i,g_i)=0.$$

Now assuming that $d_X$ is an $\infty$-metric, we can show that the distance $d^p_{X^{\!<\!\w}}$ is an $\infty$-metric and hence the equalities $d^p_{X^{\!<\!\w}}(f_i,g_i)=0$ imply $f_i=g_i$ for all $i\in\{0,\dots,l\}$. Then 
$$a=Ff_0(a_0)=Fg_0(a_0)=Ff_1(a_1)=Fg_1(a_1)=\dots=Fg_l(a_l)=b.$$
\end{proof} 

The distance $d_X$ on $X$ induces the equivalence relation $\sim$ on $X$, defined by $x\sim y$ iff $d_X(x,y)<\infty$. For an element $x\in X$ let $O(x;\infty):=\{y\in X:d_X(x,y)<\infty\}$ be the equivalence class of $x$ under the equivalence relation $\sim$.  Let $X/_\sim=\{O(x;\infty):x\in X\}$ be the quotient set of $X$ by the relation $\sim$ and $q:X\to X/_\sim$ be the quotient map.

For two finite sets $A,B\subseteq X$ let $\ud_X(A,B)=\inf\{d_X(x,y):x\in A,\;y\in B\}$. If 
one of the sets $A$ or $B$ is empty, then we put $\ud_X(A,B)=0$ (by definition).
Let $$\ud^p_X(A,B)=\Big(\sum_{E\in X/_{\!\sim}}\ud_X(A\cap E,B\cap E)^p\Big)^{\frac1p}$$
and observe that  $\ud^p_X(A,B)$ 
is finite (in contrast to $\ud_X(A,B)$ which can be infinite).

For a subset $S\subseteq X$ by $$\bar d_X(S)=\sup\big(\{0\}\cup\{d_X(x,y):x,y\in S,\;d_X(x,y)<\infty\}\big)$$we denote the real diameter of $S$.

\begin{theorem}\label{t:up} Let $q:X\to X/_{\!\sim}$ be the quotient map defined above. 
For any two elements $a,b\in FX$ the following conditions are equivalent: 
\begin{enumerate}
\item[\textup{1)}] $d^p_{FX}(a,b)<\infty$,
\item[\textup{2)}] $Fq(a)=Fq(b)$,
\item[\textup{3)}] $d^p_{FX}(a,b)\le |S_a|^{\frac1p}\cdot\bar d_X(S_a)+\ud^p_X(S_a,S_b)+|S_b|^{\frac1p}\cdot \bar d_X(S_b)$,
\item[\textup{4)}] $d^p_{FX}(a,b)\le \min\big\{|S_a|^{\frac1p}\bar d_X(S_a)+|S_b|^{\frac1p}\bar d_X(S),\;|S_a|^{\frac1p}\bar d_X(S)+|S_b|^{\frac1p}\bar d_X(S_b)\big\}$,
\item[\textup{5)}] $d^p_{FX}(a,b)\le (|S_a|^{\frac1p}+|S_b|^{\frac1p})\cdot \bar d_X(S)$,
\end{enumerate}  where $S_a=\supp(a)$, $S_b=\supp(b)$ and $S=\supp(a)\cup\supp(b)$.
\end{theorem}

\begin{proof}  Endow the set $Y=X/_{\!\sim}$ with the unique $\{0,\infty\}$-valued $\infty$-metric $d_Y$, and observe that the quotient map $q:X\to Y$ is Lipschitz with Lipschitz constant $\Lip(q)=0$ (here we assume that $0\cdot\infty=\infty$). By Theorem~\ref{t:Lip}, $\Lip(Fq)=0$. By Lemma~\ref{l:ivalued}, the distance $d^p_{FY}$ is a $\{0,\infty\}$-valued $\infty$-metric. Now we are ready to prove that $(1)\Ra(2)\Ra(3)\Ra(1)$ and $(2)\Ra(4)\Ra(5)\Ra(1)$. The implications $(3)\Ra(1)\Leftarrow(5)\Leftarrow(4)$ are trivial.
\smallskip

$(1)\Ra(2)$ If $d^p_{FX}(a,b)<\infty$, then $d^p_{FY}(Fq(a),Fq(b))\le\Lip(Fq)\cdot d^p_{FX}(a,b)=0$ and  $Fq(a)=Fq(b)$ (because the distance $d^p_{FY}$ is an $\infty$-metric).
\smallskip

$(2)\Ra(3,4)$. Assume that $Fq(a)=Fq(b)$ and put $c=Fq(a)=Fq(b)$. Let $S_a:=\supp(a)$ and $S_b:=\supp(b)$. For every equivalence class $y\in Y=X/_\sim$ with $S_a\cap y\ne\emptyset \ne S_b\cap y$ choose two points $\alpha_{y,a}\in S_a\cap y$ and $\beta_{y,b}\in S_b\cap y$ such that $d_X(\alpha_{y,a},\beta_{y,b})=\ud_X(S_a\cap y,S_b\cap y)$. Fix two maps $\alpha,\beta:Y\to X$ such that for every $y\in Y$ the following conditions hold:
\begin{itemize}
\item $\alpha(y)=\alpha_{y,a}$ and $\beta(y)=\beta_{y,a}$ if $S_a\cap y\ne\emptyset\ne S_b\cap y$;
\item $\alpha(y)=\beta(y)\in S_a\cap y$ if $S_a\cap y\ne\emptyset=S_b\cap y$;
\item $\alpha(y)=\beta(y)\in S_b\cap y$ if $S_a\cap y=\emptyset \ne S_b\cap y$;
\item $\alpha(y)=\beta(y)\in y$ if $S_a\cap y=\emptyset=S_b\cap y$.
\end{itemize}   
It can be shown that 
$$\Big(\sum_{y\in Y}d_X(\alpha(y),\beta(y))^p\Big)^{\frac1p}=\ud^p_X(S_a,S_b).$$

Choose a set $\dot S_c\subseteq Y$ of cardinality $|\dot S_c|=\max\{1,|\supp(c)|\}$ such that $\supp(c)\subseteq \dot S_c$. Next, choose a set $\dot S_a\subseteq X$ of cardinality $|\dot S_a|=\max\{1,|S_a|\}$ such that $S_a\subseteq \dot S_a$, and $\dot S_a\subseteq \alpha[Y]$ if $S_a=\emptyset$. Also choose a set $\dot S_b\subseteq X$ of cardinality $|\dot S_b|=\max\{1,|S_b|\}$ such that $S_b\subseteq \dot S_b$, and $\dot S_b\subseteq \beta[Y]$ if $S_b=\emptyset$.

Let $n_0=|\dot S_a|$, $n_1=|\dot S_c|$, $n_2=|\dot S_b|$, and $f_0:n_0\to X$, $\gamma:n_1\to Y$, and $g_2:n_2\to X$ be injective maps such that $f_0[n_0]=\dot S_1$, $\gamma[n_1]=\dot S_c$ and $g_n[n_2]=\dot S_b$. By Proposition~\ref{p:BMZ}, $a\in F[\dot S_a;X]$, $b\in F[\dot S_b;X]$ and $c\in F[\dot S_c;Y]$. This allows us to find elements $a_0\in Fn_0$, $a_1\in Fn_1$, $a_2\in Fn_2$ such that $a=Ff_0(a_0)$, $c=F\gamma(a_1)$, $b=Fg_2(a_2)$. Let $g_0=\alpha\circ q\circ f_0$, $f_1=\alpha\circ\gamma$, $g_1=\beta\circ\gamma$, $f_2=\beta\circ q\circ g_2$. To see that $\big((a_0,f_0,g_0),(a_1,f_1,g_1),(a_2,f_2,g_2)\big)$ is an $(a,b)$-linking chain, we need to check that $Fg_0(a_0)=Ff_1(a_1)$ and $Fg_1(a_1)=Ff_2(a_2)$.

For the first equality, observe that
$$Fg_0(a_0)=F\alpha\circ Fq\circ Ff_0(a_0)=F\alpha\circ Fq(a)=F\alpha(c)=F\alpha(F\gamma(a_1))=Ff_1(a_1).$$
By analogy we can check that $Fg_1(a_1)=Ff_2(a_2)$.
Then $\big((a_i,f_i,g_i)\big){}_{i=0}^2\in L_{FX}(a,b)$ and
$$
\begin{aligned}
d^p_{FX}(a,b)&\le d^p_{X^{\!<\!\w}}(f_0,g_0)+d^p_{X^{\!<\!\w}}(f_1,g_1)+
d^p_{X^{\!<\!\w}}(f_2,g_2)=\\
&=d^p_{X^{\!<\!\w}}(f_0,\alpha\circ q\circ f_0)+d^p_{X^{\!<\!\w}}(\alpha\circ\gamma,\beta\circ\gamma)+d^p_{X^{\!<\!\w}}(\beta\circ q\circ g_2,g_2)\le\\
&\le |S_a|^{\frac1p}\cdot\bar d_X(S_a)+\ud^p_X(S_a,S_b)+|S_b|^{\frac1p}\cdot\bar d_X(S_b).
\end{aligned}
$$

Considering the $(a,b)$-linking chain $((a_0,f_0,g_0),(a_2,f_2',g_2)\big)$ with $f_2'=\alpha\circ q\circ g_2$, we can see that  
$$
\begin{aligned}
d^p_{FX}(a,b)&\le d^p_{X^{\!<\!\w}}(f_0,g_0)+
d^p_{X^{\!<\!\w}}(f_2',g_2)=\\
&=d^p_{X^{\!<\!\w}}(f_0,\alpha\circ q\circ f_0)+d^p_{X^{\!<\!\w}}(\alpha\circ q\circ g_2,g_2)\le |S_a|^{\frac1p}\cdot\bar d_X(S_a)+|S_b|^{\frac1p}\cdot\bar d_X(S),
\end{aligned}
$$
where $S=S_a\cup S_b$.

Considering the $(a,b)$-linking chain $((a_0,f_1,g'_1),(a_2,f_2,g_2)\big)$ with $g_0'=\beta\circ q\circ g_2$, we can see that  
$$
\begin{aligned}
d^p_{FX}(a,b)&\le d^p_{X^{\!<\!\w}}(f_0,g'_0)+
d^p_{X^{\!<\!\w}}(f_2,g_2)=\\
&=d^p_{X^{\!<\!\w}}(f_0,\beta\circ q\circ f_0)+d^p_{X^{\!<\!\w}}(\beta\circ q\circ g_2,g_2)\le |S_a|^{\frac1p}\cdot\bar d_X(S)+|S_b|^{\frac1p}\cdot\bar d_X(S_b).
\end{aligned}
$$
\end{proof}

\begin{corollary}\label{c:up}  Assume that the distance $d_X$ is a pseudometric and let $q:X\to 1$ be the constant map. Then for any $a,b\in FX$ the following statements hold.
\begin{enumerate}
\item[\textup{1)}] If $Fq(a)\ne Fq(b)$, then $d^p_{FX}(a,b)=\infty$.
\item[\textup{2)}] If $Fq(a)=Fq(b)$, then $d^p_{FX}(a,b)\le (|S_a|^{\frac1p}+|S_b|^{\frac1p})\cdot\bar d_X(S_a\cup S_b)$, where $S_a=\supp(a)$, $S_b=\supp(b)$.
\end{enumerate}
\end{corollary}

\begin{corollary}\label{c:up2} If $F$ is a functor of finite degree $n=\deg(F)$, then the distance space $F^pX$ has real diameter
$$\bar d^p_{FX}(FX)\le 2{\cdot}|n|^{\frac1p}{\cdot}\bar d_X(X).$$
\end{corollary}




\section{Lower bounds on the distance $d^p_{FX}$}\label{s:low}

In this section we find some lower bounds on the distance $d^p_{FX}$. We recall that for a subset $A\subseteq X$ its separatedness number $\ud_X(A)$ is defined as 
$$\ud_X(A):=\inf\big(\{d_X(x,y):x,y\in A,\;x\ne y\}\cup\{\infty\}\big).$$

\begin{proposition}\label{p:e-sep} The distance space $(FX,d^p_{FX})$ has the separatedness number\newline $\ud^p_{FX}(FX)\ge \ud_X(X)$.
\end{proposition}

\begin{proof} Assuming that $\ud^p_{FX}(FX)<\ud_X(X)$, we can find two distinct elements $a,b\in FX$ with $d^p_{FX}(a,b)<\ud_X(X)$ and then find an $(a,b)$-linking chain $\big((a_i,f_i,g_i)\big){}_{i=0}^l\in L_{FX}(a,b)$ such that $\sum_{i=0}^ld^p_{X^{\!<\!\w}}(f_i,g_i)<\ud_X(X)$. The definition of the separatedness number $\ud_X(X)$ ensures that $f_i=g_i$ for all $i\in\{0,\dots,l\}$. Then
$$a=Ff_0(a_0)=Fg_0(a_0)=Ff_1(a_1)=Fg_1(a_1)=\dots=Fg_l(a_l)=b,$$which contradicts the choice of the distinct elements $a,b$.
\end{proof}

We  say that a distance space $(Y,d_Y)$ is {\em separated} if it has non-zero separatedness number $\ud_Y(Y)$. In this case its distance $d_Y$ is called {\em separated}. Proposition~\ref{p:e-sep} implies:

\begin{corollary}\label{c:sep} If the distance space $(X,d_X)$ is separated, then so is the distance space $(FX,d^p_{FX})$.
\end{corollary}

We will say that a distance space $(X,d_X)$ is {\em Lipschitz disconnected} if for any distinct points $x,y\in X$ there exists a Lipschitz map $f:X\to \{0,1\}\subset\IR$ such that $f(x)=0$ and $f(y)=1$. The Lipschitz property of $f$ implies that $1=d_\IR(f(x),f(y))\le \Lip(f)\cdot d_X(x,y)$ and hence $d_X(x,y)>0$, witnessing that {\em each Lipschitz disconnected distance is an $\infty$-metric}. On the other hand, {\em each separated distance space is Lipschitz disconnected}.

\begin{theorem}\label{t:Ldis} If the distance space $(X,d_X)$ is Lipschitz disconnected, then the distance $d^p_{FX}$ is an $\infty$-metric.
\end{theorem}

\begin{proof} Given two distinct points $a,b\in FX$, we should prove that $d^p_{FX}(a,b)>0$. This inequality follows from Theorem~\ref{t:small} if the set $S=\supp(a)\cup\supp(b)$ contains at most one point. So, we can assume that $|S|>1$. In this case $a,b\in F[S;X]$ by Proposition~\ref{p:BMZ} and we can find elements $a',b'\in FS$ such that $a=Fi_{S,X}(a')$ and $b=Fi_{S,X}(b')$. Since $a\ne b$, the elements $a',b'$ are distinct. Since the set $S$ is finite, the $\infty$-metric $d_S=d_X{\restriction}_{S\times S}$ is separated. By Corollary~\ref{c:sep}, the distance $d^p_{FS}$ is separated and hence is an $\infty$-metric. Consequently, $d^p_{FS}(a',b')>0$.

Using the Lipschitz disconnectedness of $X$, we can construct a Lipschitz map $f:X\to \{0,1,\infty\}^m\subset\bar\IR^m$ for a suitable $m$ such that $f{\restriction}_S$ is injective and for any $x,y\in S$ with $d_X(x,y)=\infty$ we have $d_{\bar \IR^m}(f(x),f(y))=\infty$. Take any function $g:\{0,1,\infty\}^m\to S$ such that $g\circ f(x)=x$ for every $x\in S$, and observe that the function $r=g\circ f:X\to S$ is Lipschitz and $r\circ i_{S,X}$ is the identity map of $S$.
Then $a'=Fr\circ Fi_{S,X}(a')=Fr(a)$ and $b'=Fr(b)$.
By Theorem~\ref{t:Lip}, the function $F^pr:F^pX\to F^pS$ is Lipschitz. Taking into account that $$0<d^p_{FS}(a',b')=d^p_{FS}(Fr(a),Ff(b))\le\Lip(F^pr)\cdot d^p_{FX}(a,b),$$ we conclude that $d^p_{FX}(a,b)>0$, which means that the distance $d^p_{FX}$ is an $\infty$-metric.
\end{proof}

A conditions implying that $d^1_{FX}$ is an $\infty$-metric is given in the following theorem. 

\begin{theorem}\label{t:d1} For any distinct elements $a,b\in FX$ we have  $$d^1_{FX}(a,b)\ge \ud_X(\supp(a)\cup\supp(b)).$$ This implies that $d^1_{FX}$ is an $\infty$-metric if $d_X$ is an $\infty$-metric.
\end{theorem}

\begin{proof} We need to prove that $d^1_{FX}(a,b)\ge\ud_X(S)$ where  $S:=\supp(a)\cup\supp(b)$. To derive a contradiction, assume that $d^1_{FX}(a,b) < \ud_X(S)$ and find an $(a,b)$-linking chain $\big((a_i,f_i,g_i)\big){}_{i=0}^l\in L_{FX}(a,b)$ such that $\sum_{i=0}^l d^1_{X^{\!<\!\w}}(f_i,g_i)< \ud_X(S)$.

Let $\sim$ be the smallest equivalence relation on $X$ containing all pairs $(f_i(z),g_i(z))$, where $i\le l$ and  $z\in \dom(f_i)=\dom(g_i)$.
We claim that $x\not\sim y$ for any distinct points $x,y\in S$. Indeed, assuming that $x\sim y$, we could find a sequence $x=x_0,\dots,x_l=y$ of pairwise distinct points of $X$ such that for every positive $i\le l$ there exist a number $j_i\in\{0,\dots,l\}$ and a point $z_i\in\dom(f_{j_i})=\dom(g_{j_i})$ such that $\{x_{i-1},x_{i}\}=\{f_{j_i}(z_i),g_{j_i}(z_i)\}$.

We claim that for any $j\in\{0,\dots,l\}$ and $z\in\dom(f_j)$ the set $I_{j,z}=\big\{i\in\{0,\dots,l\}:(j_i,z_i)=(j,z)\big\}$ contains at most one number. Otherwise, we could find two distinct numbers $i,\ji\in I_{j,z}$ and conclude that $$\{x_{i-1},x_i\}=\{f_{j_i}(z_i),g_{j_i}(z_i)\}=\{f_j(z),g_j(z)\}=\{f_{j_{\ji}}(z_{\ji}),g_{j_{\ji}}(z_{ji})\}=\{x_{\ji-1},x_{\ji}\},$$which is not possible as the points $x_0,\dots,x_l$ are pairwise distinct.

Then
$$
\begin{aligned}
d_X(x,y)\le\;&\sum_{i=1}^ld_X(x_{i-1},x_i)=
\sum_{i=1}^ld_X(f_{j_{i}}(z_{i}),g_{j_i}(z_i))=\\
&\sum_{j=0}^l\sum_{z\in\dom(f_j)}\sum_{i\in I_{j,z}}d_X(f_{j_i}(z_i),g_{j_i}(z_i))=\\
&\sum_{j=0}^l\sum_{z\in\dom(f_j)}|I_{j,z}|\cdot d_X(f_{j}(z),g_{j}(z))\le\\
&\sum_{j=0}^l\sum_{z\in\dom(f_j)}d_X(f_{j}(z),g_{j}(z))=\sum_{j=0}^ld^1_{X^{\!<\!\w}}(f_{j},g_{j})<\ud_X(S),
\end{aligned}
$$which contradicts the definition of $\ud_X(S)$.
This contradiction completes the proof of non-equivalence $x\not\sim y$ for any distinct points $x,y\in S$.

Now choose any function $r:X\to S\subseteq X$ such that
\begin{itemize}
\item $r(x)=r(y)$ for any points $x,y\in X$ with $x\sim y$ and
\item $r(x)=x$ for any point $x\in S$.
\end{itemize}
The choice of the function $r$ is always possible since distinct points of the set $S$ are not equivalent.

The definition of the equivalence relation $\sim$ guarantees that  $r\circ f_i = r\circ g_i$ for any $i\le l$. Consequently, $$F(r{\circ} f_0)(a_0) = F(r{\circ} g_0)(a_0) = F(r{\circ} f_1)(a_1) = \dots = F(r{\circ} g_l)(a_l).$$
We claim that $F(r{\circ} f_0)(a_0)=Ff_0(a_0)=a$. Proposition~\ref{p:BMZ} guarantees that $a\in F[S;X]$, so we can find an element $a'\in FS$ such that $a=Fi_{S,X}(a')$ where $i_{S,X}:S\to X$ denotes the identity inclusion. The equality $r\circ i_{S,X}=i_{S,X}$ implies that $Fr\circ Fi_{S,X}=Fi_{S,X}$ and hence $$a=Fi_{S,X}(a')=Fr{\circ} Fi_{S,X}(a')=Fr(a)=Fr(Ff_0(a_0))=F(r{\circ} f_0)(a_0).$$
By analogy we can prove that $F(r{\circ} g_l)(a_l)=Fr(b)=b$.
Then $$a=F(r{\circ} f_0)(a_0)=F(r{\circ} g_l)(b_l)=b,$$
which contradicts the choice of the (distinct) points $a,b$.
\end{proof}

Theorems~\ref{t:d1} and \ref{t:ineq} imply the following corollary.

\begin{corollary}\label{c:imetric} If $d_X$ is an $\infty$-metric and the functor $F$ has finite degree $n=\deg(F)$, then for every $p\in[1,\infty]$ the distance $d^p_{FX}$ is an $\infty$-metric. Moreover, for any distinct elements $a,b\in FX$ either $d^p_{FX}(a,b)=\infty$ or
$$d^p_{FX}(a,b)\ge n^{\frac{1-p}{p}}\cdot d^1_{FX}(a,b)\ge n^{\frac{1-p}p}\cdot\ud_X(\supp(a)\cup\supp(b))>0.$$
\end{corollary}

Another condition guaranteeing that the distance $d^p_{FX}$ is an $\infty$-metric was found by Shukel in \cite{Shukel}. Theorem~\ref{t:shukel} below is a modification of Theorem 1 of \cite{Shukel}. 

\begin{theorem}\label{t:shukel} If the functor $F$ preserves supports, then for any distinct elements $a,b\in FX$ we have 
$$d^p_{FX}(a,b)\ge d^\infty_{FX}(a,b)\ge \max\big\{d_{HX}(\supp(a),\supp(b)),\tfrac13\ud_X(\supp(a)\cup\supp(b))\}.$$
\end{theorem}

\begin{proof}  The inequality $d^p_{FX}(a,b)\ge d^\infty_{FX}(a,b)$ follows from Theorem~\ref{t:ineq}. It remains to prove that $d^\infty_{FX}(a,b)\ge\e$ where
$$\e:=\max\{d_{HX}(S_a,S_b),\tfrac13\ud_X(S)\},$$ $S_a:=\supp(a)$, $S_b=\supp(b)$, and $S=S_a\cup S_b$.

To derive a contradiction, assume that $d^\infty_{FX}(a,b)<\e$ and find an $(a,b)$-linking chain $\big((a_i,f_i,g_i)\big){}_{i=0}^l\in L_{FX}(a,b)$ such that $\sum_{i=0}^ld^\infty_{X^{\!<\!\w}}(f_i,g_i)<\e$. We can assume that the length $l+1$ of the chain is the smallest possible. 

For every $i\in\{0,\dots,l\}$, consider the element $b_i=Ff_i(a_i)\in FX$ and put $b_{l+1}=Fg_l(a_l)$. It follows that $b_0=a$, $b_{l+1}=b$, and $b_i=Ff_i(a_i)=Fg_{i-1}(a_{i-1})$ for any $i\in\{1,\dots,l\}$. The minimality of $l$ ensures that $b_i\ne b_{i+1}$ for all $i\in\{0,\dots,l\}$.

Taking into account that the functor $F$ preserves supports, we conclude that $\supp(b_i) {=}\break f_i[\supp(a_i)]$ for every $i\in\{0,\dots,l\}$. We claim that for every $i\in\{0,\dots,l\}$ the support $S_i=\supp(a_i)$ is not empty. Assuming that $S_i=\emptyset$, we conclude that $\supp(b_i)=f_i(\supp(a_i))=\emptyset=g_i(\supp(a_i))=\supp(b_{i+1})$. Since $b_i\ne b_{i+1}$, Theorem~\ref{t:small} ensures that $d_{FX}^\infty(b_i,b_{i+1})=\infty$. On the other hand, the definition of the distance $d_{FX}^\infty$ yields $d_{FX}^\infty(b_i,b_{i+1})=d_{FX}^{\infty}(Ff_i(a_i),Fg_i(a_i))\le d^\infty_{X^{\!<\!\w}}(f_i,g_i)<\e\le\infty$. This contradiction shows that $S_i\ne\emptyset$ and hence $\supp(b_i)=f_i[\supp(a_i)]$ is not empty, too.

The chain $a=b_0,\dots,b_{l+1}=b$ of elements of $FX$ induces the chain of finite non-empty subsets $$S_a=\supp(b_0),\dots,\supp(b_{l+1})=S_b$$ in $X$. Then by the triangle inequality for the Hausdorff distance $d_{HX}$, we obtain 
$$
\begin{aligned}
d_{HX}(S_a,S_b)&\le\sum_{i=0}^{l}d_{HX}(\supp(b_{i}),\supp(b_{i+1}))
=\sum_{i=0}^{l}d_{HX}\big(f_i[\supp(a_{i})],g_i[\supp(a_i)]\big)\le\\
&\le \sum_{i=0}^ld^\infty_{X^{\!<\!\w}}(f_i,g_i)<\e=\max\{d_{HX}(S_a,S_b),\tfrac13\ud_X(S)\},
\end{aligned}
$$which implies that $d_{HX}(S_a,S_b)<\e=\frac13\ud_X(S)$ and  $S_a=S_b=S$ by the definition of $d_{HX}(S_a,S_b)$ and $\ud_X(S)$. 

Let $O[S;\e)=\bigcup_{x\in S}O(x;\e)$ be the open $\e$-neighborhood of the finite non-empty set $S$ in the distance space $(X,d_X)$.

Let $r:X\to S\subseteq X$ be any function such that $r[O(x;\e)]=\{x\}$ for any $x\in S$. We claim that $r\circ f_i{\restriction}_{S_i}=r\circ g_i{\restriction}_{S_i}$ for any $i\in\{0,\dots,l\}$. Indeed, 
$$
d_{HX}(S,\supp(b_i))\le
\sum_{j=0}^{i-1}d_{HX}(\supp(b_{j}),\supp(b_{j+1}))
\le \sum_{j=0}^{i-1}d^\infty_{X^{\!<\!\w}}(f_{j},g_{j})<\e
$$and hence $f_i[\supp(a_i)]=\supp(b_i)\subseteq O[S;\e)$ and $g_i[\supp(a_i)]=\supp(Fg_i(a_i))=\break \supp(b_{i+1})\subseteq O[S;\e)$.
As we already know, for every $i\in\{0,\dots,l\}$ the support $S_i=\supp(a_i)$ is a non-empty set of the finite ordinal $n_i:=\dom(f_i)=\dom(g_i)\in\IN$.

The inclusions $f_i[\supp(a_i)]\cup g_i[\supp(a_i)]\subseteq O[S;\e)$ and the inequality $d^\infty_{X^{<\w}}(f_i,g_i)<\e=\frac13\ud_X(S)$ imply that for any $x\in S_i$ there exists a point $s\in S$ with $f_i(x),g_i(x)\in O[s;\e)$. Consequently, $r\circ f_i(x)=s=r\circ g_i(x)$.

 Proposition~\ref{p:BMZ} guarantees that $a_i\in F[\supp(a_i);X]$ and hence $a_i=Fi_{S_i,n_i}(a_i')$ for some $a_i'\in FS_i$, where $i_{S_i,n_i}:S_i\to n_i$ denotes the identity embedding of $S_i$ into $n_i$. It follows from $r\circ f_i{\restriction}_{S_i}=r\circ g_i{\restriction}_{S_i}$ that $r\circ f_i\circ i_{S_i,n_i}=r\circ g_i\circ i_{S_i,n_i}$ and hence
\begin{multline}\label{Frf=g}
Fr\circ Ff_i(a_i)=Fr\circ Ff_i\circ Fi_{S_i,n_i}(a_i')=\\
=Fr\circ Fg_i\circ Fi_{S_i,n_i}(a_i')=Fr\circ Fg_i(a_i)=Fr\circ Ff_{i+1}(a_{i+1}).
\end{multline}

Since $\supp(a)\subseteq S$, by Proposition~\ref{p:BMZ}, there exists an element $a'\in FS$ such that $a=Fi_{S,X}(a')$. Applying the functor $F$ to the equality $r\circ i_{S,X}=i_{S,X}$, we conclude that $Fr(a)=Fr\circ Fi_{S,X}(a')=Fi_{S,X}(a')=a$. By analogy we can prove that $Fr(b)=b$. Taking into account the equality (\ref{Frf=g}), we finally conclude that
\begin{multline*}
a=Fr(a)=Fr\circ Ff_0(a_0)=Fr\circ Fg_0(a_0)=Fr\circ Ff_1(a_1)=Fr\circ Fg_1(a_1)=\dots\\
\dots=Fr\circ Fg_k(a_k)=Fr(b)=b,
\end{multline*}
which contradicts the choice of the (distinct) points $a,b$.
\end{proof}





\section{Preservation of continuous functions by the functor $F^p$}


\begin{theorem}\label{t:continuous} If the functor $F$ is finitary, has finite degree and preseves supports, then for any continuous function $f:X\to Y$ between distance spaces, the function $F^p\!f:F^pX\to F^pY$ is continuous.
\end{theorem}

\begin{proof} Let $n=\max\{1,\deg(F)\}$. To check the continuity of $F^p\!f$, fix any any real number $\e\in(0,\infty)$ and any element $a\in FX$. 

By the continuity of $f$, there exists $\delta\in(0,\infty)$ such that $$f[O(x;\delta)]\subseteq O(f(x);\e/\!\sqrt[p]{n})$$ for all $x\in\supp(a)$. 

Choose a subset $S\subseteq X$ such that for every $x\in X$ (with $x\in \supp(a)$) the intersection $S\cap O[x;0]$ is a singleton (in the set $\supp(a)$). We recall that $O[x;0]=\{y\in X:d_X(x,y)\le 0\}$. Let $s:X\to X$ be a function assigning to each $x\in X$ a unique element of the set $S\cap O[x;0]$. Let $a'=Fs(a)$. Since $F$ preserves supports, $$\supp(a')=s[\supp(a)]\subseteq S\cap\supp(a).$$  
The choice of the set $S\supseteq\supp(a')$ ensures that $\ud_X(\supp(a'))>0$ and hence we can choose a positive real number $\delta'$ such that  $$\delta'<\frac1{n\cdot|Fn|}\min\{\delta,\ud_X(\supp(a'))\}.$$
Given any element $b\in FX$ with $d^p_{FX}(a,b)<\delta'$, we shall prove that $d^p_{FY}(Ff(a),Ff(b))<\e$. 
This inequality is trivially true if $a=b$. So we assume that $a\ne b$. By Theorem~\ref{t:small}, $$d^p_{FX}(a,b)<\delta'<\infty$$ implies that $|\supp(a)\cup \supp(b)|>1$. 

Let $b'=Fs(b)$. Since $F$ preserves supports, $\supp(b')=s[\supp(b)]$.
By Lemma~\ref{l:2short}, there exists a function $r:X\to X$ such that $Fr(b')=Fr(a')$, $r\circ r=r$, $r[\supp(a')]\subseteq \supp(a')$, and $\sup_{x\in X}d_X(r(x),x)<n\cdot|Fn|\cdot\delta'<\ud_X(\supp(a'))$. The latter strict inequality and the inclusion $r[\supp(a')]\subseteq \supp(a')$ imply that $r(x)=x$ for every $x\in \supp(a')$. 

  Since $F$ preserves supports, $r[\supp(b')]=\supp(Fr(b'))=\supp(Fr(a'))=r[\supp(a')]$ and hence the sets $\supp(a')=s[\supp(a)]$, $\supp(b')=s[\supp(b)]$, $\supp(a)$, $\supp(b)$ are not empty (because $|\supp(a)\cup \supp(b)|>1)$. 
By Proposition~\ref{p:BMZ2}, there exist functions $f_0,g_1\in X^n$ and elements $a_0,a_1\in Fn$ such that $a=Ff_0(a_0)$, $f_0[n]=\supp(a)$, and $b=Fg_1(a_1)$,   $g_1[n]=\supp(b)$. Let $g_0=s\circ f_0$ and $f_1=r\circ s\circ g_1$. It follows from $r(x)=x$ for all $x\in\supp(a')=s[\supp(a)]=s[f_0[n]]$ that $r\circ s\circ f_0=s\circ f_0$ and hence $$Fr(b')=Fr(a')=Fr\circ Fs(a)=Fr\circ Fs\circ Ff_0(a_0)=Fs\circ Ff_0(a_0)=a'.$$ Since  
\begin{multline*}
Fg_0(a_0)=Fs\circ Ff_0(a_0)=Fs(a)=a'=Fr(b')=\\
=Fr\circ Fs(b)=Fr\circ Fs\circ Fg_1(a_1)=Ff_1(a_1),
\end{multline*}
the sequence $\big((a_0,f_0,g_0),(a_1,f_1,g_1)\big)$ is an $(a,b)$-linking chain and hence\\
$\big((a_0,f\circ f_0,f\circ g_0),(a_1,f\circ g_1,f\circ g_1)\big)$ is an $(Ff(a),Ff(b))$-linking chain. 

For every $x\in \supp(b')$, we have $$r(x)\in r[\supp(b')]=\supp(Fr(b'))=\supp(a')=\supp(Fs(a))=s[\supp(a)]\subseteq\supp(a)$$ and hence $d_X(r(x),x)<n\cdot|Fn|\cdot\delta'\le\delta$, which implies $d_Y(f\circ r(x),f(x))<\e/\sqrt[p]{n}$ by the choice of $\delta$. Then 
$$d^p_{Y^n}(f\circ r\circ s\circ g_1,f\circ s\circ g_1)<\sqrt[p]{n}\cdot\frac{\e}{\sqrt[p]{n}}=\e.$$
The continuity of $f$ and $\sup_{x\in X}d_X(s(x),x)=0$ imply $\sup_{x\in X}d_Y(f\circ s(x),f(x))=0$.
Then 
\begin{multline*}
d^p_{FY}(Ff(a),Ff(b))\le d^p_{Y^n}(f\circ f_0,f\circ g_0)+d^p_{Y^n}(f\circ f_1,f\circ g_1)=\\
=d^p_{Y^n}(f\circ f_0,f\circ s\circ f_0)+d^p_{Y^n}(f\circ r\circ s\circ g_1,f\circ g_1)\le\\
\le 0+d^p_{Y^n}(f\circ r\circ s\circ g_1,f\circ s\circ g_1)+d^p_{Y^n}(f\circ s\circ g_1,f\circ g_1)<0+\e+0=\e.
\end{multline*}
\end{proof}

\begin{example}\label{ex:projsquare} Consider the functor $F:\Set\to\Set$ assigning to each set $X$ the set $FX=\{\emptyset\}\cup \{\{x,y\}:x,y\in X,\;x\ne y\}$ and to each function $f:X\to Y$ the function $Ff:FX\to FY$ such that $Ff(\emptyset)=\emptyset$ and $Ff(\{x,y\})=\{f(x)\}\triangle\{f(y)\}$ for any $\{x,y\}\in FX\setminus\{\emptyset\}$. It is clear that $F$ is finitary and has finite degree $\deg(F)=2$, but $F$ does not preserve supports.
It can be shown that for the continuous map $f:\IR\to\IR$, $f:x\mapsto x^3$, the map $F^pf:F^p\IR\to F^p\IR$ is not continuous.  
This example shows that the preservation of supports in Theorem~\ref{t:continuous} is essential.
\end{example}

\section{Metric properties of the maps $\xi^a_{X}:X\to FX$.}

Given any element $a\in F1$ consider the map $\xi_X^a:X^1\to FX$, $\xi^a_X:f\mapsto Ff(a)$. If the set $X$ is not empty, the the power $X^1$ can be identified with $X$ by identifying each point $x\in X$ with the function $\bar x_{X}:1\to\{x\}\subseteq X$. In this case $\xi^a_X(x)=F\bar x_{X}(a)$. If $X$ is empty, then $X^1$ is empty and hence is equal to $X$. In both cases we can consider the map $\xi^a_X$ as a function defined on $X$.

The following proposition shows that $\xi^a$ is a natural transformation of the identity functor into the functor $F$.

\begin{proposition}\label{p:xi-natural} For any $a\in F1$ and any function $f:X\to Y$ between sets, we have $Ff\circ \xi^a_X=\xi^a_Y\circ f$.
\end{proposition}

\begin{proof} For every $x\in X$ and the map $\bar x_X:1\to\{x\}\subset X$, we have 
$$Ff\circ\xi^a_X(x)=Ff\circ F\bar x_X(a)=F(f\circ \bar x_X)(a)=\xi^a_{Y}(f(x))=\xi^a_Y\circ f(x).$$
\end{proof}

By Theorem~\ref{t:alt}, for every $a\in F1$ the map $\xi_X^a:(X,d_X)\to (FX,d^p_{FX})$ is non-expanding. In this section we find conditions guaranteeing that this map is an injective isometry.

\begin{lemma}\label{l:delta} Let $a\in F1$ and assume that $|X|\ge 2$.
\begin{enumerate}
\item[\textup{1)}] If the map $\xi_2^a:2\to F2$ is not injective, then
\begin{enumerate}
\item[\textup{1a)}] the map $\xi^a_X:X\to FX$ is constant;
\item[\textup{1b)}] $\supp(\xi^a_X(x))=\emptyset$ for every $x\in X$.
\end{enumerate}
\item[\textup{2)}] If the map $\xi^a_2:2\to F2$ is injective, then
\begin{enumerate}
\item[\textup{2a)}] $a\notin F[\emptyset;1]\subseteq F1$;
\item[\textup{2b)}] $\xi_2^a(0)\notin F[\{1\},2]$;
\item[\textup{2c)}] $\xi_X^a(x)\notin F[X\setminus\{x\};X]$ for every $x\in X$;
\item[\textup{2d)}] $\supp(\xi^a_X(x))=\{x\}$ for any $x\in X$;
\item[\textup{2e)}] the map $\xi^a_X:X\to FX$ is injective.
\end{enumerate}
\end{enumerate}
\end{lemma}

\begin{proof} 1. Assume that the map $\xi^a_2:2\to F2$ is not injective. Then $F\bar 0_{2}(a)=\xi^a_2(0)=\xi^a_2(1)=F\bar 1_{2}(a)$.

1a. To prove that the map $\xi^a_X:X\to FX$ is constant, choose any distinct points $x,y\in X$. Let $f:2\to X$ be the map defined by $f(0)=x$, $f(1)=y$. It follows that $\bar x_{X}=f\circ \bar 0_{2}$ and $\bar y_{X}=f\circ \bar 1_{2}$. Then 
$$
\xi^a_X(x)=F\bar x_{X}(a)=Ff\circ F\bar 0_{2}(a)=Ff(\xi_2^a(0))=Ff(\xi_2^a(1))=Ff\circ F\bar 1_2(a)=F\bar y_X(a)=\xi^a_X(y).
$$

1b. Therefore, for any points $x,y\in X$ we have  $\xi^a_X(x)=\xi^a_X(y)=F\bar y_X(a)\in F[\{y\};X]$ and hence $\supp(\xi^a_X(x))\subseteq \bigcap_{y\in X}\{y\}=\emptyset$.
\smallskip

2. Assume that the map $\xi^a_2:2\to F2$ is injective.

2a. First we show that $a\notin F[\emptyset;1]$. To derive a contradiction, assume that $a\in F[\emptyset;1]$. In this case we can find an element $a'\in F\emptyset$ such that $a=Fi_{\emptyset,1}(a')$. It follows that for every $z\in 2$
$$\xi_2^a(z)=F\bar z_2(a)=F\bar z_2(Fi_{\emptyset,1}(a'))=F(\bar z_2\circ i_{\emptyset,1})(a')=Fi_{\emptyset,2}(a')$$and hence the map $\xi_2^a:2\to \{Fi_{\emptyset,2}(a')\}\subseteq F2$ is constant, which contradicts our assumption. This contradiction shows that $a\notin F[\emptyset;1]$.
 \smallskip

2b. Now we prove that $\xi_2^a(0)\notin F[\{1\};2]$. To derive a contradiction, assume that $\xi_2^a(0)\in F[\{1\};2]=Fi_{\{1\},2}[F\{1\}]$. Then we can find an element $a'\in F\{1\}$ such that $\xi_2^a(0)=Fi_{\{1\},2}(a')$. Observe that $\bar 1_2=i_{\{1\},2}\circ g$ where $g:1\to\{1\}$ is the unique bijection. Since $Fg:F1\to F\{1\}$ is a bijection, there exists an element $a''\in F1$ such that $a'=Fg(a'')$. It follows that $$F\bar 0_2(a)=\xi_2^a(0)=Fi_{\{1\},2}(a')=Fi_{\{1\},2}(Fg(a''))=F(i_{\{1\},2}\circ g)(a'')=F\bar 1_2(a'').$$ Now consider the constant function $\gamma:2\to1$ and observe that $\gamma\circ \bar 0_2=\gamma\circ \bar 1_2$ is the identity map of the singeton $1$. Then $F\gamma\circ F\bar 0_2=F\gamma\circ F\bar 1_2$ is the identity map of $F1$.
Consequently,
$$a''=F\gamma\circ F\bar 1_2(a'')=F\gamma\circ F\bar 0_2(a)=a$$and then $\xi_2^a(0)=F\bar 0_2(a)=F\bar 1_2(a'')=F\bar 1_2(a)=\xi_2^a(1)$, which contradicts the injectivity of the map $\xi^a_2:2\to F2$.
\smallskip

2c. Given any $x\in X$ we will show that $\xi_X^a(x)\notin F[X\setminus\{x\};X]$. To derive a contradiction, assume that $\xi_X^a(x)\in F(A;X)$ where $A:=X\setminus\{x\}$.
Then $\xi_X^a(x)=Fi_{A,X}(a')$ for some $a'\in FA$. Let $\phi:A\to\{1\}$ be the constant function and $\chi:X\to 2$ be the function defined by $\chi(x)=0$ and $\chi(a)=1$ for any $a\in A$. It follows that $\bar 0_2=\chi\circ \bar x_X$ and $\chi\circ i_{A,X}=i_{\{1\},2}\circ \phi$. Then 
\begin{multline*}
\xi_2^a(0)=F\bar 0_2(a)=F\chi\circ F\bar x_X(a)=F\chi(\xi_X^a(x))=F\chi\circ Fi_{A,X}(a')=\\
=Fi_{\{1\},2}\circ F\phi(a')\in Fi_{\{1\},2}[F\{1\}]\subseteq F[\{1\},2],
\end{multline*}
which contradicts (2b).
\smallskip

2d. Now we are able to prove that $\supp(\xi^a_X(x))=\{x\}$ for every $x\in X$. It follows from
$$
\xi_X^a(x)=F\bar x_X(a)=Fi_{\{x\},X}\circ F\bar x_{\{x\}}(a)\in Fi_{\{x\},X}[F\{x\}]=F[\{x\};X]$$ that $\supp(\xi^a_X(x))\subseteq\{x\}$. 

Assuming that $\supp(\xi^a_X(x))\ne\{x\}$, we conclude that $\supp(\xi^a_X(x))=\emptyset$. By Proposition~\ref{p:BMZ}, $\xi^a_X(x)\in F[X\setminus\{x\};X]$, which contradicts the statement (2c).
\smallskip

2e. By the statement (2d), for any distinct points $x,y\in X$ the elements $\xi^a_X(x)$ and $\xi^a_X(y)$ are distinct as they have distinct supports. So, the function $\xi^a_X:X\to FX$ is injective.
\end{proof}

\begin{lemma}\label{l:xi} For any element $a\in F1$, the map $\xi^a_X:(X,d_X)\to (FX,d^p_{FX})$ is an injective isometry if the map $\xi^a_2:2\to F2$ is injective, the distance $d_X$ is an $\infty$-metric and either $p=1$ or $F$ preserves supports.
\end{lemma}

\begin{proof} Assume that the map $\xi^a_2:2\to F2$ is injective, $d_X$ is an $\infty$-metric and either $p=1$ or $F$ preserves supports. By Lemma~\ref{l:delta} and Theorem~\ref{t:alt}, the map $\xi^a_X:(X,d_X)\to (FX,d^p_{FX})$ is injective and non-expanding.
To show that $\xi^a_X$ is an isometry, fix any distinct elements $x,y\in X$. We need to show that $d^p_{FX}(\xi_X^a(x),\xi_X^a(y))\ge d_X(x,y)$.
 By Lemma~\ref{l:delta}(2d), $\supp(\xi^a_X(z))=\{z\}$ for any $z\in X$.

If $p=1$, then by  Theorem~\ref{t:d1},
$$
\begin{aligned}
d^p_{FX}(\xi^a_X(x),\xi^a_X(y))&\ge \min\{d_X(u,v):u,v\in\supp(\xi^a_X(x))\cup\supp(\xi^a_X(y)),\;u\ne v\}=\\
&=\min\{d_X(u,v):u,v\in\{x\}\cup\{y\},\;u\ne v\}=d_X(x,y).
\end{aligned}
$$

Next, we consider the case of functor $F$ that preserves supports. To derive a contradiction, assume that $d^p_{FX}(\xi^a_X(x),\xi^a_X(y))<d_X(x,y)$ and hence
 $$d^\infty_{FX}\big(\xi^a_X(x),\xi^a_X(y)\big)\le d^p_{FX}\big(\xi^a_X(x),\xi^a_X(y)\big)<d_X(x,y)$$ 
 according to Lemma~\ref{l:ineq}. 
 
 By the definition of the distance $d^\infty_{FX}$, there exists a $(\xi^a_X(x),\xi^a_X(y))$-linking chain $\big((a_i,f_i,g_i)\big){}_{i=0}^l\in L_{FX}(a,b)$ such that $\sum_{i=0}^ld^\infty_{X^{\!<\!\w}}(f_i,g_i)<d_X(x,y)$.
Let $b_0:=\xi^a_X(x)=Ff_0(a_0)$, $b_{l+1}:=\xi^a_X(y)=Fg_l(a_l)$ and $b_i=Ff_i(a_i)=Fg_{i-1}(a_{i-1})$ for $i\in\{1,\dots,l\}$. Since the functor $F$ preserves supports, $$\supp(b_i)=f_i[\supp(a_i)]\mbox{ \ and \ }\supp(b_{i+1})=g_i[\supp(a_i)]$$ for all $i\in\{0,\dots,l\}$. Consequently, $d_{HX}(\supp(b_i),\supp(b_{i+1}))\le d^\infty_{X^{\!<\!\w}}(f_i,g_i)$ for every $i\in\{0,\dots,l\}$.
Here by $d_{HX}$ we denote the Hausdorff distance on the hyperspace $[X]^{<\w}$ of all finite subsets of $X$. It follows that
\begin{multline*}
d_X(x,y)=d_{HX}(\{x\},\{y\})=d_{HX}(\supp(b_0),\supp(b_{k+1}))\le\\
\le\sum_{i=0}^ld_{HX}(\supp(b_i),\supp(b_{i+1}))
\le\sum_{i=0}^ld^\infty_{X^{\!<\!\w}}(f_i,g_i)<d_X(x,y),
\end{multline*}
which is a desired contradiction completing the proof of the inequality\newline $d^p_{FX}(\xi^a_X(x),\xi^a_X(y))\ge d_X(x,y)$.
\end{proof}

Now we can prove the main result of this section.

\begin{theorem}\label{t:xi} For an element $a\in F1$, the map $\xi^a_X:(X,d_X)\to (FX,d^p_{FX})$ is an injective isometry if the map $\xi^a_2:2\to F2$ is injective and either $p=1$ or $F$ preserves supports.
\end{theorem}

\begin{proof} Assume that the map $\xi^a_2:2\to F2$ is injective and either $p=1$ or $F$ preserves supports. By Lemma~\ref{l:delta} and Theorem~\ref{t:alt}, the map $\xi^a_X:(X,d_X)\to (FX,d^p_{FX})$ is injective and non-expanding.
To show that $\xi^a_X$ is an isometry, it remains to check that $d^p_{FX}(\xi^a(x),\xi^a(y))\ge d_X(x,y)$ for any points $x,y\in X$.

On the set $X$ consider the equivalence relation $\approx$ defined by $x\approx y$ iff $d_X(x,y)=0$. For any $x\in X$ let $O[x;0]=\{y\in X:d_X(x,y)\le0\}$ be its equivalence class. The quotient set $\tilde X=\{O[x;0]:x\in X\}$ of $X$ by the equivalence relation $\approx$ carries a unique $\infty$-metric $d_{\tilde X}$ such that the quotient map $q:X\to\tilde X$, $q:x\mapsto O[x;0]$, is an isometry.
Since $d_{\tilde X}$ is an $\infty$-metric, we can apply Lemma~\ref{l:xi} and conclude that for any points $x,y\in \tilde X$ we have the equality $d^p_{F\tilde X}(\xi^a_{\tilde X}(x),\xi^a_{\tilde X}(y))=d_{\tilde X}(x,y)$.

By Theorem~\ref{t:Lip}, the map $Fq:(FX,d^p_{FX})\to (F\tilde X,d^p_{F\tilde X})$ is non-expanding. Consequently, for any points $x,y\in X$ we obtain the desired inequality:
\begin{multline*}
d^p_{FX}(\xi_X^a(x),\xi_X^a(y))\ge d^p_{F\tilde X}(Fq\circ \xi^a_X(x),Fq\circ\xi^a_X(y))=\\=d^p_{F\tilde X}\big(\xi^a_{\tilde X}(q(x)),\xi^a_{\tilde X}(q(y))\big)=d_{\tilde X}(q(x),q(y))=d_X(x,y),
\end{multline*}
witnessing that the map $\xi^a_X:X\to F^pX$ is an isometry.
\end{proof}

\section{Preserving the compactness by the functor $F^p$}

A distance space $(X,d_X)$ is called {\em compact} if each sequence in $X$ contains a convergent subsequence. It is a standard exercise to prove that a distance space $X$ is compact if and only if it is compact in the topological sense (i.e., each open cover of $X$ has a finite subcover). Here on $X$ we consider the topology consisting of sets $U\subseteq X$ such that for every $x\in X$ there exists $\e>0$ such that the $\e$-ball $O(x;\e)=\{y\in X:d_X(y,x)<\e\}$ is contained in $U$. Discussing topological properties of distance spaces we have in mind this topology generated by the distance. The topology of a distance space $(X,d_X)$ is Hausdorff if and only if the distance $d_X$ is an $\infty$-metric.

\begin{theorem}\label{t:top} If the functor $F$ is finitary and has finite degree, then: \begin{enumerate}
\item[\textup{1)}] For any compact distance space $(X,d_X)$ the distance space $(FX,d_{FX})$ is compact;
\item[\textup{2)}] For any continuous map $f:(X,d_X)\to (Y,d_Y)$ between compact distance spaces the map $Ff:(FX,d^p_{FX})\to (FY,d^p_{FY})$ is continuous.
\end{enumerate}
\end{theorem}

\begin{proof} 1. Let $(X,d_X)$ be a compact distance space. If $X=\emptyset$, then by the finitarity of the functor $F$, the set $FX=F\emptyset$ is finite and hence compact. So, we assume that $X\ne \emptyset$.

Since the functor $F$ has finite degree, there exists $n\in\IN$ such that the map $\xi_X:Fn\times X^n\to FX$, $\xi_X:(a,f)\mapsto Ff(a)$, is surjective. The finitary property of the functor $F$ guarantees that the space $Fn$ is finite. Endow the space $Fn$ with the discrete topology.

By Theorem~\ref{t:alt}, for every $a\in Fn$ the map $\xi^a_X:(X^n,d^n_{X^{\!<\!\w}})\to (FX,d^p_{FX})$ is non-expanding and hence continuous. Then the distance space $(FX,d^p_{FX})$ is compact, being a continuous image of the compact topological space $Fn\times X^n$.
\smallskip

2. Let $f:X\to Y$ be any continuous map between compact distance spaces $(X,d_X)$ and $(Y,d_Y)$. By the compactness of $X$, the map $f$ is uniformly continuous. By Corollary~\ref{c:uniform}, the map $F^pf:F^pX\to F^pY$ is uniformly continuous and hence continuous.
\end{proof}

\section{Preserving the completeness by the functor $F^p$}

In this section we adress the problem of preservation of the completeness of distance spaces by the functor $F^p$. The completeness of distance spaces is defined by analogy with the completeness of metric spaces.

A sequence $\{x_n\}_{n\in\w}\subseteq X$ in a distance space $(X,d_X)$ is {\em Cauchy} if for every $\e\in(0,+\infty)$ there exists $n\in\w$ such that $d_X(x_m,x_k)<\e$ for any $m,k\ge n$.

A distance space $(X,d_X)$ is {\em complete} if every Cauchy sequence $\{x_n\}_{n\in\w}\subseteq X$ converges to some point $x_\infty\in X$. The latter means that for every $\e\in(0,+\infty)$ there exists $n\in\w$ such that $d_X(x_m,x_\infty)<\e$ for any $m\ge n$. 

The following theorem is the principal result of this section.

\begin{theorem}\label{t:complete} Assume that the functor $F$ is finitary and has finite degree. If a distance space $(X,d_X)$ is complete, then so is the distance space $F^pX$.
\end{theorem}

The proof of this theorem is preceded by several definitions and lemmas.
Let $\Omega$ be an infinite subset of $\w$. A sequence $\{x_n\}_{n\in\Omega}\subseteq X$ is called
\begin{itemize}
\item {\em convergent} if there exists $x\in X$ such that $\lim\limits_{\Omega\ni i\to\infty}d_X(x_i,x)=0$, i.e., \newline
$\exists x\in X\;\forall \e\in(0,+\infty)\;\exists F\in[\Omega]^{<\w}\;\forall i\in\Omega\setminus F\;\; d_X(x_i,x)<\e$;
\item {\em Cauchy} if $\lim_{\Omega\ni i,j\to\infty}d_X(x_i,x_j)=0$, i.e.,\newline
$\forall \e\in(0,\infty)\;\exists F\in[\Omega]^{<\w}\;\forall i,j\in\Omega\setminus F\;\;d_X(x_i,x_j)<\e$;
\item {\em $\e$-separated} for some $\e\in(0,\infty)$ if $d_X(x_i,x_j)\ge\e$ for any distinct numbers $i,j\in\Omega$;
\item {\em separated} if it is $\e$-separated for some $\e>0$.
\end{itemize}

The following lemma is well-known for metric spaces and can be proved by analogy for distance spaces.

\begin{lemma}\label{l:css} A Cauchy sequence $(x_n)_{n\in\w}$ in a distance space is convergent if and only if contains a convegent subsequence $(x_n)_{n\in\Omega}$. 
\end{lemma}

 For an infinite subset $\Omega$ of $\w$ we put $[\Omega]^2:=\{(i,j)\in \Omega\times\Omega:i<j\}$.

\begin{lemma}\label{l:dichotomic} For every $n\in\IN$, and every sequence $\{f_i\}_{i\in\w}\subseteq X^n$ there exists an infinite subset $\Omega\subseteq\w$ such that for any $k,m\in n$ either  $$\lim_{\Omega\ni i,j\to\infty}d_X(f_i(k),f_j(m))=0\mbox{  \  or \ }\inf_{(i,j)\in[\Omega]^2}d_X(f_i(k),f_j(m))>0.$$\end{lemma}

\begin{proof} To prove the lemma for  $n=1$, it suffices to check that each sequence $\{x_i\}_{i\in\w}\subseteq X$ contains a subsequence $(x_i)_{i\in\Omega}$, which is either separated or Cauchy. Assuming that $(x_n)_{n\in\w}$ contains no separated subsequences, we will construct a Cauchy subsequence in $(x_n)_{n\in\omega}$.

For every $k\in\w$, consider the function $\chi_k:[\w]^2\to\{0,1\}$ defined by 
$$\chi_k(i,j)=\begin{cases}0,&\mbox{if $d_X(x_i,x_j)<2^{-k}$}\\
1,&\mbox{otherwise}.
\end{cases}
$$
Applying the  Ramsey Theorem \cite[1.3]{Tod}, find an infinite subset $\Omega_0$ of $\w$ such that $\chi_0[[\Omega_0]^2]$ is a singleton. If $\chi_0[[\Omega_0]^2]=\{1\}$, then $(x_n)_{n\in\Omega_0}$ is a $1$-separated subsequence of $(x_n)_{n\in\w}$, which is forbidden by our assumption. Therefore, $\chi_0[[\Omega_0]^2]=\{0\}$. Applying the Ramsey Theorem inductively, we can construct a decreasing sequence of infinite sets $(\Omega_k)_{k\in\w}$ such that $\chi_k[[\Omega_k]^2]=\{0\}$ for every $k\in\w$. After completing the inductive construction, choose an infinite set $\Omega\subseteq \w$ such that $\Omega\setminus \Omega_k$ is finite for every $k\in\w$, and observe that the sequence $(x_n)_{n\in\Omega}$ is Cauchy. This completes the proof of the lemma for $n=1$.
\smallskip

Now assume that that for some $n\ge2$ we have proved that for every $l<n$, every sequence $\{f_i\}_{i\in\w}\subseteq X^{l}$  contains a  subsequence $(f_i)_{i\in\Omega}$ such that for every $k,m\in l$, either
$$\lim_{\Omega\ni i,j\to\infty}d_X(f_i(k),f_j(m))=0\mbox{  \  or \ }\inf_{(i,j)\in[\Omega]^2}d_X(f_i(k),f_j(m))>0.$$
Fix any sequence $\{f_i\}_{i\in\w}\subseteq X^n$. By the inductive assumption, there exists an infinite subset $\Omega'\subseteq\w$ such that for every $k,m\in n-1$, either  
$$\lim_{\Omega'\ni i,j\to\infty}d_X(f_i(k),f_j(m))=0\mbox{  \  or \ }\e':=\inf_{(i,j)\in[\Omega']^2}d_X(f_i(k),f_j(m))>0.$$
Applying the inductive assumption once more, find an infinite set $\Omega\subseteq\Omega'$ such that 
$$\lim_{\Omega\ni i,j\to\infty}d_X(f_i(n-1),f_j(n-1))=0\mbox{  \  or \ }\inf_{(i,j)\in[\Omega]^2}d_X(f_i(n-1),f_j(n-1))\ge\e$$
for some $\e\in(0,\e']$. Let $C$ be the set of all numbers $k\in n$ such that the sequence $(f_i(k))_{i\in\Omega}$ is Cauchy.

Two cases are possible.

1) $n-1\in C$. This case has two subcases.

1a) For some infinite subset $\Lambda\subseteq \Omega$ and some $l\in n-1$ we have $$\lim\limits_{\Lambda\ni i\to\infty}d_X(f_i(n-1),f_i(l))=0.$$ Replacing $\Lambda$ by a smaller infinite subset, we can additionally assume that
$$\sup\limits_{j\in\Lambda}d_X(f_j(n-1),f_j(l))<\tfrac12\e.$$
We claim that for any $k,m\in n$ with $\lim_{\Lambda\ni i,j\to\infty}d_X(f_i(k),f_j(m))\ne 0$ we have\break $\inf_{(i,j)\in[\Lambda]^2}d_X(f_i(k),f_j(m))\ge\tfrac12\e$. If $k,m\in n-1$, then this follows from the definition of $\e'\ge\e>\frac12\e$. Now assume that $k$ or $m$ does not belong to $n-1$. The case $k=m=n-1$ is impossible as $n-1\in C$. So, $k\ne m$ and we lose no generality assuming that $k<m=n-1$. It follows from $$\lim\limits_{\Lambda\ni i\to\infty}d_X(f_i(k),f_i(n-1))\ne 0=\lim_{\Lambda\ni i\to\infty}d_X(f_i(n-1),f_i(l))$$ that $\lim\limits_{\Lambda\ni i\to\infty}d_X(f_i(k),f_j(l))\ne0$ and then 
$$
\inf_{(i,j)\in[\Lambda]^2}d_X(f_i(k),f_j(n{-}1))\ge \inf_{(i,j)\in[\Lambda]^2}d_X(f_i(k),f_j(l))-\sup_{j\in\Lambda}d_X(f_j(n{-}1),f_j(l))\ge
 \e'-\tfrac12\e\ge\tfrac12\e.
$$

1b) For every infinite set $\Lambda\subseteq \Omega$ and every $l\in n-1$ we have\newline $\lim\limits_{\Lambda\ni i\to\infty}d_X(f_i(n-1),f_i(l))\ne 0$. In this case, we can find an infinite subset $\Lambda'\subseteq\Omega$ and a number $\delta\in(0,\frac14\e]$ such that $\inf\limits_{k\in n-1}\inf\limits_{i\in\Lambda'}d_X(f_i(k),f_i(n-1))\ge 2\delta$.  Find an infinite subset $\Lambda''\subseteq\Lambda'$ such that $\sup_{i,j\in\Lambda''}d_X(f_i(k),f_j(k))<\delta$ for every $k\in C$. By our assumption, $n-1\in C$ and for every $k\in n\setminus C$ the sequence $(f_i(k))_{i\in \Omega'}$ is $\e'$-separated. Then for the smallest number $\lambda$ of the set $\Lambda''$,  we can find an infinite subset $\Lambda\subseteq\Lambda''$ such that $$\min_{k\in n\setminus C}\inf_{i\in\Lambda}d_X(f_i(k),f_\lambda(n-1))\ge\tfrac12\e'.$$
We claim that for any $k,m\in n$, $$\lim_{\Lambda\ni i,j\to\infty}d_X(f_i(k),f_j(m))\ne 0\;\Ra\;\inf_{(i,j)\in[\Lambda]^2}d_X(f_i(k),f_j(m))\ge\delta.$$ If $k,m\in n-1$, then this follows from the definition of $\e'\ge\e\ge4\delta$. Now assume that $k$ or $m$ does not belong to $n-1$. We lose no generality assuming that $k\le m=n-1$. If $k\in C$, then $\lim_{\Lambda\ni i,j\to\infty}d_X(f_i(k),f_j(n-1))\ne 0$ implies $\lim_{\Lambda\ni i\to\infty}d_X(f_i(k),f_i(n-1))\ne 0$ and hence 
$$\inf_{(i,j)\in[\Lambda]^2}\!d_X(f_i(k),\!f_j(n-1))\ge \inf_{i\in\Lambda}d_X(f_i(k),\!f_i(n-1))-\sup_{i,j\in \Lambda}d_X(f_j(n-1),\!f_i(n-1))\ge 2\delta-\delta.$$
If $k\in n\setminus C$, then 
$$\!\inf_{(i,j)\in[\Lambda]^2}\!d_X(f_i(k),\!f_j(n-1))\ge \inf_{i\in\Lambda}\!d_X(f_i(k),\!f_\lambda(n-1))-\sup_{j\in\Lambda}d_X(f_\lambda(n-1),\!f_j(n-1))\ge \tfrac12\e'-\delta.$$

2) Now consider the second case: $n-1\notin C$. In this case the sequence $(f_i(n-1))_{i\in\Omega}$ is $\e$-separated by the choice of $\e$. Find an element $\lambda_0\in\Omega$ such that $$\sup\{d_X(f_i(k),f_j(k)):i,j\in\Omega\cap[\lambda_0,\infty)\}<\tfrac14\e\quad\mbox{for every $k\in C$}.$$ Using the $\e$-separatedness of the sequences $(f_i(l))_{i\in\Omega}$ for $l\in n\setminus C$, we can inductively construct an increasing sequence $\{\lambda_k\}_{k\in\w}\subseteq\Omega$ such that for every $k\in\w$ the following condition holds:
\begin{itemize}
\item[$(*)$] for any $i\le \lambda_k$, $j\in \Omega\cap[\lambda_{k+1},\infty)$, $m\in n$ and $l\in n\setminus C$ we have $d_X(f_j(l),f_{i}(m))\ge\frac12\e$.
\end{itemize}
Let $\Lambda=\{\lambda_i:i\in\IN\}$. We claim that for any $k,m\in n$,
$$\lim_{\Lambda\ni i,j\to\infty}d_X(f_i(k),f_j(m))\ne 0\;\Ra\;\inf_{(i,j)\in[\Lambda]^2}d_X(f_i(k),f_j(m))\ge\tfrac14\e.$$ If $k,m\in n-1$, then this follows from the definition of $\e'\ge\frac14\e$. Now assume that $k$ or $m$ does not belong to $n-1$. We lose no generality assuming that $k\le m=n-1$. Choose any distinct numbers $\lambda_i,\lambda_j\in\Lambda$. If $k\in C$, then
$$d_X(f_{\lambda_i}(k),f_{\lambda_j}(n-1))\ge d_X(f_{\lambda_0}(k),f_{\lambda_j}(n-1))-d_X(f_{\lambda_0}(k),f_{\lambda_i}(k))\ge \tfrac12\e-\tfrac14\e=\tfrac14\e,$$
by $(*)$ and the choice of $\lambda_0$. If $k\notin C$, then
$d_X(f_{\lambda_i}(k),f_{\lambda_j}(n-1))\ge\frac12\e$ by $(*)$.
\end{proof}

\begin{proof}[Proof of Theorem~\ref{t:complete}] 
Given any Cauchy sequence $\{b_i\}_{i\in\w}\subseteq F^pX$, we should prove that it converges to some element $b_\infty\in FX$. 
Let $s\in\w$ be the smallest number for which there exists an infinite set $\Omega'\subseteq\w$ and a sequence $\{b_i'\}_{i\in\w}\subseteq FX$ such that 
\begin{itemize}
\item[(a)] $\lim\limits_{\Omega'\ni i\to\infty}d^p_{FX}(b_i,b_i')=0$;
\item[(b)] $d^p_{FX}(b_i',b_j')<\infty$ for all $i,j\in\Omega'$;
\item[(c)] $|\supp(b_i')|=s$ for all $i\in\Omega'$.
\end{itemize}
Since the functor $F$ has finite degree, the number $s$ is well-defined and $s\le \deg(F)$. Fix an infinite set $\Omega'\subseteq\w$ satisfying the conditions (a)--(c). The condition (a) implies that the sequence $(b_i')_{i\in\Omega'}$ is Cauchy. If $s=0$, then Theorem~\ref{t:small} and the condition (b) imply $b_i'=b_j'$ for all $i,j\in\Omega'$. In this case the sequence $(b_i')_{i\in\Omega'}$ is convergent and so is the sequence $(b_i)_{i\in\w}$ (to the same limit point).

So, we assume that $s>0$. Using Proposition~\ref{p:BMZ}, for every $i\in\Omega'$,  find a bijective map $h_i:s\to \supp(b_i')\subseteq X$ and an element $a_i\in Fs$ such that $b'_i=Fh_i(a_i)$. Since $F$ is finitary, the set $Fs$ is finite. Then by the Pigeonhole Principle, there exists an element $a\in Fs$ such that the set $\Omega''=\{i\in\Omega':a_i=a\}$ is infinite. By Lemma~\ref{l:dichotomic}, there exists an infinite set $\Omega\subseteq\Omega''$ and $\e\in(0,\infty)$ such that for any $k,m\in s$ either
$$\lim_{\Omega\ni i,j\to\infty}d_X(h_i(k),h_j(m))=0\mbox{ \ or \ }\inf_{(i,j)\in[\Omega]^2}d_X(h_i(k),h_j(m))\ge\e.$$
Replacing $\Omega$ by a smaller infinite subset of $\Omega$, we can additionally assume that for any $k,m\in s$,
$$\lim_{\Omega\ni i,j\to\infty}d_X(h_i(k),h_j(m))=0\;\Ra\;\sup_{i,j\in\Omega}d_X(h_i(k),h_j(m))<\tfrac14\e.$$

 Let $C$ be the set of all numbers $k\in s$ such that the sequence $(h_i(k))_{i\in\Omega}$ is Cauchy and hence converges to some point $h(k)\in X$ of the complete distance space $(X,d_X)$. We can assume that $h(k)=h(m)$ for any $k,m\in C$ with $d_X(h(k),h(m))=0$. The choice of $\e$ guarantees that $$d_X(h(k),h(m))\ge \inf_{(i,j)\in[\Omega]^2}d_X(h_i(k),h_j(m))\ge\e$$ for any $k,m\in C$ with $d_X(h(k),h(m))>0$. Consequently, the set $h[C]:=\{h(k):k\in C\}$ is $\e$-separated. We claim that $C=s$. To derive a contradiction, assume that $C\ne s$.

Consider the (unique) map $r:X\to X$ such that 
$$r(x)=\begin{cases}y&\mbox{if $d_X(x,y)<\frac14\e$ for some $y\in h[C]$};\\
x&\mbox{otherwise}.
\end{cases}
$$
Observe that for every $k\in C$, $m\in s\setminus C$, and distinct numbers $i,j\in\Lambda$, we have
$$d_X(r\circ h_i(k),r\circ h_j(m))=d_X(h(k),h_j(m))\ge d_X(h_i(k),h_j(m))-d_X(h_i(k),h(k))\ge \e-\tfrac14\e.$$

For every $i\in\Omega$, consider the element $b_i''=Fr\circ Fh_i(b)$ and observe that
$$d^p_{FX}(b_i'',b_i')\le d^p_{X^s}(r\circ h_i,h_i)\le\Big(\sum_{k\in S'}d_X(h_i(k),h(k))^p\Big)^{\frac1p}.$$
The latter inequality implies that $$\lim_{\Omega_s\ni i\to\infty}d^p_{FX}(b_i,b_i'')\le
\lim_{\Omega_s\ni i\to\infty}d^p_{FX}(b_i,b_i')+\lim_{\Omega_s\ni i\to\infty}d^p_{FX}(b_i',b_i'')=0+0=0$$and hence the sequence $(b_i'')_{i\in\Omega_s}$ is Cauchy. Observe that for every $i\in\Omega_s$ we have $\supp(b_i'')\subseteq r\circ h_i[s]$ and hence $|\supp(b_i'')|\le|s|=s$. Now the choice of $s$ ensures that the set $\{i\in\Omega:|\supp(b_i'')|<s\}$ is finite. Since $(b_i'')_{i\in\Omega}$ is Cauchy, we can find an infinite set $\Lambda\subseteq\Omega$ such that
\begin{itemize}
\item[(d)] $\supp(b_i'')=r\circ h_i[s]$ for all $i\in\Lambda$;
\item[(e)] $d^1_{FX}(b''_i,b''_j)\le n^{1-\frac1p}d^p_{FX}(b_i'',b_j'')<\tfrac12\e$ for all $i,j\in\Lambda$.
\end{itemize}
The first inequality in (e) follows Theorem~\ref{t:ineq}.

Fix any distinct numbers $i,j\in\Lambda$. If $C=\emptyset$, then 
\begin{multline*}
\ud_X(\supp(b''_i)\cup\supp(b''_j))=\ud_X(r\circ h_i[s]\cup r\circ h_j[s])=\ud_X(h_i[s],h_j[s])=\\
\min_{k,m\in s}d_X(h_i(k),h_j(m))\ge\e.
\end{multline*}
If $h[C]$ is a singleton, then 
\begin{multline*}
\ud_X(\supp(b''_i)\cup\supp(b''_j))=\ud_X(r\circ h_i[s]\cup r\circ h_j[s])=\\
\min\{\ud_X(h_i[s\setminus C],h_j[s\setminus C]),\ud_X(h[C],h_j[s\setminus C]),\ud_X(h_i[s\setminus C],h[C]\}\ge\min\{\e,\tfrac34\e,\tfrac34\e\}=\tfrac34\e.
\end{multline*}
If $|h[C]|>1$, then 
\begin{multline*}
\ud_X(\supp(b''_i)\cup\supp(b''_j))=\ud_X(r\circ h_i[s]\cup r\circ h_j[s])=\\
\min\{\ud_X(h[C]),\ud_X(h_i[s\setminus C],h_j[s\setminus C]),d_X(h[C],h_j[s\setminus C]),\ud_X(h_i[s\setminus C],h[C]\}\ge\\
\min\{\e,\e,\tfrac34\e,\tfrac34\e\}=\tfrac34\e.
\end{multline*}
In any case we obtain $\ud_X(\supp(b''_i)\cup\supp(b''_i))\ge\tfrac34\e>\frac12\e$.
The last inequality, condition (e) and Theorem~\ref{t:d1} imply that $b_i''=b_j''$. Then $\supp(b''_i)=\supp(b''_j)$ and the symmetric difference $\supp(b''_i)\Delta\supp(b''_j)$ is empty. On the other  hand, 
$$\supp(b_i'')\Delta\supp(b_j'')=(r\circ h_i[s])\Delta (r\circ h_j[s])=h_i[s\setminus C]\cup h_j[s\setminus C]\ne\emptyset.$$ This contradiction shows that $C=s$. Then then for the element $Fh(b)\in FX$ we have
$$\lim_{\Omega\ni i\to\infty}d^p_{FX}(b'_i,Fh(b))=\lim_{\Omega\ni i\to\infty}d^p_{FX}(Fh_i(b),Fh(b))\le \lim_{\Omega\ni i\to\infty} d^p_{X^s}(h_i,h)=0,$$which means that the sequence $(b_i')_{i\in\Omega}$ converges to $Fh(b)$ and so does the Cauchy sequence $(b_i)_{i\in\w}$ (as $\lim\limits_{\Omega\ni i\to\infty}d^p_{FX}(b_i,b_i')=0$).
\end{proof}

By a {\em completion} of a distance space $X$ we understand a pair $(Y,i)$ consisting of a complete distance space $(Y,d_Y)$ and an injective isometry $i:X\to Y$ such that the set $i[X]$ is dense in $Y$ and $d_Y(y,y')>0$ for every points $y\in Y\setminus i[X]$ and $y'\in Y\setminus\{y\}$. 

It can be shown that each distance space $X$ has a completion and this completion is unique in the sense that for any completions $(Y_1,i_1)$ and $(Y_2,i_2)$ of $X$ there exists a unique bijective isometry $j:Y_1\to Y_2$ such that $j\circ i_1=i_2$.

\begin{problem} Given a distance space $X$, a finitely supported functor $F:\Set\to\Set$ and number $p\in[1,\infty]$, describe the completion of the distance space $F^pX$.
\end{problem}

\begin{remark} Let $H:\Set\to\Set$ be the functor assigning to each set $X$ the set $HX$ of nonempty finite subsets of $X$ and to each function $f:X\to Y$ between sets the function $Hf:HX\to HY$, $Hf:A\mapsto f[A]$. For any complete metric space $X$ the completion of the metric space $H^\infty X$ can be identified with the space of all nonempty compact subsets of $X$, endowed with the Hausdorff metric. On the other hand, the completion of the metric space $H^1X$ can be identified with the space of all nonempty compact subsets of zero length in $X$, see \cite{BGW}.
\end{remark}

\section{The functor $F^p$ and the metric entropy of distance spaces}

For a distance space $(X,d_X)$ and an $\e\in(0,\infty)$, let $\IE_\e(X,d_X)$ be the smallest cardinality of a cover of $X$ by subsets of diameter $\le\e$. If the distance $d_X$ is clear from the context, then we write $\IE_\e(X)$ instead of $\IE_\e(X,d_X)$. 

For two numbers $\e,\delta\in(0,\infty)$ let 
$$\IE_{\e,\delta}(X)=\sup\{\IE_\e(A): A\subseteq X\;\wedge\;\diam(A)<\delta\}.$$
The functions $\IE_\e( X)$ and $\IE_{\e,\delta}( X)$ (of the variables $\e,\delta$) are called  {\em the metric entropy functions} of the distance space $X$.

\begin{theorem}\label{t:entropy1} If the functor $F$ has finite degree, then for every $\e\in[0,\infty)$, 
$$\IE_\e(F^pX)\le |Fn|\cdot\IE_\e(X^{n},d^p_{X^n})\le |Fn|\cdot\big(\IE_{\e/\!\sqrt[p]{n}}(X)\big)^n,$$
where $n=\max\{1,\deg(F)\}$.
\end{theorem}

\begin{proof} By Theorem~\ref{t:altm}, for every $a\in Fn$ the map $\xi^a:(X^n,d^p_{X^n})\to (FX,d^p_{FX})$ is non-expanding.
By the definition of the cardinal $\IE_{\e}(X^n,d^p_{X^n})$ there exists a cover of $\mathcal V$ of $X^n$ such that $|\V|=\IE_{\e}(X^n)$ and $\mesh(\V)\le\e$. 
 Taking into account that $FX=\!\bigcup\limits_{a\in Fn}\xi^a[X^n]$, we conclude that $\W=\{\xi^a[V]:a\in Fn,\;V\in\V\}$ is a cover of $F^pX$ of cardinality $|\W|\le|Fn|\cdot|\V|=|Fn|\cdot\IE_\e(X^n)$. The non-expanding property of the maps $\xi^a:X^n\to F^pX$ ensures that $\mesh(\W)\le\mesh(\V)\le\e$, witnessing that $\IE_{\e}(F^pX)\le|\W|\le |Fn|\cdot|\V|=|Fn|\cdot\IE_\e(X^n,d^p_{X^n})$.

It remains to prove the inequality $\IE_\e(X^n,d^p_{X^n})\le \big(\IE_{\e/\!\sqrt[p]{n}}(X)\big)^n$. By the definition of the cardinal  $\IE_{\e/\!\sqrt[p]{n}}(X)$, there exists a cover $\U$ of $X$ such that $|\U|=\IE_{\e/\!\sqrt[p]{n}}(X)$ and $\mesh(\U)\le\e/\!\sqrt[p]{n}$. For any sets $U_0,\dots,U_{n-1}\in\U$ consider their product $\prod_{i\in n}U_i \subseteq X^n$ and observe that 
$$\diam(\prod_{i\in n}U_i)\le \Big(\sum_{i\in n}\diam(U_i)^p\Big)^{\frac1p}\le n^{\frac1p}\cdot\mesh(\U)\le \e.$$
Then $\U'=\{\prod_{i\in n}U_i:U_0,\dots,U_{n-1}\in\U\}$ is a cover of $X^n$ witnessing that
$$\IE_\e(X^n,d^p_{X^n})\le |\U'|\le|\U|^n=\big(\IE_{\e/\!\sqrt[p]{n}}(X)\big)^n.$$
\end{proof}

Now we deduce an upper bound for the metric entropy function $\IE_{\e,\delta}(F^pX)$ of the distance space $F^pX$.

\begin{theorem}\label{t:entropy2} Assume that the functor $F$ has finite degree and preserves supports. Then for any $\e,\delta\in[0,\infty)$ we have 
$$\IE_{\e,\delta}(F^pX)\le n^n\cdot |Fn|\cdot \IE_{\e,\delta'}(X^n,d^p_{X^n})\le n^n\cdot |Fn|\cdot \big( \IE_{\e'\!,\delta'}(X)\big)^n,$$
where $n=\max\{1,\deg(F)\}$, $\e'=\e/\!\sqrt[p]{n}$ and $\delta'=2\delta\sqrt[p]{n}$.
\end{theorem}

\begin{proof} Given any nonempty subset $A\subseteq F^pX$ of  $\diam(A)< \delta<\infty$, we will prove that $$\IE_\e(A)\le   n^n\cdot |Fn|\cdot \IE_{\e,\delta'}(X^n,d^p_{X^n}).$$
If $\bigcup_{a\in A}\supp(a)=\emptyset$, then Theorem~\ref{t:small} implies that $A$ is a singleton and hence $\IE_\e(A)=|A|=1$. So, suppose that $\bigcup_{a\in A}\supp(a)\ne\emptyset$ and choose an element $a\in A$ with nonempty support $S=\supp(a)$. Since the functor $F$ has finite degree, $|S|\le n$. 

For every function $g\in S^n$, consider the ball
$$O_\infty(g;\delta):=\{f\in X^n:d^\infty_{X^n}(f,g)<\delta\}$$in the distance space $(X^n,d^\infty_{X^n})$. 

\begin{claim}
$A\subseteq \bigcup_{c\in Fn}\bigcup_{g\in S^n}\xi^c[O_\infty(g;\delta)].$
\end{claim}

\begin{proof} Given any element $b\in A$,  apply Theorem~\ref{t:shukel} to conclude that $$d_{H\!X}(\supp(a),\supp(b))\le d^p_{FX}(a,b)<\delta<\infty.$$
Then $\supp(b)\ne\emptyset$ and we can apply Proposition~\ref{p:BMZ} to find a function $f:n\to \supp(b)$ and an element $c\in Fn$ such that $b=Ff(c)$. Since $d_{HX}(\supp(b),S)<\delta$, there exists a function $r:\supp(b)\to S$ such that  $d_X(r(x),x)<\delta$ for every $x\in\supp(b)$. Consider the map $g=r\circ f:n\to S$ and observe that $d^\infty_{X^n}(f,g)<\delta$ and hence $f\in O_\infty(g;\delta)$. Then $b=Ff(c)=\xi^c(f)\in\xi^c[O_\infty(g;\delta)]$.
\end{proof}

By Theorem~\ref{t:ineq}, $O_\infty(g;\delta)\subseteq O_p(g;n^{\frac1p}\delta):=\{f\in X^n:d^p_{X^n}(f,g)<n^{\frac1p}\delta\}$. Taking into account that for every $c\in Fn$, the map $\xi^c:X^n\to F^pX$ is non-expanding (for the distance $d^p_{X^n}$ on $X$), we conclude that 
\begin{multline*}
\IE_\e(A)\le\sum_{c\in Fn}\sum_{g\in S^n}\IE_\e(\xi^c[O_\infty(g;\delta)])\le 
\sum_{c\in Fn}\sum_{g\in S^n}\IE_\e(\xi^c[O_p(g;\delta\sqrt[p]{n})])\le\\
 \sum_{c\in Fn}\sum_{g\in S^n}\IE_\e(O_p(g;\delta\sqrt[p]{n}))\le \sum_{c\in Fn}\sum_{g\in S^n}\IE_{\e,2\delta\!\sqrt[p]{n}}(X^n,d^p_{X^n})\le|Fn|\cdot n^n\cdot \IE_{\e,\delta'}(X^n,d^p_{X^n}).
\end{multline*}
The inequality $\IE_{\e,\delta'}(X^n,d^p_{X^n})\le \big(\IE_{\e',\delta'}(X)\big)^n$ can be verified by analogy with the inequality  $\IE_\e(X^n,d^p_{X^n})\le \big(\IE_{\e}(X)\big)^n$, proved in Theorem~\ref{t:entropy1}.
\end{proof}

The metric entropy functions alow us to define five {\em cardinal entropy characteristics} of a distance space $X$:
$$
\begin{aligned}
&\IE_0(X)=\!\sup_{\e\in(0,\infty)}\IE_\e( X)^+,\quad\IE_{0,\infty}( X)=\!\sup_{\e,\delta\in(0,\infty)}\IE_{\e,\delta}( X)^+,\quad\IE_\infty(X)=\!\min_{\e\in(0,\infty)}\IE_\e(X)^+,\\
&\IE_{0,0}(X)=\!\min_{\delta\in(0,\infty)}\sup_{\e\in(0,\infty)}\IE_{\e,\delta}(X)^+,\quad\IE_{\infty,\infty}(X)=\!\min_{\e\in(0,\infty)}\sup_{\delta\in(0,\infty)}\IE_{\e,\delta}(X)^+.
\end{aligned}
$$In the above definitions  $\kappa^+$stands for the successor of a cardinal $\kappa$.  For example, $\w_1=\w^+$.

It is easy to see that for any distance space $X$ we have the inequalities:
$$\IE_{0,0}(X)\le \IE_{0,\infty}(X)\le\IE_0(X)\ge \IE_\infty(X)\ge \IE_{\infty,\infty}(X).$$ 
 
Many (known) properties of a distance space are equivalent to the countability of suitable cardinal entropy characteristics.

\begin{definition} A distance space $(X,d_X)$ is defined to be 
\begin{itemize}
\item {\em $(0)$-finitary} (or else {\em totally bounded}) if $\IE_0(X)\le\w$ (i.e., for any $\e\in(0,\infty]$ the entropy $\IE_\e(X)$ is finite);
\item {\em $(\infty)$-finitary} if $\IE_\infty(X)\le\w$ (i.e., $X$ has finitely many pseudometric components);
\item {\em $(0,0)$-finitary} (or else {\em Assouad-bounded\/}) if $\IE_{0,0}(X)\le\w$ (i.e., there exists $\delta\in(0,\infty)$ such that for every $\e\in(0,\delta]$ the entropy $\IE_{\e,\delta}(X)$ is finite);
\item {\em $(0,\infty)$-finitary} if $\IE_{0,\infty}(X)\le\w$ (i.e., for any $\e,\delta\in(0,\infty)$ the entropy $\IE_{\e,\delta}(X)$ is finite);
\item {\em $(\infty,\infty)$-finitary} (or else {\em of bounded geometry}) if $\IE_{\infty,\infty}(X)\le\w$ (i.e., there exists $\e\in(0,\infty)$ such that for every $\delta\in(0,\infty)$ the cardinal $\IE_{\e,\delta}(X)$ is finite).
\end{itemize}
\end{definition}
For any distance space these finitarity properties relate as follows:
$$
\xymatrix@C=17pt{
\mbox{$(0,\!0)$-finitary}&\mbox{$(0,\!\infty)$-finitary}\ar@{=>}[l]&\mbox{$(0)$-finitary}\ar@{=>}[l]\ar@{=>}[r]&\mbox{$(\infty)$-finitary}\ar@{=>}[r]&\mbox{$(\infty,\!\infty)$-finitary}\\
\mbox{Assouad}\atop\mbox{bounded}\ar@{=}[u]&&\mbox{totally}\atop\mbox{bounded}\ar@{=}[u]&&\mbox{of bounded}\atop\mbox{geometry}\ar@{=}[u]
}
$$
Assouad bounded metric spaces were studied by Assouad in \cite{Ass}.

Theorems~\ref{t:entropy1}, \ref{t:entropy2} imply the following upper bounds for the  cardinal entropy characteristics of the distance space $F^pX$.

\begin{corollary}\label{c:card-E-char} If the functor $F$ has finite degree, then for $n:=\max\{1,\deg(F)\}$, the distance space $F^pX$ has the following cardinal entropy characteristics.
\begin{enumerate}
\item[\textup{1)}] $\IE_0(F^pX)\le |Fn|\cdot (\IE_0(X))^n$;
\item[\textup{2)}] $\IE_\infty(F^pX)\le |Fn|\cdot (\IE_\infty(X))^n$.
\end{enumerate}
If, in addition, the functor $F$  preserves supports,  then 
\begin{enumerate}
\item[\textup{3)}] $\IE_{0,0}(F^pX)\le n^n\cdot |Fn|\cdot (\IE_{0,0}(X))^n$;
\item[\textup{4)}] $\IE_{0,\infty}(F^pX)\le n^n\cdot |Fn|\cdot (\IE_{0,\infty}(X))^n$;
\item[\textup{5)}] $\IE_{\infty,\infty}(F^pX)\le n^n\cdot |Fn|\cdot (\IE_{\infty,\infty}(X))^n$.
\end{enumerate}
\end{corollary}

Corollary~\ref{c:card-E-char}(1,5) has two important special cases:

\begin{corollary}\label{c:tb} Assume that the functor $F$ is finitary and has finite degree. If a distance space $(X,d_X)$ is totally bounded, then so is the distance space $F^pX$.
\end{corollary}

\begin{corollary}\label{c:bg} Assume that the functor $F$ is finitary, has finite degree and preserves supports. If a distance space $X$ has bounded geometry (or is Assouad-bounded), then so is the distance space $F^pX$.
\end{corollary}

\begin{example} Consider the functor $F:\Set\to\Set$ assigning to each set $X$ the set $FX=\{\emptyset\}\cup \{\{x,y\}:x,y\in X,\;x\ne y\}$ and to each function $f:X\to Y$ the function $Ff:FX\to FY$ such that $Ff(\emptyset)=\emptyset$ and $Ff(\{x,y\})=\{f(x)\}\triangle\{f(y)\}$ for any $\{x,y\}\in FX\setminus\{\emptyset\}$. It is clear that $F$ is finitary and has finite degree $\deg(F)=2$, but $F$ does not preserve supports.  For the space  $X=\IR$ we have  $\IE_{0,0}(X)=\IE_{0,\infty}(X)=\IE_{\infty,\infty}(X)=\w$ and  $\IE_{0,0}(F^pX)=\IE_{0,\infty}(F^pX)=\IE_{\infty,\infty}(F^pX)=\w_1$.
This example shows that the preservation of supports in Corollary~\ref{c:bg} (and also in the statements (3)--(5) of Corollary~\ref{c:card-E-char}) is essential.
\end{example} 

Observe that a distance space $X$ is {\em separable} (= has a countable dense subset) if and only if $\IE_0(X)\le\w_1$. By Corollary~\ref{c:card-E-char}(1), the distance space $F^pX$ is separable if $X$ is separable and the functor $F$ is finitary and has finite degree. In fact, this implication holds in a more general situation.

\begin{proposition} If for every $n\in\IN$ the set $Fn$ is at most countable, then for any separable distance space $X$, the distance space $F^pX$ is separable.
\end{proposition}

\begin{proof} By Proposition~\ref{p:BMZ}, $FX=\bigcup_{n\in\IN}\bigcup_{a\in Fn}\xi^a(X^n)$. For every $n\in\IN$ the separability of the distance space $X$ implies the separability of the power $X^n$ endowed with the distance $d^p_{X^n}$. By Theorem~\ref{t:alt}, for every $a\in Fn$, the map $\xi^a:X^n\to FX$ is non-expanding, which implies that the space $\xi^a[X^n]$ is separable and then $FX$ is separable, being a countable union of separable spaces.
\end{proof}

\section{The functor $F^p$ and the entropy dimensions of distance spaces}

In this section we study the dependence between various entropy dimensions of  distance spaces $X$ and  $F^pX$.

We will consider eight entropy dimensions divided into four pairs. The first pair form two well-known box-counting dimensions, defined for any totally bounded distance space. In our terminology totally bounded spaces are called $(0)$-finitary.

For a totally bounded distance space $X$, its upper and lower box-counting dimensions are defined by the formulas
$$\overline{\dim}_BX=\limsup_{\e\to+0}\frac{\ln(\IE_\e(X))}{\ln(1/\e)}
\mbox{ \ and \ }\underline{\dim}_BX=\liminf_{\e\to+0}\frac{\ln(\IE_\e(X))}{\ln(1/\e)}.
$$
 The box-counting dimensions play an important role in the Theory of Fractals, see \cite[Ch.3]{Fal}.
 
\begin{example} The Cantor ternary set $$C=\Big\{\sum_{n=0}^\infty\frac{x_n}{3^{n}}:(x_n)_{n\in\w}\in\{0,1\}^\w\Big\}\subset \IR$$ is totally bounded and has the box-counting dimensions
$$\overline{\dim}_BC=\underline{\dim}_BC=\frac{\ln2}{\ln3}.$$
\end{example}

The definition of the entropy dimensions and Theorem~\ref{t:entropy1}, imply the following corollary.

\begin{theorem}\label{t:tb-dim} Assume that the functor $F$ is finitary and has finite degree $n=\deg(F)$. If a distance space $X$ is totally bounded, then the distance space $F^pX$ is totally bounded and has upper and lower box-counting dimensions
$$\overline{\dim}_B F^pX\le n\cdot \overline{\dim}_B X\mbox{ \ and \ \ }\underline{\dim}_B F^pX\le n\cdot \underline{\dim}_B X.$$
\end{theorem}

A micro-version of the (upper and lower) box-counting dimensions are  the (upper and lower) $(0,0)$-entropy dimensions defined for any $(0,0)$-finitary distance space $X$ by the formulas
$$
\odim_{(0,0)}X=\lim_{\delta\to0}\limsup_{\e\to 0}\frac{\ln(\IE_{\e,\delta}(X))}{\ln(\delta/\e)}
\mbox{ \ and \ }\udim_{(0,0)}X=\lim_{\delta\to 0}\liminf_{\e\to 0}\frac{\ln(\IE_{\e,\delta}(X))}{\ln(\delta/\e)}.$$

Large-scale counterparts of the (upper and lower) $(0,0)$-entropy dimensions are the (upper and lower) $(\infty,\infty)$-entropy dimensions, defined for any $(\infty,\infty)$-finitary distance space $X$ by the formulas
$$
\overline{\dim}_{(\infty,\infty)}X=\lim_{\e\to\infty}\limsup_{\delta\to\infty}\frac{\ln(\IE_{\e,\delta}(X))}{\ln(\delta/\e)}
\mbox{ \ and \ }\underline{\dim}_{(\infty,\infty)}X=\lim_{\e\to\infty}\liminf_{\delta\to\infty}\frac{\ln(\IE_{\e,\delta}(X))}{\ln(\delta/\e)}.$$

Finally, for any $(0,\infty)$-finitary distance space $X$, let us consider the (upper and lower) $(0,\infty)$-entropy dimensions:
$$
\overline{\dim}_{(0,\infty)}X=\limsup_{(\e,\delta)\to (0,\infty)}\frac{\ln(\IE_{\e,\delta}(X))}{\ln(\delta/\e)}
\mbox{ \ and \ }\underline{\dim}_{(\infty,\infty)}X=\liminf_{(\e,\delta)\to (0,\infty)}\frac{\ln(\IE_{\e,\delta}(X))}{\ln(\delta/\e)}.$$
The upper $(0,\infty)$-entropy dimension $\overline{\dim}_{(0,\infty)}$ was introduced and studied by Assouad in \cite{Ass}.

The relations between the six entropy dimensions are described by the following diagram in which an arrow $a\to b$ indicates that $a\le b$.
$$\xymatrix{
\odim_{(0,0)}(X)\ar[r]&\overline{\lim}_{(0,\infty)}(X)&\odim_{(\infty,\infty)}(X)\ar[l]\\
\udim_{(0,0)}(X)\ar[r]\ar[u]&\udim_{(0,\infty)}(X)\ar[u]&\udim_{(\infty,\infty)}(X)\ar[l]\ar[u]
}$$

\begin{example} The extended Cantor set $$C=\Big\{\sum_{n=m}^{\infty} \frac{x_n}{3^n}:m\in\IZ\;\wedge\;(x_n)_{n\in\IZ}\in\{0,1\}^{\IZ}\Big\}\subset\IR$$ is $(0,\infty)$-finitary and has the entropy dimensions
$$\udim_{(0,0)}C=\odim_{(0,0)}C=\udim_{(0,\infty)}C=\odim_{(0,\infty)}C=\udim_{(\infty,\infty)}=\odim_{(\infty,\infty)}C=\frac{\ln2}{\ln3}.$$
For characterizations of the extended Cantor set in various categories, see \cite{BZ}.
\end{example}

The definition of the entropy dimensions and Theorem~\ref{t:entropy1}, imply the following corollary.

\begin{theorem}\label{t:bg-dim} Assume that the functor $F$ is finitary, prseserves supports and has finite degree $n=\deg(F)$. Let $(\e,\delta)\in\{(0,0),(0,\infty),(\infty,\infty)\}$. If a distance space $X$ is $(\e,\delta)$-finitary, then the distance space $F^pX$ is $(\e,\delta)$-finitary and has upper and lower $(\e,\delta)$-entropy dimensions: 
$$
\odim_{(\e,\delta)}F^pX\le n\cdot\odim_{(\e,\delta)}X, \quad
\udim_{(\e,\delta)}F^pX\le n\cdot\udim_{(\e,\delta)}X.
$$
\end{theorem}

\section{The functor $F^p$ and the Hausdorff dimension of distance spaces}

In this section we study the dependence between the Hausdorff dimensions of the distance spaces $X$ and $F^pX$. 

The Hausdorff dimension of distance spaces is defined by analogy with the Hausdorff dimension of metric spaces, see \cite[Ch.2]{Fal}.

Let $X$ be a nonempty distance space. For any $\e\in(0,\infty]$, let $\cov_\e(X)$ be the family of all countable covers of $X$ by sets of diameter $\le\e$. For a cover $\U\in\cov_\e(X)$ and a positive real number $s$, let $$\mathcal H^d(\U)=\sum_{U\in\U}\diam(U)^s:=\sup_{\F\in[\U]^{<\w}}\sum_{F\in\F}\diam(F)^s\in[0,\infty],$$
and
$$\mathcal H^s_\e(X)=\inf\big(\{\mathcal H^d(\U):\U\in\cov_\e(X)\}\cup\{\infty\}\big).$$
It is clear that $\mathcal H^s_{\e'}(X)\le\mathcal H^s_{\e}(X)$ for any $\e\le \e'$. Consequently, there exists the (finite or infinite) limit
$$\mathcal H^s(X)=\lim_{\e\to+0}\mathcal H^s_\e(X)=\sup_{\e\in(0,\infty)}\mathcal H^s_\e(X),$$
called the {\em $s$-dimensional Hausdorff measure} of the distance space $X$.
It can be shown that if $\mathcal H^s(X)<\infty$ for some $s\in(0,\infty)$, then $\mathcal H^t(X)=0$ for every $t>s$.

The (finite or infinite) number $$\dim_HX=\inf\big(\{s\in (0,\infty):\mathcal H^s(X)=0\}\cup\{\infty\}\big)$$
is called {\em the Hausdorff dimension} of the distance space $X$.

The following properties of the Hausdorff dimension of distances spaces can be proved by analogy with the corresponding properties of the Hausdorff dimension for metric spaces, see \cite[\S2.2]{Fal}.

\begin{proposition}\label{p:HD} Let $X$ be a distance space.
\begin{enumerate} 
\item[\textup{1)}] Each subspace $A\subseteq X$ has $\dim_HA\le \dim_HX$.
\item[\textup{2)}] For any countable cover $\{X_i\}_{i\in\w}$ of $X$ we have $\dim_HX=\sup_{i\in\w}\dim_HX_i$.
\item[\textup{3)}] For any surjective Lipschitz map $f:X\to Y$ to a distance space $Y$ we have $\dim_H Y\le\dim_HX$.
\end{enumerate}
\end{proposition}

\begin{theorem}\label{t:H-dim} Assume that the functor $F$ has finite degree $n=\deg(F)$ and the set $Fn$ is at most countable. Then $$\dim_H F^pX\le \dim_H X^n,$$
where $X^n$ is endowed with the distance $d^p_{X^n}$.
\end{theorem}

\begin{proof} The inequality is trivial if $\dim_H X^n=\infty$ or $\dim_H F^pX=0$ . So, we assume that $\dim_H F^pX>0$ and  $\dim_H(X^n)<\infty$. If $n=\deg(F)=0$, then we can take any map $f:1\to X$ and applying Theorem~\ref{t:small}, conclude that  $FX=Ff[F1]$ and hence $|F^pX|\le |F1|\le |Fn|\le\w$ and $\dim_H F^pX=0$ by Proposition~\ref{p:HD}(1). But this contradicts our assumption. Therefore, $n=\deg(F)>0$. In this case, by Proposition~\ref{p:BMZ}, $F^pX=\bigcup_{a\in Fn}\xi^a[X^n]$. By Theorem~\ref{t:alt}, the function $\xi^a:X^n\to F^pX$ is non-expanding and hence 
$$\dim_H F^pX=\sup_{a\in Fn}\dim_H\xi^a[X^n]\le\sup_{a\in Fn}\dim_H X^n=\dim_H X^n.$$
\end{proof}

\begin{remark} The Hausdorff dimension does not obey the logarithmic law. By Example 7.8 in \cite{Fal}, the real line contains two Borel subsets $Y,Z$ such that $\dim_H Y=\dim_H Z=0$ but $\dim_H(Y\times Z)\ge 1$. On the other hand, by Product Formulas 7.2 and 7.3 in \cite{Fal}, for any Borel subsets $Y,Z$ in a Euclidean space $\IR^n$ we have
$$\dim_H(Y\times Z)\ge\dim_HY+\dim_HZ\quad\mbox{and}\quad\dim_H(Y\times Z)\le\dim_H Y+ \overline{\dim}_B Z.$$
\end{remark}

\section{The functor $F^p$ and the controlled dimensions of distance spaces}

In this section we study the interplay between the controlled dimensions of the distance spaces $X$ and $F^pX$. The controlled dimensions of distance spaces are generalizations of the covering dimension of Lebesgue and asymptotic dimension of Gromov. They are introduced as follows.

Let $\mathcal D$ be a nonempty family of non-decreasing functions from $(0,\infty)$ to $(0,\infty]$. The family $\mathcal D$ will be called a {\em dimension controlling class} and its elements {\em dimension controlling functions}.
 
Let $D:(0,\infty)\to(0,\infty]$ be a dimension controlling function and $n$ be a cardinal. We say that a distance space $(X,d_X)$ has {\em $D$-controlled dimension} $\Dim[D] X\le n$ if for every $\e\in(0,+\infty)$ there exists a cover $\mathcal U=\bigcup_{i\in 1+n}\U_i$ of $X$ such that $$\mesh(\U)\le D(\e)\mbox{ \ and \ }\inf_{i\in 1+n}\ud_X(\U_i)\ge \e,$$where $$\ud_X(\U_i)=\inf\big(\{\infty\}\cup\{\ud_X(U,V):U,V\in\U_i,\;\;U\ne V\}\big).$$ 
The smallest cardinal $n$ such that $\Dim[D] X\le n$ is called the {\em $D$-dimension} of the distance space $X$ and is denoted by $\Dim[D]X$.

\begin{example}\label{ex:zero} If the distance $d_X$ is a $\{0,\infty\}$-valued $\infty$-metric, then $\dim_DX=0$ for every dimension controlling function $D$.
\end{example}

\begin{proof} The cover $\U_0=\{\{x\}:x\in X\}$ of $X$ by singletons has $\mesh(\U_0)=0\le D(\e)$ and $\ud_X(\U_0)=\infty>\e$ for every $\e\in(0,\infty)$, witnessing that $\dim_DX=0$.
\end{proof} 

Given a  dimension controlling class $\mathcal D$, the cardinal
$$\Dim[\mathcal D]X:=\min_{D\in\mathcal D}\dim_D X$$ 
is called the {\em $\mathcal D$-controlled dimension} of the distance space $X$.

\begin{definition} A dimension controlling class $\mathcal D$ is called a {\em dimension controlling scale} if it satisfies the following conditions:
\begin{itemize}
\item the identity function $(0,\infty)\to(0,\infty)$ belongs to $\mathcal D$;
\item for any functions $f,g\in\mathcal D$ the functions $f+g$ and $f\circ g$ belong to $\mathcal D$;
\item for any function $f\in\mathcal D$, any non-decreasing function $g\le f$ belongs to $\mathcal D$.
\end{itemize}
\end{definition}
Each dimension controlling class generates a unique dimension controlling scale (equal to the intersection of all dimension controlling scales that contain this class).

The main result of this section is the following theorem generalizing topological results of Basmanov \cite{Bas1,Bas2,Bas3} and ``coarse'' results of Radul, Shukel, Zarichnyi  \cite{RS, Shukel,SZ}.

\begin{theorem}\label{t:main-dim} If a functor $F:\Set\to\Set$ is finitary and has finite degree, then for any dimension controlling scale $\mathcal D$ and any distance space $X$, the distance space $F^pX=(FX,d^p_{FX})$ has $\mathcal D$-controlled dimension $$\dim_{\mathcal D}F^pX\le \deg(F)\cdot\dim_{\mathcal D} X.$$
\end{theorem} 

This theorem follows from Theorems~\ref{t:main-dim1} and \ref{t:main-dim2}, treating the cases of finite and infinite dimension $\dim_DX$, respectively.
The proofs of Theorems~\ref{t:main-dim1} and \ref{t:main-dim2} are preceded by  some  lemmas.

For two families of subsets $\V,\W$ of a distance space $(X,d_X)$, a positive real number $\e$,  and a subset $W\subseteq X$, let 
$$W\cup_\e\V:=W\cup\textstyle{\bigcup}\{V\in \V:\ud_X(V,W)<\e\}$$and
let $$\W\cup_\e\V=\{W\cup_\e\V:W\in\W\}\cup\{V\in\V:\ud_X(V,\textstyle{\bigcup}\W)\ge\e\}.$$

The following lemma easily follows from the definition of the family $\W\cup_\e\V$.

\begin{lemma}\label{l:sume} Let $\e$ be a positive real number and $\V,\W$ be two families of sets in a distance space $(X,d_X)$. If $\ud_X(\V)\ge\e$ and $\ud_X(\W)\ge 2\e+\mesh(\V)$, then 
$$\ud_X(\W\cup_\e\V)\ge\e\mbox{ \ and \ }\mesh(\W\cup_\e \V)\le 2\e+2\cdot \mesh(\V)+\mesh(\W).$$ 
\end{lemma}





The following lemma is crucial for our proof. It is a slightly modified version of Lemma 31 of Bell and Dranishnikov \cite{BellDran} (who call it Kolmogorov's trick).



\begin{lemma}\label{l:Kolmogorov} Let $k\le m$ be two nonzero finite cardinals, $\e$ be a positive real number, and $\U=\bigcup_{i\in m}\U_i$ be a cover of a distance space $X$ such that $\inf_{i\in m}\ud_X(\U_i)\ge \e$, and $\bigcup_{i\in S}\bigcup\U_i=X$ for every  $S\in [m]^k$. Then there exists a cover $\V=\bigcup_{i\in m+1}\V_i$ of $X$ such that 
\begin{enumerate}
\item[\textup{1)}] $\inf\limits_{i\in m+1}\ud_X(\V_i)\ge\frac13\e$;
\item[\textup{2)}] $\mesh(\V)\le \mesh(\U)+\frac23\e$;
\item[\textup{3)}] $X=\bigcup_{i\in S}\bigcup\V_i$ for every $S\in [m+1]^k$.
\end{enumerate}
\end{lemma}

\begin{proof}  For every $i\in m$ and $U\in\U_i$, consider the family $\V_i=\{O[U;\frac13\e):U\in\U_i\}$ of $\frac13\e$-neighborhoods of the sets $U\in\U_i$ in the distance space $(X,d_X)$. Also let $$\V_{m}=\big\{\bigcap_{i\in T}U_i\setminus\bigcup_{j\in m\setminus T}{\textstyle\bigcup}\V_i:T\in[m]^{m-k+1}\;\wedge\;(U_i)_{i\in T}\in \textstyle\prod_{i\in T}\U_i\big\}.$$
We claim that the cover $\V=\bigcup_{i\in m+1}\V_i$ has the properties (1)--(3). The properties (1),(2) will follow from the definitions of the families $\V_i$ as soon as we check that $\ud_X(\V_m)\ge \frac13\e$. Assuming that $\ud_X(\V_m)<\frac13\e$, we could find two distinct sets $V,V'\in\V_m$ and points $v\in V$, $v'\in V'$ such that $d_X(v,v')<\frac13\e$. Find two sets $T,T'\in[m]^{m-k+1}$ and families of sets $(U_i)_{i\in T}\in\prod_{i\in T}\U_i$, $(U')_{i\in T'}\in\prod_{i\in T'}\U_i$ such that $V=\bigcap_{i\in T}U_i\setminus \bigcup_{i\in m\setminus T}\bigcup\V_i$ and $V'=\bigcap_{i\in T}U'_i\setminus \bigcup_{i\in m\setminus T'}\bigcup\V_i$. If $T=T'$, then there exists $i\in T$ such that $U_i\ne U_i'$. Then 
$$\tfrac13\e>d_X(v,v')\ge\ud_X(U_i,U'_i)\ge \e,$$
which is a contradiction showing that $T\ne T'$. Then there exists $i\in T\setminus T'$. The inequality $d_X(v',v)<\frac13\e$ implies $v'\in O(v;\frac13\e)\subseteq \bigcup \V_i$ and hence $v'\notin V'\subseteq X\setminus \bigcup\V_i$. This contradiction shows that the conditions (1), (2) are satisfied.
\smallskip

To verify the condition (3), assume that for some set $S\in [m+1]^k$, the family $\bigcup_{i\in S}\bigcup\V_i$ does not contain some point $x\in X$. Then also $x\notin\bigcup_{i\in S\cap m}\bigcup\V_i$. The choice of the family $\U$ guarantees that $|S\cap m|<k$ and hence $m\in S$ and $S\cap m\in [m]^{k-1}$. Consider the set $T:=m\setminus S\in[m]^{m-k+1}$ and observe that for every $i\in T$ the set $(S\cap m)\cup\{i\}$ has cardinality $k$ and hence $x\in U_i$ for some $U_i\in\U_i$. Then $x\in \bigcap_{i\in T}U_i\setminus \bigcup_{i\in m\setminus T}\bigcup\V_i\subseteq\bigcup\V_m\subseteq\bigcup_{i\in S}\bigcup\V_i$, which contradicts the choice of $x$ and completes the proof of the lemma.
\end{proof}

\begin{lemma}\label{l:BellDra} Let $n,m$ be two nonzero finite cardinals, $\e$ be a positive real number, and $\U=\bigcup_{i\in m+1}\U_i$ be a cover of a distance space $X$ such that $\inf_{i\in m+1}\ud_X(\U_i)\ge 3^{nm-m+1}\e$. Then there exists a cover $\V=\bigcup\limits_{i\in nm+1}\V_i$ of $X$ such that 
\begin{enumerate}
\item[\textup{1)}] $\inf_{i\in nm+1}\ud_X(\V_i)\ge\e$;
\item[\textup{2)}] $\mesh(\V)\le \mesh(\U)+3^{nm-m+1}\e$;
\item[\textup{3)}] for any $T\in[X]^n$ there exists $i\in nm+1$ such that $T\subseteq \bigcup\V_i$.
\end{enumerate}
\end{lemma}

\begin{proof} Let $\delta=3^{nm-m+1}\e$ and $\W_1=\bigcup_{i\in m+1}\W_{1,i}$ where $\W_{1,i}=\U_i$ for $i\in m+1$.\break Observe that for the unique set $S\in[m+1]^{1+m}$ we have $\bigcup_{i\in S}\bigcup\W_{1,i}=X$ and\break $\inf_{i\in m+1}\ud_X(\W_{1,i})\ge\delta$. Applying Lemma~\ref{l:Kolmogorov} inductively, for every $k\in\IN$ construct a cover $\W_k=\bigcup_{i\in m+k}\W_{k,i}$ satisfying the following conditions:
\begin{itemize}
\item $\inf\limits_{i\in m+k}\ud_X(\W_{k,i})\ge\frac1{3^k}\delta$;
\item $\mesh(\W_k)\le\mesh(\U)+\sum_{i=1}^k\big(\frac23\big)^i\delta\le \mesh(\U)+\delta$;
\item $X=\bigcup_{i\in T}\bigcup\W_{k,i}$ for every $T\in [m+k]^{1+m}$.
\end{itemize}
Then the cover $\V=\bigcup_{i\in nm+1}\V_i$, where $\V_i=\W_{nm-m+1,i}$ for $i\in nm+1$, has the properties:
\begin{itemize}
\item $\inf\limits_{i\in nm+1}\ud_X(\V)\ge\frac1{3^{nm-m+1}}\delta=\e$;
\item $\mesh(\V)\le\mesh(\U)+\sum_{i=1}^{nm-m+1}\big(\frac23\big)^k\delta\le\mesh(\U)+\delta=\mesh(\U)+3^{nm-m+1}\e$;
\item $X=\bigcup_{i\in S}\bigcup\V_i$ for every $S\in [nm+1]^{1+m}$.
\end{itemize}
The last property implies that for every $x\in X$, the set  $S_x=\{i\in nm+1:x\notin \bigcup\V_i\}$ has cardinality $|S_x|\le m$. Then for every set $T\in [X]^n$, there exists an index $i\in (nm+1)\setminus\bigcup_{x\in T}S_x$. For this index we have $T\subseteq\bigcup\V_i$, which means that the condition (3) of the lemma is satisfied.
\end{proof}

Given a finite cardinal $k$, consider the subspace
$$F^p_kX=\{a\in F^p X:|\supp(a)|\le k\}$$of the distance space $F^pX=(FX,d^p_{FX})$.

\begin{lemma}\label{l:dim1} Let $p\in[1,\infty]$, $n\in\IN$, and $D,D'$ be dimension controlling functions such that $D(n^{1-\frac1p}\e)\le D'(\e)$ for every $\e\in(0,\infty)$. For every distance space $X$ we have 
\begin{enumerate}
\item[\textup{(1)}] $\dim_D(F^p_1X)\le\dim_D(X)$ if $p=1$ or $F$ preserves supports;
\item[\textup{(2)}] $\dim_{D'}(F^p_1X)\le \dim_D(X)$ if $F$ has finite degree $\deg(F)\le n$.
\end{enumerate}
\end{lemma}

\begin{proof} If $|X|\le 1$, then $\dim_DX=0$ by Example~\ref{ex:zero}. By Theorem~\ref{t:small}, the distance $d^p_{FX}$ is a $\{0,\infty\}$-valued $\infty$-metric and by Example~\ref{ex:zero}, $$\dim_{D'}F^pX=\dim_D F^pX=0=\dim_DX.$$

So, we assume that $|X|\ge 2$. Let $$
c=\begin{cases}1&\mbox{if $F$ preserves supports};\\
n^{1-\frac1p}&\mbox{otherwise}.
\end{cases}
$$
By the definition of the cardinal $\kappa=\dim_D X$ for every $\e\in(0,\infty)$ there exists a cover $\U=\bigcup_{i\in 1+\kappa}\U_i$ such that $\mesh(\U)\le D(c\e)$ and $\ud_X(\U_i)\ge c\e$ for every $i\in 1+\kappa$.

 For every $x\in 2$ let $\bar x_2:1\to\{x\}\subseteq 2$ be the constant map. In the set $F1$ consider the subsets $F_\bullet 1=\{a\in F1:F\bar 0_2(a)\ne F\bar 1_2(a)\}$ and $F_\circ 1=F1\setminus F_\bullet 1$. For every $a\in F_\bullet 1$ consider the function $\xi^a_X:X\to FX$, $x\mapsto F\bar x_X(a)$, where $\bar x_X:1\to\{x\}\subseteq X$ is the constant map. 

\begin{claim}\label{cl:cover-new} $F_1X=F_0X\cup\bigcup\limits_{a\in F_\bullet 1}\xi^a_X[X]$.
\end{claim}

\begin{proof} Given any $b\in F_1X\setminus F_0X$, consider its support $S=\supp(a)$, which is a singleton and hence admits a bijective function $h:1\to S$. By Proposition~\ref{p:BMZ}, $b\in F[S;X]=Fi_{S,X}[FS]=Fi_{S,X}\circ Fh[F1]$ and hence $b=Fi_{S,X}\circ Fh(a)$ for some $a\in F1$. Then for the element $x=i_{S,X}\circ h(0)$ we obtain $b=\xi^a_X(x)$. Assuming that $a\notin F_\bullet 1$, we can apply Lemma~\ref{l:delta}(1b) and conclude that $\supp(b)=\supp(\xi^a_X(x))=\emptyset$, which contradicts the choice of $b$. Therefore, $a\in F_\bullet 1$ and $b=\xi^a_X(x)\in\xi^a_X[X]$.
\end{proof}

Let $$\V_0=\big\{\{b\}:b\in F_0X\big\}\cup\big\{\xi^a_X[U]:a\in F_\bullet 1,\;\;U\in\U_0\big\}$$and for any nonzero ordinal $i\in 1+\kappa$ let $$\V_i=\{\xi^a_X[U]:a\in F_\bullet 1,\;\;U\in\U_i\}.$$
Claim~\ref{cl:cover-new} implies that $\V=\bigcup_{i\in 1+\kappa}\V_i$ is a cover of $F_1X$. Since for every $a\in F1$ the map $\xi^a_X:X\to F^pX$ is not expanding, $\mesh(\V)\le\mesh(\U)\le D(c\e)\le D(n^{1-\frac1p}\e)\le D'(\e)$. It remains to prove that $\ud^p_{FX}(\V_i)\ge\e$ for every $i\in 1+\kappa$.

Assuming that $\ud^p_{FX}(\V_i)<\e$ for some $i\in 1+\kappa$, find two distinct sets $V,V'\in\V_i$ such that $\ud^p_{FX}(V,V')<\e$.

Consider the constant function $\gamma:X\to 1$ and observe that for every $x\in X$ and the constant function $\bar x:1\to\{x\}\subseteq X$, the composition $\gamma\circ\bar x$ is the identity map of $1$. Then $F\gamma\circ F\bar x$ is the identity map of $F1$ and hence for every $a\in F1$ we have
$$a=F\gamma\circ F\bar x(a)=F\gamma(\xi^a_X(x)).$$ 
Then the sets $F\gamma[V']$ and $F\gamma[V']$ are singletons.
Endow the singleton $1$ with a unique $\{0\}$-valued metric and applying Theorem~\ref{t:small}, conclude that the induced distance on $F^p1$ is an $\{0,\infty\}$-valued $\infty$-metric.  
By Theorem~\ref{t:Lip}, the map $F\gamma:F^pX\to F^p1$ is Lipschitz, which implies that  $F\gamma[V]=F\gamma[V']=\{a\}$ for some $a\in F1$. If $a\in F_\bullet 1$, then $V=\xi^a_X[U]$ and $V'=\xi^a_X[U']$ for some distinct sets $U,U'\in \U_i$. 

If $p=1$ or the functor $F$ preserves supports, then by Theorem~\ref{t:xi}, the function $\xi^a:X\to F^pX$ is a injective isometry. In this case
$$\ud^p_{FX}(V,V')=\ud^p_{FX}(\xi^a[U],\xi^a[U'])=\ud_X(U,U')\ge \ud(\U_i)\ge c\e=\e,$$which contradicts the choice of the sets $V,V'$.

If $F$ does not preserve supports but has degree $\deg(F)\le n$, then by  Theorem~\ref{t:ineq}, 
$$\ud^p_{FX}(V,V')\ge n^{\frac1p-1}\ud^1_{FX}(V,V')=c^{-1}\ud^1_{FX}(\xi^a_X[U],\xi^a_X[U'])=c^{-1}\ud_X(U,U')\ge c^{-1}c\e=\e,$$
which contradicts the choice of $V,V'$. This contradiction shows that $a\in F_\circ 1$. In this case $i=0$ and $V,V'$ are distinct singletons in the set $F_0X$. By Theorem~\ref{t:small}, $\ud^p_{FX}(V,V')=\infty\ge\e$, which contradicts the choice of $V,V'$. So, in all cases we arrive to a contradiction, which proves that $\ud^p_{FX}(\V_i)\ge\e$ and finally that $\dim_{D'}(F^p_1X)\le \kappa=\dim_DX$. If $p=1$ or $F$ preserves supports, then $\mesh(\V)\le D(c\e)=D(\e)$, witnessing that $\dim_D(F^pX)\le\kappa=\dim_D(X)$.

\end{proof}

\begin{lemma}\label{l:indstep} Assume that the functor $F$ is finitary and has finite degree. Let $ n:=\max\{1,\deg(F)\}$ and $k,m$ be finite cardinals such that $2\le k\le n$. Let $D,D',D''$ be dimension controlling functions such that for any $\e\in(0,\infty)$ and the numbers $$\e'=2n\cdot|Fn|\cdot\e\mbox{ \ and \ }\delta=D(3^{km-m+1}\e')+3^{km-m+1}\e',$$ we have $$D''(\e)\ge 2\e+2k^{\frac1p}\delta+ 2\e'+ D'(2\e+2\e'+k^{\frac1p}\delta).$$
Then for any distance space $X$ with $\dim_DX\le m$, the distance space $F^p_kX$ has $D''$-controlled dimension $$\dim_{D''}(F^p_kX)\le\max\{\dim_{D'}(F^p_{k-1}X),km\}.$$
\end{lemma}

\begin{proof} 
By definition of the cardinal $m'=\dim_{D'}(F^p_{k-1}X)$, there exists  a cover $\V=\bigcup_{i\in 1+m'}\V_i$ of $F^p_{k-1}X$ such that $$\inf_{i\in 1+m'}\ud_{F^p_{k-1}X}(\V_i)\ge 2\e+2\e'+k^{\frac1p}\delta\mbox{ \  and \ }\mesh(\V)\le D'(2\e+2\e'+k^{\frac1p}\delta).$$ 
Let $\widetilde{\V}=\bigcup_{i\in 1+m'}\widetilde\V_i$ where $\widetilde \V_i=\{O[V;\e'):V\in\V_i\}$ for $i\in 1+m'$. 

Consider the subspace $W=F^p_kX\setminus\bigcup\widetilde\V$.

\begin{claim}\label{cl:uL} For every $a\in W$ we have $|\supp(a)|=k$ and $\ud_X(\supp(a))\ge \e'$.
\end{claim}

\begin{proof} The equality $|\supp(a)|=k$ follows from the inclusion $W\subseteq F^p_kX\setminus F^p_{k-1}X$. Assuming that $\ud_X(\supp(a))<\e'$, find two distinct points $x,y\in\supp(a)$ with $d_X(x,y)<\e'$.
Using Proposition~\ref{p:BMZ2}, find a function $f:k\to X$ and an element $b\in Fk$ such that $f[k]=\supp(a)$ and $Ff(b)=a$. Let $j\in k$ be the unique number such that $f(j)=x$, and $g:k\to X$ be the function such that $g(j)=y$ and $g(i)=f(i)$ for every $i\in k\setminus\{j\}$. Then the element $c=Fg(b)$ has support $\supp(c)\subseteq g[k]=\supp(a)\setminus\{x\}$ and hence $c\in F_{k-1}X$. Then the inequality $$d^p_{FX}(a,c)\le d^p_{X^k}(f,g)=d_X(f(i),g(i))=d_X(x,y)<\e'$$implies that $a\in\bigcup\widetilde\V$, which contradicts the inclusion $a\in W$.
\end{proof}

Since  $\dim_DX\le m$, we can apply Lemma~\ref{l:BellDra} and find a  cover $\U=\bigcup_{i\in km+1}\U_{i}$ of $X$ such that $\mesh(\U)\le D(3^{km-m+1}\e')+3^{km-m+1}\e'=\delta$, $\inf_{i\in km+1}\ud_X(\U_i)\ge\e'$ and each set $T\in[X]^k$ is contained in  $\bigcup\U_i$ for some $i\in km+1$. 
For every $i\in km+1$, consider the sets $X_i=\bigcup\U_i\subseteq X$ and
$$W_i=\{a\in W:\supp(a)\subseteq X_i\}\subseteq F^p_kX.$$

For every $a\in W_i$ let $\W_i(a)$ be the set of all elements $b\in W_i$ for which there exists a sequence of elements $a_0,\dots,a_l\in W_i$ such that $a=a_0$, $a_l=b$ and $d^p_{FX}(a_{j-1}a_j)<\e$ for every $j\in\{1,\dots,l\}$. Let $$\W_i:=\{\W_i(a):a\in W_i\}.$$ It is clear that $\W_i$ is a cover of the set $W_i$ and $\ud_{F^p\!X}(\W_i)\ge\e$. 



\begin{claim}\label{cl:mesh} $\mesh (\W_i)\le k^{\frac1p}\cdot \mesh(\U_i)\le k^{\frac1p}\cdot \delta$.
\end{claim}

\begin{proof} It suffices two show that for any  distinct elements $a,b\in W_i$ with $\W_i(a)=\W_i(b)$, we have $d^p_{FX}(a,b)\le k^{\frac1p}\cdot\mesh(\U_i)$. Since $\W_i(a)=\W_i(b)$, there exists a sequence of pairwise distinct elements $b_0,\dots,b_l\in W_i$ such that $a=b_0$, $b_l=b$, and $d^p_{FX}(b_{j-1},b_j)<\e$ for every $j\in\{1,\dots,l\}$. For every $j\in\{0,\dots,l\}$ consider the $k$-element set $S_j=\supp(b_j)$ and using Proposition~\ref{p:BMZ}, find an element $b_j'\in FS_j$ such that $b_j=F_{S_j,X}(b_j')$.
By Lemma~\ref{l:2short}, for every $j\in\{1,\dots,l\}$ there exists a function $r_j:X\to X$ such that $Fr_j(b_{j-1})=Fr_j(b_j)$, $r_j\circ r_j=r_j$, $r_j[S_{j-1}\cup S_j]\subseteq S_{j-1}\cup S_j$,  and $$\sup_{x\in X}d_X(x,r_j(x))<n\cdot|Fn|\cdot \e=\tfrac12\e'.$$ The last inequality and Claim~\ref{cl:uL} imply that the restrictions $r_j{\restriction}_{S_{j-1}}$ and $r_{j}{\restriction}_{S_j}$ are injective maps. Then $|r_j[X]|\ge |r_j[S_j]|=|S_j|=k>1$. By Corollary~\ref{c:supp},  
$$r_j[S_{j-1}]=r_j[\supp(b_{j-1})]=\supp(Fr_j(b_{j-1}))=\supp(Fr_j(b_j))=r_j[\supp(b_j)]=r_j[S_j],$$
and we can define a bijective function $h_j:S_{j-1}\to S_j$ assigning to each point $x\in S_{j-1}$ a unique point $y\in S_j$ such that $r_j(x)=r_j(y)$.
Observe that
$$d_X(x,h_j(x))\le d_X(x,r_j(x))+d_X(r_j(y),y))<\tfrac12\e'+\tfrac12\e=\e'.$$

\begin{claim}\label{cl:heq} $Fh_j(b_{j-1}')=b_j'$.
\end{claim}

\begin{proof} The definition of the map $h_j$ ensures that $r_j\circ i_{S_j,X}\circ h_j=r_j\circ i_{S_{j-1},X}$ and hence 
$$Fr_j\circ Fi_{S_j,X}\circ Fh_j(b'_{j-1})=Fr_j\circ Fi_{S_{j-1},X}(b_{j-1}')=Fr_j(b_{j-1})=Fr_j(b_j)=Fr_j\circ i_{S_j,X}(b_j').$$Since the function $r_j\circ i_{S_j,X}:S_j\to X$ is injective, there exists a function $\gamma:X\to S_j$ such that $\gamma\circ r_j\circ i_{S_j,X}$ is the identity function of $S_j$. Then $F\gamma\circ Fr_j\circ Fi_{S_j,X}$ is the identity function of $FS_j\ni b_j'$ and hence
$$b_j'=F\gamma\circ Fr_j\circ Fi_{S_j,X}(b_j')=F\gamma\circ Fr_j\circ Fi_{S_j,X}\circ Fh_j(b'_{j-1})=Fh_j(b'_{j-1}).$$
\end{proof}



By Claim~\ref{cl:heq}, for the bijective map $h=h_l\circ\dots\circ h_1:S_0\to S_l$ we have $Fh(b_0')=b'_l$.

For every $j\in\{0,\dots,l\}$ consider the map $u_j:S_j\to \U_i$ assigning to each point $x\in S_j$ a unique set $u_j(x)\in\U_i$ containing $x$. Taking into account that $\ud_X(\U_i)\ge\e'>d_X(x,h_j(x))$, we conclude that $u_{j-1}=u_j\circ h_j$ for every $j\in\{1,\dots,l\}$. Consequently, 
$$u_0=u_1\circ h_1=u_2\circ h_2\circ h_1=\dots=u_l\circ h_l\circ\cdots\circ h_1=u_l\circ h$$
and for every $x\in S_0=\supp(a)$, 
$$d_X(x,h(x))\le\diam \big(u_0(x)\big)\le\mesh(\U_i).$$
Choose any bijective function $f:k\to S_0$ and find an element $a'\in Fk$ such that $Ff(a')=b'_0$. Then $a=b_0=Fi_{S_0,X}(b_0')=Fi_{S_0,X}\circ Ff(a')$ and 
$$b=b_l=Fi_{S_l,X}(b'_l)=Fi_{S_l,X}\circ Fh(b'_0)=Fi_{S_l,X}\circ Fh\circ Ff(a').$$ Now the definition of the metric $d^p_{FX}$ ensures that
$$d^p_{FX}(a,b)\le d^p_{X^k}(i_{S_0,X}\circ f,i_{S_k,X}\circ h\circ f)\le k^{\frac1p}\cdot\mesh(\U_i)\le k^{\frac1p}\delta.$$
This completes the proof of Claim~\ref{cl:mesh}.
\end{proof}

\begin{claim}\label{cl:meshsh} $\W=\bigcup_{i\in km+1}\W_i$ is a cover of the set $W$ such that  $$
\begin{aligned}
&\inf_{i\in km+1}\ud_{F^p\!X}(\W_i)\ge \e,\\&\mesh(\W)\le k^{\frac1p}\cdot\mesh(\!\U)\le k^{\frac1p}\cdot\delta\mbox{ \ and \ }\\
&2\e+\mesh(\W)\le \inf_{i\in 1+m'}\ud_{F^p\!X}(\widetilde \V_i).
\end{aligned}
$$
\end{claim}

\begin{proof} The inequality $\e\le\inf_{i\in km+1}\ud_{F^p\!X}(\W_i)$ follows from the definition of the families $\W_i$; the upper bound $\mesh(\W)\le k^{\frac1p}\cdot\mesh(\U)\le k^{\frac1p}\cdot\delta$  is proved in Claim~\ref{cl:mesh}.

The choice of the covers $\V$ and $\widetilde \V$ guaranees that for every $i\in 1+m'$ we have
$$\ud_{F^p\!X}(\widetilde\V_i)\ge\ud_{F^p\!X}(\V_i)-2\e'\ge 2\e+2\e'+k^{\frac1p}\delta-2\e'=2\e+k^{\frac1p}\delta\ge 2\e+\mesh(\W).$$
\end{proof}

Let $m''=\max\{m',km\}$ and choose any surjective maps $\alpha:1+m''\to 1+m'$ and $\beta:1+m''\to km+1$. For every $i\in 1+m''$, consider the family $\widetilde\W_i=\widetilde \V_{\alpha(i)}\cup_\e \W_{\beta(i)}$, and the cover $\widetilde\W=\bigcup_{i\in 1+m''}\widetilde\W_i$ of the distance space $F^p_kX$. Claim~\ref{cl:meshsh} and Lemma~\ref{l:sume} imply that $\inf_{i\in 1+m''}\ud_{F^pX}(\widetilde\W_i)\ge\e$ and
\begin{multline*}
\mesh(\widetilde\W)\le 2\e+2\cdot\mesh(\W)+\mesh(\widetilde\V)\le 
2\e+2\cdot\mesh(\W)+2\e'+\mesh(\V)\le\\
2\e+2k^{\frac1p}\delta+2\e'+D'(2\e+2\e'+k^{\frac1p}\delta)\le D''(\e),
\end{multline*}
witnessing that $$\dim_{D''} F^p_kX\le m''=\max\{\dim_{D'}F^p_{k-1}X,km\}.$$This completes the proof of Lemma~\ref{l:indstep}.
\end{proof}

\begin{theorem}\label{t:main-dim1} Assume that the functor $F$ is finitary and has finite degree. Let $n=\max\{1,\deg(F)\}$. Assume that for some dimension controlling function $D$, the distance space $X$ has finite $D$-controlled dimension $m=\dim_DX$. Consider the sequence of dimension controlling functions $(D_k)_{k\in\w}$ defined by the recursive formula: $D_0(\e)=0$, $D_1(\e)=D(n^{1-\frac1p}\e)$, and
$$
\begin{aligned}
D_{k}(\e)=\;&(2+4n|Fn|)\e+2k^{\frac1p}\big(D(2n\cdot|Fn|\cdot 3^{km-m+1}\e)+2n\cdot|Fn|\cdot 3^{km-m+1}\e\big)+\\
&D_{k-1}\big(2\e(1+n\cdot |Fn|)+k^{\frac1p}(D(2n{\cdot}|Fn|{\cdot} 3^{km-m+1}\e)+2n{\cdot}|Fn|{\cdot} 3^{km-m+1}\e)\big)
\end{aligned}$$
for every $k\ge 2$ and $\e\in(0,\infty)$. Then $\dim_{D_k}F^p_kX\le k\cdot\dim_DX$ for every $k\le \deg(F)$. In particular, $\dim_{D_n}F^pX\le \deg(F)\cdot\dim_D X$.
\end{theorem}

\begin{proof} By Theorem~\ref{t:small}, for any distinct points $a,b\in F_0X$ we have $d^p_{FX}(a,b)=\infty$, which implies that $\dim_{D_0}F^p_0X=0\le 0\cdot\dim_DX$.

 Now assume that for positive number $k\le \deg(F)$ we proved that $\dim_{D_{k-1}}F_{k-1}X\le (k-1)\cdot \dim_{D_{k-1}}X$. If $k=1$, then $\dim_{D_1}F^p_1X\le \dim_DX$ according to Lemma~\ref{l:dim1}(2). If $k\ge 2$, the we apply Lemma~\ref{l:indstep} and conclude that 
$$
\dim_{D_k}F^p_kX\le \max\{\dim_{D_{k-1}}F^p_{k-1},km\}\le
 \max\{(k-1)\cdot\dim_DX,k\cdot\dim_DX\}=k\cdot\dim_D X.
$$
 \end{proof}

Next, we consider the case of infinite dimension $\dim_DX$.

\begin{lemma}\label{l:indstep2} Assume that the functor $F$ is finitary and has finite degree. Let $n:=\max\{1,\deg(F)\}$ and $k$ be a integer number such that $2\le k\le n$. Let $D,D',D''$ be dimension controlling functions such that for any $\e\in(0,\infty)$ and numbers $$\e'=2n\cdot|Fn|\cdot\e\mbox{ \ and \ }\delta=D(3^{k^2-2k+2}\e')+3^{k^2-2k+2}\e',$$ we have $$D''(\e)\ge 2\e+2k^{\frac1p}\delta+2\e'+ D'(2\e+2\e'+k^{\frac1p}\delta).$$
Then for any distance space $X$ with $\dim_DX\ge \w$,  the distance space $F^p_kX$ has $D''$-controlled dimension  $$\dim_{D''}(F^p_kX)\le\max\{\dim_{D'}(F^p_{k-1}X),\dim_D(X)\}.$$
\end{lemma}

\begin{proof} The inequality $\dim_DX\ge\w$ implies that $|X|>1$.  
By definition of the cardinal $m':=\dim_{D'}(F^p_{k-1}X)$, there exists  a cover $\V=\bigcup_{i\in 1+m'}\V_i$ of $F^p_{k-1}X$ such that $$\inf_{i\in 1+m'}\ud_{F^pX}(\V_i)\ge 2\e+2\e'+k^{\frac1p}\delta\mbox{ \  and \ }\mesh(\V)\le D'(2\e+2\e'+k^{\frac1p}\delta).$$ 
Let $\widetilde{\V}=\bigcup_{i\in 1+m'}\widetilde\V_i$ where $\widetilde \V_i=\{O[V;\e'):V\in\V_i\}$ for $i\in 1+m'$. 

Consider the subspace $W=F^p_kX\setminus\bigcup\widetilde\V$. By analogy with Claim~\ref{cl:uL}, we can prove that for every $a\in W$ we have $|\supp(a)|=k$ and $\ud_X(\supp(a))\ge \e'$.

By definition of the (infinite) cardinal $m:=\dim_DX\ge\w$, there exists a cover $\U=\bigcup_{i\in 1+m}\U_i$ of $X$ such that $\mesh(\U)\le D(3^{k^2-2k+2}\e')$ and $\inf_{i\in 1+m}\ud_X(\U_i)\ge 3^{k^2-2k+2}\e'$. For any set $\alpha\in [1+m]^k$, consider the cover $\bigcup_{i\in \alpha}\U_i$ of the set $X_\alpha=\bigcup_{i\in\alpha}(\bigcup\U_{i})$. Observe that for every (in fact, unique) set $S\in [\alpha]^k$, we have $X_\alpha=\bigcup_{i\in S}\bigcup\U_i=\bigcup_{i\in\alpha}\bigcup\U_i$ and $$\inf_{i\in \alpha}\ud_X(\U_i)\ge \inf_{i\in 1+m}\ud_X(\U_i)\ge 3^{k^2-2k+2}\e'=3^{k(k-1)-(k-1)+1}\e'.$$ 
By Lemma~\ref{l:BellDra}, there exists a cover $\U_\alpha=\bigcup_{i\in k^2-k+1}\U_{(\alpha,i)}$ of $X_\alpha$ such that 
\begin{itemize}
\item[(a)] $\inf\limits_{i\in k^2-k+1}\ud_X(\U_{(\alpha,i)})\ge\e'$;
\item[(b)] $\mesh(\U_\alpha)\le \mesh(\U)+3^{k^2-2k+2}\e'\le D(3^{k^2-2k+2}\e')+3^{k^2-2k+2}\e'=\delta$;
\item[(c)] for any $T\in [X_\alpha]^k$ there exists $i\in k^2-k+1$ such that $T\subseteq \bigcup\U_{(\alpha,i)}$.
\end{itemize}

For every $\alpha\in [1+m]^k$ and $i\in k^2-k+1$, consider the sets $X_{(\alpha,i)}=\bigcup\U_{(\alpha,i)}\subseteq X_\alpha$ and
$$W_{(\alpha,i)}=\{a\in W:\supp(a)\subseteq X_{(\alpha,i)}\}\subseteq F^p_kX.$$ 

Let us show that $$W=\bigcup_{\alpha\in[1+m]^k}\bigcup_{i\in k^2-k+1}W_{\alpha,i}.$$
Given any element $a\in W\subseteq F_kX\setminus F_{k-1}X$, consider its support $\supp(a)\in [X]^k$ and find a set $\alpha\in[1+m]^k$ such that $\supp(a)\subseteq \bigcup_{i\in\alpha}(\bigcup\U_i)=X_\alpha$. By the condition (c), there exists $i\in k^2-k+1$ such that $\supp(a)\subseteq\bigcup\U_{(\alpha,i)}=X_{(\alpha,i)}$. Then $a\in W_{(\alpha,i)}$ by the definition of $W_{(\alpha,i)}$.

For every $a\in W_{(\alpha,i)}$ let $W_{(\alpha,i)}(a)$ be the set of all elements $b\in W_{(\alpha,i)}$ for which there exists a sequence of elements $a_0,\dots,a_l\in W_{(\alpha,i)}$ such that $a=a_0$, $a_l=b$ and $d^p_{FX}(a_{j-1}a_j)<\e$ for every $j\in\{1,\dots,l\}$. Let $\W_{(\alpha,i)}:=\{W_{(\alpha,i)}(a):a\in W_{(\alpha,i)}\}$. It is clear that $\W_{(\alpha,i)}$ is a cover of the set $W_{(\alpha,i)}$ and $\ud_{F^p\!X}(\W_{(\alpha,i)})\ge\e$. 

By analogy with Claim~\ref{cl:mesh}, we can prove that
$$\mesh (\W_{(\alpha,i)})\le k^{\frac1p}\cdot \mesh(\U_i)\le k^{\frac1p}\cdot \delta,$$ which implies that $\W=\bigcup_{\alpha\in[1+m]^k}\bigcup_{i\in k^2-k+1}\W_{(\alpha,i)}$ is a  cover of the set $W$ such that\break $\inf\limits_{\alpha\in[1+m]^k}\inf\limits_{i\in k^2-k+1}\ud_X(\W_{(\alpha,i)})\ge\e$ and $\mesh(\W)\le k^{\frac1p}\cdot\mesh(\U)$. By analogy with\break Claim~\ref{cl:meshsh}, we can show that $2\e+\mesh(\W)\le \inf\limits_{i\in 1+m'}\ud_{F^pX}(\widetilde \V_i)$.

Let $m''=\max\{m',m\}$ and choose any surjective maps $\alpha:1+m''\to 1+m'$ and $\beta:1+m''\to [1+m]^k\times (k^2-k+1)$ (the map $\beta$ exists since the cardinal $m$ is infinite). For every $i\in 1+m''$, consider the family $\U_i=\widetilde \V_{\alpha(i)}\cup_\e \W_{\beta(i)}$, and the cover $\U=\bigcup_{i\in 1+m''}\U_i$ of the distance space $F^p_kX$. Using Lemma~\ref{l:sume}, we can show that $\inf\limits_{i\in 1+m''}\ud_{F^pX}(\U_i)\ge\e$ and
$$\mesh(\U)\le 2\e+2\cdot\mesh(\W)+\mesh(\widetilde\V)\le 2\e+2k^{\frac1p}\delta+2\e'+D'(2\e+2\e'+k^{\frac1p}\delta)\le D''(\e),$$
witnessing that $$\dim_{D''} F^p_kX\le m''=\max\{\dim_{D'}F^p_{k-1}X,\dim_D X\}.$$
\end{proof}

Lemmas~\ref{l:dim1} and \ref{l:indstep2} imply the following infinite version of Theorem~\ref{t:main-dim1}.

\begin{theorem}\label{t:main-dim2} Assume that the functor $F$ is finitary and has finite degree. Let $n=\max\{1,\deg(F)\}$. Assume that for some dimension controlling function $D$, the distance space $X$ has infinite $D$-controlled dimension $\dim_DX$. Consider the sequence of dimension controlling functions $(D_k)_{k\in\w}$ defined by the recursive formulas:\\ $D_0(\e)=0$, $D_1(\e)=n^{1-\frac1p}\e)$, and
$$
\begin{aligned}
D_{k}(\e)=\;&(2+4n|Fn|)\e+2k^{\frac1p}\big(D(2n\cdot|Fn|\cdot 3^{k^2-2k+2}\e)+2n\cdot|Fn|\cdot 3^{k^2-2k+2}\e\big)+\\
&D_{k-1}\big(2\e(1+n\cdot |Fn|)+k^{\frac1p}(D(2n{\cdot}|Fn|{\cdot} 3^{k^2-2k+2}\e)+2n{\cdot}|Fn|{\cdot} 3^{k^2-2k+2}\e)\big)
\end{aligned}$$
for every $k\ge 2$ and $\e\in(0,\infty)$. Then $\dim_{D_k}F^p_kX\le \dim_DX$ for every $k\le \deg(F)$. In particular, $\dim_{D_n}F^pX\le\dim_D X$.
\end{theorem}

Theorems~\ref{t:main-dim1}, \ref{t:main-dim2} imply Theorem~\ref{t:main-dim} because for any dimension controlling scale $\mathcal D$ and any function $D\in\mathcal D$, the dimension controlling functions $D_k$ defined by the recursive formulas in Theorems~\ref{t:main-dim1}, \ref{t:main-dim2} belong to the scale $\mathcal D$.

Now we consider three important examples of dimension controlling scales.

\begin{example} Consider the family $\mathcal{AN}$ of all non-decreasing functions $f:(0,\infty)\to(0,\infty)$ such that $$\sup_{x\in(0,\infty)}\frac{f(x)}x<\infty.$$It is easy to see that $\mathcal{AN}$ is the smallest dimension controlling scale. For metric spaces, the $\mathcal{AN}$-controlled dimension $\dim_{\mathcal{AN}}$ is known as the {\em Assouad-Nagata dimension}, see \cite{LS} and \cite{BL}. For this dimension, Theorem~\ref{t:main-dim} implies the formula
$$\dim_{\mathcal{AN}}F^pX\le\deg(F)\cdot\dim_{\mathcal{AN}}X$$
holding for any $p\in[1,\infty]$, distance space $X$ and finitary functor $F:\Set\to\Set$ of finite degree. 
\end{example} 

\begin{example} Consider the dimension controlling scale $\mathcal{AS}$ consisting of all non-decreasing functions $f:(0,\infty)\to(0,\infty)$. For metric spaces, the $\mathcal{AS}$-controlled dimension $\dim_{\mathcal{AN}}$ is known as the {\em asymptotic dimension of Gromov}, see \cite{Gromov}. For the asymptotic dimension, Theorem~\ref{t:main-dim} implies the formula
$$\dim_{\mathcal{AS}}F^pX\le\deg(F)\cdot\dim_{\mathcal{AS}}X$$
holding for any $p\in[1,\infty]$, distance space $X$ and finitary functor $F:\Set\to\Set$ of finite degree. This formula has been announced by Radul and Shukel in \cite{RS}, but their proof works only for support preserving functors.
\end{example} 

\begin{example} Consider  the dimension controlling scale $\U$ consisting of all non-decrea\-sing functions $f:(0,\infty)\to(0,\infty]$ such that $$\lim_{x\to+0}f(x)=0.$$ For metric spaces, the $\U$-controlled dimension $\dim_{\U}$ is known as the {\em uniform covering dimension of Isbell}, see \cite[Ch.5]{Isbell-book}. For the uniform dimension, Theorem~\ref{t:main-dim} implies the formula
$$\dim_{\U} F^pX\le\deg(F)\cdot\dim_{\U}X$$
holding for any $p\in[1,\infty]$, distance space $X$ and finitary functor $F:\Set\to\Set$ of finite degree. For a compact metric space $X$, its uniform covering dimension $\dim_{\U}X$ coincides with its topological dimension $\dim X$. In this case the above formula turns into the formula
$$\dim FX\le\deg(F)\cdot\dim X,$$
proved for normal functors in the category of compact Hausdorff spaces by Basmanov, see \cite{Bas1,Bas2,Bas3}.
\end{example} 

\section{Concluding theorems on the properties of the distance $d^p_{FX}$}

In the following two theorems we unify the main results of the preceding sections, describing the properties of the functor $F^p:\Dist\to\Dist$.

\begin{theorem}\label{t:main} For any $p\in[1,\infty]$, and a nonempty distance space $(X,d_X)$, the $\ell^p$-metrization $F^p:\DMetr\to\DMetr$ of the functor $F:\Set\to\Set$ has the following properties:
\begin{enumerate}\itemsep3pt
\item[\textup{1)}] $d^p_{FX}$ is the largest distance on $FX$ such that for every $n\in\IN$ and $a\in Fn$, the map $\xi_X^a:(X^n,d^p_{X^n})\to FX$, $\xi^a_X:f\mapsto Ff(a)$, is non-expanding.
\item[\textup{2)}]  if $F$ has finite degree $n=\deg(F)$, then $d^p_{FX}$ is the largest distance on $FX$ such that for every $a\in Fn$, the map $\xi_X^a:(X^n,d^p_{X^n})\to FX$, $\xi^a_X:f\mapsto Ff(a)$, is non-expanding.
\item[\textup{3)}]  $d^p_{FX}(a,b)=\infty$ for any distinct points $a,b\in FX$ with $|\supp(a)\cup\supp(b)|\le 1$.
\item[\textup{4)}]  $d^\infty_{FX}\le d^p_{FX}\le d^1_{FX}$ and $d^p_{FX}\le n^{\frac1p}\cdot d^\infty_{FX}$ where $n=\max\{1,\deg(F)\}$.
\item[\textup{5)}]  For every Lipschitz map $f$ between distance spaces, the map
$F^p\!f$ is Lipschitz with Lipschitz constant $\Lip(F^p\!f)\le \Lip(f)$.
\item[\textup{6)}]  If $f:X\to Y$ is an injective isometry of distance spaces and the image $f[X]$ is dense in $Y$, then the map $F^p\!f$ is an isometry.
\item[\textup{7)}] If $f:X\to Y$ is an isometry between distance spaces and the functor $F$ is finitary and has finite degree, then for every $a,b\in FX$,\newline $d^\infty_{FX}(a,b)\le 2n\cdot|Fn|\cdot d^\infty_{FY}(Ff(a),Ff(b))$ where $n=\max\{1,\deg(F)\}$.
\item[\textup{8)}]  For any $a,b\in FX$ with $d^p_{FX}(a,b)<\infty$ we have $d^p_{FX}(a,b)\le(|S_a|^{\frac1p}+|S_b|^{\frac1p})\cdot\bar d_X(S_a\cup S_b)$ and $d^p_{FX}(a,b)\le |S_a|^{\frac1p}\cdot\bar d_X(S_a)+\ud_X(S_a,S_b)+|S_b|^{\frac1p}\cdot \bar d_X(S_b)$, where $S_a=\supp(a)$ and $S_b=\supp(b)$.
\item[\textup{9)}]  If the functor $F$ has finite degree $\le n$, then $\bar d^p_{FX}(FX)\le 2n^{\frac1p}\cdot\bar d_X(X)$.
\item[\textup{10)}]  $\ud_{FX}^p(FX)\ge\ud_X(X)$; moreover $d^1_{FX}(a,b)\ge\ud_X(\supp(a)\cup\supp(b))$ for any distinct elements $a,b\in FX$.
\item[\textup{11)}]  If $F$ preserves supports, then for any distinct elements $a,b\in FX$ we have\newline
$d^p_{FX}(a,b)\ge d^\infty_{FX}(a,b)\ge \max\{d_{HX}(\supp(a),\supp(b)),\frac13\ud_X(\supp(a)\cup\supp(b))\}$.
\item[\textup{12)}]  If $F1$ is a singleton, and $d_X$ is a pseudometric, then the distance $d^p_{FX}$ is a pseudometric.
\item[\textup{13)}]  The distance $d^p_{FX}$ is an $\infty$-metric if $d_X$ is an $\infty$-metric and one of the following conditions holds:  (i) the $\infty$-metric space $(X,d_X)$ is Lipschitz disconnected, (ii) $p=1$, (iii) the functor $F$ has finite degree, (iv) the functor $F$ preserves supports.
\item[\textup{14)}]  The distance $d^p_{FX}$ is a metric if $d_X$ is a metric, $F1$ is a singleton and one of the following conditions holds:  (i) the metric space $(X,d_X)$ is Lipschitz disconnected, (ii) $p=1$, (iii) the functor $F$ has finite degree, (iv) the functor $F$ preserves supports.
\item[\textup{15)}]  If for some $a\in F1$ the map $\xi^a_2:2\to F2$ is injective and $p=1$ or $F$ preserves supports, then the function $\xi^a_X:X\to F^pX$ is an injective isometry.
\item[\textup{16)}] If $d_X$ is $\{0,\infty\}$-valued (and $\infty$-metric) on $X$, then $d^p_{FX}$ is $\{0,\infty\}$-valued (and $\infty$-metric) on $FX$.
\end{enumerate}
\end{theorem}

\begin{proof}
1. See Theorem~\ref{t:alt}.

2. See Theorem~\ref{t:altm}.

3. See Theorem~\ref{t:small}.

4. See Theorem~\ref{t:ineq}.

5. See Theorem~\ref{t:Lip}.

6. See Theorem~\ref{t:isom}.

7. See Proposition~\ref{p:weakiso}.

8. See Theorem~\ref{t:up}.

9. See Corollary~\ref{c:up2}.

10. See Proposition~\ref{p:e-sep} and Theorem~\ref{t:d1}.

11. See Theorem~\ref{t:shukel}.

12. See Corollary~\ref{c:up}.

13. Apply Theorems~\ref{t:Ldis}, \ref{t:d1}, \ref{t:ineq}, and \ref{t:shukel}.

14. Apply the stetements (11) and (12) of this theorem.
 
15. See Theorem~\ref{t:xi}.

16. See Lemma~\ref{l:ivalued}.
\end{proof}

For finitary functors of finite degree we have some additional results.

\begin{theorem}\label{t:main-fdeg} If the functor $F$ is finitary and has finite degree $n=\deg(F)>0$, then for any $p\in[1,\infty]$, and any nonempty distance space $(X,d_X)$, the $\ell^p$-metrization $F^p:\DMetr\to\DMetr$ of the functor $F:\Set\to\Set$ has the following properties:
\begin{enumerate}
\item[\textup{1)}] For any function $f:X\to Y$ to a distance space $Y$, the function $F^pf:F^pX\to F^pY$ has continuity modulus $\w_{F^p\!f}\le |Fn|\cdot n^{\frac1p}\cdot\w_f$. 
\item[\textup{2)}] If a function $f:X\to Y$ to a distance space is uniformly continuous (resp. asymptotically Lipschitz, macro-uniform, a quasi-isometry, a coarse equivalence), then so is the function $F^pf$.
\item[\textup{3)}] If the functor $F$ preserves supports, then for any continuous function $f:X\to Y$ between distance spaces, the function $F^pf:F^pX\to F^pY$ is continuous.
\item[\textup{4)}] If the distance space $X$ is complete (and compact), then so is the space $F^pX$.
\item[\textup{5)}] The distance space $F^pX$ has metric entropy\newline $\IE_\e(F^pX)\le|Fn|\cdot\IE_\e(X^n,d^p_{X^n})\le|Fn|\cdot\big(\IE_{\e/\!\sqrt[p]{n}}(X)\big)^n$ for every $\e\in(0,\infty)$.
\item[\textup{6)}] The distance space $F^pX$ has cardinal entropy characteristics $\IE_0(F^pX)\le|Fn|\cdot \IE_0(X)^n$ and $\IE_\infty(F^pX)\le|Fn|\cdot \IE_\infty(X)^n$.
\item[\textup{7)}] If the distance space $X$ is totally bounded, then the distance space $F^pX$ is totally bounded and has upper and lower box-counting dimensions\newline $\overline{\dim}_B F^pX\le \deg(F)\cdot \overline{\dim}_B X$ \ and \  $\underline{\dim}_B F^pX\le \deg(F)\cdot \underline{\dim}_B X.$
\item[\textup{8)}] If the functor $F$ preserves supports, then the distance space $F^pX$ has metric entropy $\IE_{\e,\delta}(F^p\!X)\le
n^n\cdot |Fn|\cdot \IE_{\e,2\delta\!\sqrt[p]{n}}(X^n,d^p_{X^n})\le n^n\cdot |Fn|\cdot \big(\IE_{\e/\!\sqrt[p]{n},2\delta\!\sqrt[p]{n}}(X)\big)^n$ for any $\e,\delta\in(0,\infty)$.
\item[\textup{9)}] If the functor $F$ preserves supports, then the  distance space $F^pX$ has the cardinal entropy characteristics $\IE_{(\e,\delta)}(F^pX)\le n^n{\cdot}|Fn|\cdot \IE_{(\e,\delta)}(X)^n$ for any $(\e,\delta)\in\{(0,0),(0,\infty),(\infty,\infty)\}$.
\item[\textup{10)}] If for some $(\e,\delta)\in\{(0,0),(0,\infty),(\infty,\infty)\}$ a distance space $X$ is $(\e,\delta)$-finitary, then the distance space $F^pX$ is $(\e,\delta)$-finitary and has $(\e,\delta)$-entropy dimensions
$\odim_{(\e,\delta)}F^pX\le \deg(F)\cdot\odim_{(\e,\delta)}X$ \ and \ 
$\udim_{(\e,\delta)}F^pX\le \deg(F)\cdot\udim_{(\e,\delta)}X$.
\item[\textup{11)}] The distance space $F^pX$ has Hausdorff dimension $\dim_H F^pX\le \dim_HX^n$.
\item[\textup{12)}] For any dimension controlling scale $\mathcal D$, the distance space $F^pX$ has $\mathcal D$-controlled dimension $\dim_{\mathcal D}F^pX\le\deg(F)\cdot\dim_{\mathcal D}X$. 
\end{enumerate}
\end{theorem}

\begin{proof} 1. See Theorem~\ref{t:modulus}.

2. See Corollaries~\ref{c:uniform} and \ref{c:coarse-equiv}.

3. See Theorem~\ref{t:continuous}.

4. See Theorems~\ref{t:complete} and \ref{t:top}.

5. See Theorem~\ref{t:entropy1}.

6. See Corollary~\ref{c:card-E-char}(1,2).

7. See Corollary~\ref{c:tb} and Theorem~\ref{t:tb-dim}.

8. See Theorem~\ref{t:entropy2}.

9. See Corollary~\ref{c:card-E-char}(3,4,5).

10. See Theorem~\ref{t:bg-dim}.

11. See Theorem~\ref{t:H-dim}.

12. See Theorem~\ref{t:main-dim}. 
\end{proof}

\section{The $\check\ell^p$-metrization $\check F^p$ of the functor $F$}\label{s:stable}

In this section we indroduce a lifting $\check F^p:\DMetr\to\DMetr$ of the functor $F:\Set\to\Set$, called the $\check\ell^p$-metrization of $F$. Comparing to the $\ell_p$-metrization $F^p$, the $\check\ell_p$-metrization $\check F^p$ has some additional properties, in particular, it preserves isometries between distance spaces.
The  $\check\ell_p$-metrization is defined with the help of  injective envelopes which were discussed in Section~\ref{s:E(X)}. 
 
For any distance space $(X,d_X)$, the distance $\check d^p_{FX}:FX\times FX\to[0,\infty]$ is defined by the formula
$$\check d^p_{FX}(a,b)=d^p_{FEX}(F\mathsf e_X(a),F\mathsf e_X(b))\mbox{ \ for \ }a,b\in FX.$$
Here $\mathsf e_X:X\to EX$ is the canonical isometry of $X$ into its injective envelope $EX$, defined in Lemma~\ref{l:embE}.


The {\em $\check\ell^p$-metrization} of the functor $F$ is the functor $\check F^p:\DMetr\to\DMetr$ assigning to each distance space $\mathsf X=(X,d_X)$ the distance space $\check F^p\mathsf X=(FX,\check d^p_{FX})$ and to each function $f:(X,d_X)\to (Y,d_Y)$ between distance spaces the function $\check F^pf:=Ff$.

It is clear that $U\circ\check F^p=F\circ U$, where $U:\DMetr\to \Set$, $U:(X,d_X)\mapsto X$, is the forgetful functor. So, $\check F^p$ is a  lifting of the functor $F$ to the category $\DMetr$. In the following theorem we establish some properties of this lifting.

\begin{theorem}\label{t:mainU} The $\check\ell^p$-metrization $\check F^p:\DMetr\to\DMetr$ of the functor $F:\Set\to\Set$ has the following properties for any nonempty distance space $(X,d_X)$:
\begin{enumerate}
\item[\textup{1)}] $\check d^p(a,b)=\infty$ if $a,b\in FX$ are distinct elements with $|\supp(a)\cup\supp(b)|\le 1$.
\item[\textup{2)}] $\check d^\infty_{FX}\le \check d^p_{FX}\le \check d^1_{FX}$ and $\check d^p_{FX}\le n^{\frac1p}\cdot \check d^\infty_{FX}$, where $n=\max\{1,\deg(F)\}$.
\item[\textup{3)}] $\check d^p_{FX}\le d^p_{FX}$ and $d^\infty_{FX}\le n\cdot|Fn|\cdot \check d^\infty_{FX}$, where $n=\max\{1,\deg(F)\}$; moreover, $\check d^p_{FX}=d^p_{FX}$ if the distance space $(X,d_X)$ is injective.
\item[\textup{4)}] For every Lipschitz map $f$ between distance spaces, the map
$\check F^p\!f$ is Lipschitz with Lipschitz constant $\Lip(\check F^p\!f)\le \Lip(f)$.
\item[\textup{5)}] For any isometry $f$ between distance spaces, the map $\check F^p\!f$ is an isometry.
\item[\textup{6)}]  $\check d^p_{FX}$ is a pseudometric if $d_X$ is a pseudometric and $F1$ is a singleton.
\item[\textup{7)}]  $\check d^1_{FX}(a,b)\ge\ud_X(\supp(a)\cup\supp(b))$ for any distinct elements $a,b\in FX$.
\item[\textup{8)}] If $F$ preserves supports, then for any distinct elements $a,b\in FX$ we have\newline
$\check d^p_{FX}(a,b)\ge \check d^\infty_{FX}(a,b)\ge \max\{d_{HX}(\supp(a),\supp(b)),\frac13\ud_X(\supp(a)\cup\supp(b))\}$.
\item[\textup{9)}] $\check d^p_{FX}$ is an $\infty$-metric if the distances $d_X$ and $d^p_{FEX}$ are $\infty$-metric.
\item[\textup{10)}] $\check d^p_{FX}$ is an $\infty$-metric if $d_X$ is an $\infty$-metric and one of the following conditions is satisfied: (i) $p=1$, (ii) $F$ has finite degree, (iii) $F$ preserves supports.
\item[\textup{11)}] $\check d^p_{FX}$ is a metric if $d_X$ is a metric, $F1$ is a singleton, and one of the following conditions is satisfied: (i) $p=1$, (ii) $F$ has finite degree, (iii) $F$ preserves supports.
\item[\textup{12)}] If for some $a\in F1$ the map $\xi^a_2:2\to F2$ is injective and $p=1$ or $F$ preserves supports, then the function $\xi^a_X:X\to \check F^pX$ is an injective isometry.
\item[\textup{16)}] If $d_X$ is the $\{0,\infty\}$-valued $\infty$-metric on $X$, then $\check d^p_{FX}$ is the $\{0,\infty\}$-valued $\infty$-metric on $FX$.
\end{enumerate}
\end{theorem}

\begin{proof} 
1. Assume that $a,b$ are two distinct elements of $FX$ with $|\supp(a)\cup\supp(b)|\le 1$. Take any singleton $S\subseteq X$ containing the set $\supp(a)\cup\supp(b)$, and let $r:X\to S$ be the constant map. By Proposition~\ref{p:BMZ},  $a,b\in Fi_{S,X}[S;X]$ and hence $a=Fi_{S,X}(a')$, $b=Fi_{S,X}(b')$ for some distinct elements $a',b'\in FS$. Since the set $ES$ is a singleton (containing the constant function $S\to\{0\}$), the map $\mathsf e_S:S\to ES$  
is  bijective and so is the map $F\mathsf e_X:FS\to FES$. Let $\bar r:EX\to ES$ be the unique (constant) map. The equality $\bar r\circ \mathsf e_X\circ i_{S,X}=\mathsf e_S$ implies $F\bar r\circ F\mathsf e_X\circ Fi_{S,X}=F\mathsf e_S$. Then
$$
F\bar r\circ F\mathsf e_X(a)=F\bar r\circ F\mathsf e_X\circ Fi_{S,X}(a')=F\mathsf e_S(a')\ne
 F\mathsf e_S(b')=F\bar r\circ F\mathsf e_X\circ Fi_{S,X}(b')=F\bar r\circ F\mathsf e_X(b)
$$
 and hence $F\mathsf e_X(a)\ne F\mathsf e_X(b)$.
It follows from $F\mathsf e_X(a),F\mathsf e_X(b)\in F[\mathsf e_X[S];EX]$  that\newline $\supp(F\mathsf e_X(a))\cup\supp(F\mathsf e_X(b))\subseteq \mathsf e_X[S]$. Applying Theorem~\ref{t:small}, we conclude that $$\check d^p_{FX}(a,b)=d^p_{FEX}(F\mathsf e_X(a),F\mathsf e_X(b))=\infty.$$ 
\smallskip

2. The second statement follows from the definition of the distance $\check d_{FX}$ and Theorem~\ref{t:main}(4).
\smallskip

3. The inequality  $\check d^p_{FX}\le d^p_{FX}$ follows from the definition of the distance $\check d_{FX}$ and Theorem~\ref{t:main}(5).
\smallskip

Next, we prove that $d^\infty_{FX}\le 2n\cdot |Fn|\cdot \check d^\infty_{FX}$ where $n=\max\{1,\deg(F)\}$. To derive a contradiction, assume that  $d^\infty_{FX}(a,b)>2n\cdot |Fn|\cdot \check d^\infty_{FX}(a,b)$ for some elements $a,b\in FX$. In this case, $2n\cdot |Fn|<\infty$.  Since $d^\infty_{FEX}(F\mathsf e_X(a),F\mathsf e_X(b))=\check d^\infty_{FX}(a,b)<\frac1{2n\cdot|Fn|}d^\infty_{FX}(a,b)$, we can apply Lemma~\ref{l:2short} and find a function $r:EX\to EX$ such that  $Fr\circ F\mathsf e_X(a)=Fr\circ F\mathsf e_X(b)$, $r[\mathsf e_X[X]]\subseteq \mathsf e_X[X]$,  and $\sup_{y\in EX}d_{EX}(y,r(y))<\frac12 d^\infty(a,b)$.

Let $s:EX\to X$ be any function such that $s(y)\in \mathsf e_X^{-1}(y)$ for every $y\in r[EX]\cap\mathsf e_X[X]$. Consider the function $h=s\circ r\circ \mathsf e_X:X\to X$ and observe that for every $x\in X$ we have 
\begin{multline*}
d_X(x,h(x))=d_{EX}(\mathsf e_X(x),\mathsf e_X\circ s\circ r\circ \mathsf e_X(x))=d_{EX}(\mathsf e_X(x),r\circ \mathsf e_X(x))<\tfrac12d^\infty(a,b).
\end{multline*}
Using Proposition~\ref{p:BMZ}, find a function $f:n\to X$ and an element $a'\in Fn$ such that $a=Ff(a')$. By the definition of the distance $d^\infty_{FX}$, we have $$d^\infty_{FX}(a,Fh(a))=d^\infty_{FX}(Ff(a'),F(h\circ f)(a'))\le d^\infty_{X^n}(f,h\circ f)<\tfrac12 d^\infty_{FX}(a,b).$$ By analogy, we can prove that $d^\infty_{FX}(b,Fh(b))<\frac12d^\infty_{FX}(a,b)$.

Observe that $$Fh(a)=Fs\circ Fr\circ F\mathsf e_X(a)=Fs\circ Fr\circ F\mathsf e_X(b)=Fh(b).$$
Finally,
$$d^\infty_{FX}(a,b)\le d^\infty_{FX}(a,Fh(a))+d^\infty_{FX}(Fh(b),b)<\tfrac12d^\infty_{FX}(a,b)+\tfrac12d^\infty_{FX}(a,b)=d^\infty_{FX}(a,b),$$
which is a contradiction that completes the proof of the inequality $d^\infty_{FX}\le 2n\cdot |Fn|\cdot \check d^\infty_{FX}$.
 \smallskip

 Now assume that the distance space $(X,d_X)$ is injective. By Proposition~\ref{p:pi}, there exists a non-expanding map $r:EX\to X$ such that $r\circ \mathsf e_X$ is the identity map of $X$. By Theorem~\ref{t:main}(5), the maps $Fr:FEX\to FX$ and $F\mathsf e_X:FX\to FEX$ are non-expanding. Since $Fr\circ F\mathsf e_X$ is the identity map of $FX$, the non-expanding map $F\mathsf e_X:FX\to FEX$ is an isometry and hence  $$\check d^p_{FX}(a,b)=d^p_{FEX}(F\mathsf e_X(a),F\mathsf e_X(b))=d^p_{FX}(a,b)$$for any $a,b\in FX$.
\smallskip

4. Let $f:X\to Y$ be a Lipschitz map between distance spaces. By Proposition~\ref{p:ELip}, there exists a Lipschitz map $\bar f:EX\to EY$ such that $\bar f\circ \mathsf e_X=\mathsf e_Y\circ f$ and $\Lip(\bar f)\le\Lip(f)$. By Theorem~\ref{t:main}(5), the map $F^p\bar f=F\bar f:FEX\to FEY$ is Lipschitz with Lipschitz constant $\Lip(F^p\bar f)\le \Lip(\bar f)\le\Lip(f)$. Then for any $a,b\in FX$ we have
\begin{multline*}
\check d^p_{FY}(Ff(a),Ff(b))=d^p_{FEY}(F\mathsf e_Y\circ Ff(a),F\mathsf e_Y\circ Ff(b))=
d^p_{FEY}(F\bar f\circ F\mathsf e_X(a),F\bar f\circ \mathsf e_X(b))\le\\ \Lip(F^p\bar f)\cdot d^p_{FEX}(F\mathsf e_X(a),F\mathsf e_X(b))
\le\Lip(f)\cdot \check d^p_{FX}(a,b),
\end{multline*}
which means that the map $\check F^p f=Ff:\check F^pX\to \check F^pY$ is Lipschitz with Lipschitz constant $\Lip(\check F^p)\le\Lip(f)$.
\smallskip

5. Let $f:X\to Y$ be an isometry between distance spaces. We need to show that the map $\check F^p\!f$ is an isometry. By Proposition~\ref{p:isoext}, there exists an isometry $\bar f:EX\to EY$ such that $\bar f\circ\mathsf e_X=\mathsf e_Y\circ f$ and hence $F\bar f\circ F\mathsf e_X=F\mathsf e_Y\circ Ff$. By Lemma~\ref{l:eE}, there exists a non-expanding map $g:EY\to EX$ such that $g\circ \mathsf e_Y\circ f=\mathsf e_X$. Observe that $g\circ \bar f\circ\mathsf e_X=g\circ \mathsf e_Y\circ f=\mathsf e_X$ and by Lemma~\ref{l:id},  $g\circ\bar f$ is the identity map of $EX$. Then $Fg\circ F\bar f$ is the identity map of $FEX$. By Theorem~\ref{t:main}(5), the maps $F^pg$ and $F^p\bar f$ are non-expanding. Since $F^pg\circ F^p\bar f=\id_{FEX}$, the map $F^p\bar f$ is an isometry. Now take any elements $a,b\in FX$ and observe that
\begin{multline*}
\check d^p_{FX}(a,b)=d^p_{FEX}(F\mathsf e_X(a),F\mathsf e_X(b))=d^p_{FEY}(F\bar f\circ F\mathsf e_X(b),F\bar f\circ F\mathsf e_X(b))=\\
d^p_{FEY}(F\mathsf e_Y\circ Ff(a),F\mathsf e_Y\circ Ff(b))=\check d^p_{FY}(Ff(a),Ff(b)),
\end{multline*}
which means that the map $\check F^pf:\check F^pX\to \check F^pY$ is an isometry.
\smallskip

6. Assume that $F1$ is a singleton and $d_X$ is a pseudometric. By Theorem~\ref{t:up}, the distance $d^p_{FX}$ is a pseudometric. By Theorem~\ref{t:mainU}(3), $\check d^p_{FX}\le d^p_{FX}$ and hence $\check d^p_{FX}$ is a pseudometric.
\smallskip

7. Fix any distinct elements $a,b\in FX$. The inequality $\check d^1_{FX}(a,b)\ge\ud_X(\supp(a)\cup\supp(b))$ is trivial if $\ud_X(\supp(a)\cup \supp(b))=0$. So, we assume that $\ud_X(\supp(a)\cup\supp(b))>0$, which implies that the set $S=\supp(a)\cup\supp(b)$ has cardinality $|S|\ge 2$ and the restriction $d_X{\restriction}_{S\times S}$ is an $\infty$-metric. Taking into account that $\mathsf e_X:X\to EX$ is an isometry, we conclude that the map $\mathsf e_X{\restriction}_S$ is injective and $\ud_{EX}(\mathsf e_X[S])=\ud_X(S)>0$. Since $|\mathsf e_X[X]|\ge|\mathsf e_X[S]|=|S|\ge 2$, we can apply Corollary~\ref{c:supp}, and conclude that $\supp(F\mathsf e_X(a))\subseteq \mathsf e_X[\supp(a)]$ and $\supp(F\mathsf e_X(b))\subseteq \mathsf e_X[\supp(b)]$. Since $\mathsf e_X$ is an isometry,
\begin{multline*}
\ud_{EX}\big(\supp(F\mathsf e_X(a))\cup\supp(F\mathsf e_X(b))\big)\ge\\ \ud_{EX}\big(\mathsf e_X[\supp(a)]\cup \mathsf e_X[\supp(b)]\big)=
\ud_X\big(\supp(a)\cup\supp(b)\big).
\end{multline*}
  
Now Theorem~\ref{t:main}(10) and the definition of the distances $\check d^1_{FX}$ and $\check d^1_{FX}$ imply
\begin{multline*}
\check d^1_{FX}(a,b)=d^1_{FEX}(F\mathsf e_X(a),F\mathsf e_X(b))\ge\\
\ge \ud_{EX}\big(\supp(F\mathsf e_X(a))\cup\supp(F\mathsf e_X(b))\big)\ge\ud_X(\supp(a)\cup\supp(b)).
\end{multline*}

8. Assume that the functor $F$ preserves supports and take any distinct points $a,b\in FX$. By Theorem~\ref{t:main}(11),
\begin{multline*}
\check d^\infty_{FX}(a,b)=d^\infty_{FEX}(F\mathsf e_X(a),F\mathsf e_X(b))\ge d_{HEX}\big(\supp(F\mathsf e_X(a)),\supp(F\mathsf e_X(b))\big)=\\
=d_{HEX}\big(\mathsf e_X[\supp(a)],\mathsf e_X[\supp(b)]\big)=d_{HX}\big(\supp(a),\supp(b)\big).
\end{multline*}
Next, we prove that $\check d^\infty_{FX}(a,b)\ge\frac13\ud_X(S)$ where $S=\supp(a)\cup\supp(b)$. This inequality is trivially true if $\ud_X(S)=0$. So, we assume that $ \ud_X(S)>0$.  In this case the restriction $d_X{\restriction}_{S\times S}$ is an $\infty$-metric, the isometry $\mathsf e_X{\restriction}_S$ is injective, and $\ud_{EX}(\mathsf e_X[S])=\ud_X(S)$. Applying Theorem~\ref{t:main}(11) and taking into account that $F$ preserves supports, we conclude that
$$
\begin{aligned}
\check d^\infty_{FX}(a,b)&=d^\infty_{FEX}(F\mathsf e_X(a),F\mathsf e_X(b))\ge\tfrac13\ud_{EX}\big(\supp(F\mathsf e_X(a))\cup\supp(F\mathsf e_X(b))\big)=\\
&=\tfrac13\ud_{EX}\big(\mathsf e_X[\supp(a)]\cup \mathsf e_X[\supp(b)]\big)=\tfrac13\ud_{EX}(\mathsf e_X[S])=\tfrac13\ud_X(S).
\end{aligned}
$$The inequality $\check d^p_{FX}\ge \check d^\infty_{FX}$ has been proved in Theorem~\ref{t:mainU}(2).
\smallskip

9. Assume that the distances $d_X$ and $d^p_{FEX}$ are $\infty$-metrics. We should show that $\check d^p_{FX}(a,b)>0$ for any distinct elements $a,b\in FX$. Since $d_X$ is an $\infty$-metric, the isometry $\mathsf e_X:X\to EX$ is injective. Then we can choose a function $\gamma:EX\to X$ such that $\gamma\circ \mathsf e_X$ is the identity map of $X$ and conclude that $F\gamma\circ F\mathsf e_X$ is the identity map of $FX$, which implies that the map $F\mathsf e_X$ is injective and hence $\check d^p_{FX}(a,b)=d^p_{FEX}(F\mathsf e_X(a),F\mathsf e_X(b))>0$.
\vskip5pt

10. The statement (10) follows from the statement (9) and Theorem~\ref{t:main}(13). 
\smallskip

11. The statement (11) follows from the statements (6) and (10) of this theorem.
\smallskip

12. Assume that for some $a\in F1$ the map $\xi^a_2:2\to F2$ is injective. By Theorems~\ref{t:alt} and \ref{t:mainU}(3), the map $\xi^a_X:(X,d)\to (FX,\check d^p_{FX})$ is not expanding. 

Now assume that $p=1$ or $F$ preserves supports. In this case
 we will prove that the map $\xi^a_X:X\to \check F^pX$ is an isometry.   By Corollary~\ref{c:Einj}, the distance space $EX$ is injective and by Theorem~\ref{t:mainU}(3), $\check d^p_{FEX}=d^p_{FEX}$.  By Theorem~\ref{t:xi}, the map $\xi^a_{EX}:EX\to F^pEX=\check F^pEX$ is an isometry.
By Theorem~\ref{t:mainU}(4), the map $F\mathsf e_X:\check F^pX\to \check F^pEX=F^pEX$ is non-expanding.  
By Proposition~\ref{p:xi-natural}, the following diagram commutes.
$$
\xymatrix{
X\ar[d]_{\xi^a_X}\ar[r]^{\mathsf e_X}&EX\ar[d]^{\xi^a_{EX}}\\
\check F^pX\ar[r]^{F\mathsf e_X}&\check F^pEX\ar@{=}[r]&F^pEX
}
$$
Now the isometry property of the map $\xi^a_{EX}\circ\mathsf e_X$ and the non-expanding property of $F\mathsf e_X:\check FX\to \check F^pEX=F^pEX$ imply that the non-expanding map $\xi^a_X:X\to \check F^pX$ is an isometry.
\smallskip

13. If $d_X$ is the $\{0,\infty\}$-valued $\infty$-metric on $X$, then the $\infty$-metric of the space $EX$ is $\{0,\infty\}$-valued (by the definition of $EX$). By Theorem~\ref{t:main}(16), the distance $d^p_{FEX}$ is an $\{0,\infty\}$-valued $\infty$-metric. Then the distance $\check d^p_{FX}$ is $\{0,\infty\}$-valued and by Theorem~\ref{t:mainU}(9), $\check d^p_{FX}$ is an $\infty$-metric.  
\end{proof}

\begin{remark} In general $\check d^p_{FX}\ne d^p_{FX}$ even for support-preserving functors of finite degree. Indeed, it can be shown that for the functor $H_2$, the metric space $X\subseteq \mathbb C$ and elements $a,b\in H_2X$ from Example~\ref{ex1} we have $$\check d^p_{H_2X}(a,b)=4+4\sqrt[p]{2}<10+4\sqrt{2}=d^p_{H_2}(a,b).$$
\end{remark}

\begin{remark} If the functor $F$ is finitary and has finite degree, then by Theorem~\ref{t:mainU}(3), the identity map $F^pX\to \check F^pX$ is a by-Lipschitz isomorphism. In this case, Theorem~\ref{t:main-fdeg} (with suitable modifications of the statements (1),(5),(8)) hold also for the functor $\check F^p$.
\end{remark}

\begin{problem}\label{prob:algo2} Is there any efficient algorithm for calculating the distance $\check d^p_{FX}$ (with arbitrary precision)?
\end{problem}

\section{The $\ell^p$-metrization and natural transformations of functors}\label{s:subf}

In this section we investigate the relation between the $\ell^p$-metrizations of  functors $F,G:\Set\to\Set$, connected by a natural transformation $\eta:F\to G$.
The natural transformation $\eta$ assigns to each set $X$ a function $\eta_X:FX\to GX$ such that for any function $f:X\to Y$ between sets the diagram 
$$\xymatrix{
FX\ar^{\eta_X}[r]\ar_{Ff}[d]&GX\ar^{Gf}[d]\\
FY\ar_{\eta_Y}[r]&GY
}
$$
commutes.

If for each $X$ the function $\eta_X:FX\to GX$ is injective (resp. bijective), then we say that $\eta:F\to G$ is an {\em embedding} (resp. an {\em isomorphism}) of the functors $F$ and $G$. 


\begin{proposition}\label{p:natsupp} If $\eta:F\to G$ is an embedding of functors $F,G:\Set\to\Set$  and the functor $G$ has finite supports, then $F$ has finite supports, too.
\end{proposition}

\begin{proof} Given any set $X$ and element $a\in FX$, we should find a finite set $A\subseteq X$ such that $a\in F[A;X]:=Fi_{A,X}[FA]$. If $X$ is finite, then we put $A=X$. So, we assume that $X$ is infinite and hence not empty. Since the functor $G$ has finite supports, for the element $b:=\eta_X(a)\in \eta_X[FX]\subseteq GX$ there exists a non-empty finite set $A\subseteq X$ such that $b=Gi_{A,X}(b')$ for some $b'\in GA$. Since $A$ is not empty, there exists a function $r:X\to A$ such that $r\circ i_{A,X}$ is the identity map of $A$. Then $Fr\circ Fi_{A,X}$ is the identity map of $FA$ and $Gr\circ Gi_{A,X}$ is the identity map of $GA\subseteq GX$. 
Consider the element $a'=Fr(a)\in FA$ and observe that 
\begin{multline*}
\eta_X\circ Fi_{A,X}(a')=\eta_X\circ Fi_{A,X}\circ Fr(a)=Gi_{A,X}\circ \eta_A\circ Fr(a)=Gi_{A,X}\circ Gr\circ \eta_X(a)=\\
=Gi_{A,X}\circ Gr(b)=Gi_{A,X}\circ Gr\circ Gi_{A,X}(b')=Gi_{A,X}(b')=b=\eta_X(a).
\end{multline*}
Now the injectivity of $\eta_X$ ensures that $a=Fi_{A,X}(a')\in F[A;X]$. 
\end{proof}

\begin{theorem}\label{t:nat} Let $F,G:\Set\to\Set$ be two functors with finite supports and $\eta:F\to G$ be a natural transformation. For any distance space $(X,d_X)$ and any $p\in[1,\infty]$, the map $\eta_X:(FX,d^p_{FX})\to (GX,d^p_{GX})$ is non-expanding and so is the map $\eta_X:(FX,\check d^p_{FX})\to (GX,\check d^p_{GX})$.
\end{theorem}

\begin{proof} Given two elements $a,b\in FX$, we should prove that 
\begin{equation}\label{eq:nat}
d^p_{GX}(\eta_X(a),\eta_X(b))\le d^p_{FX}(a,b).
\end{equation}
 To derive a contradiction, assume that 
$d^p_{FX}(a,b)<d^p_{GX}(\eta_X(a),\eta_X(b))$ and find an  $(a,b)$-linking chain $w=\big((a_i,f_i,g_i)\big){}_{i=0}^l\in L_{FX}(a,b)$ with $\Sigma d^p_X(w)<d^p_{GX}(\eta_X(a),\eta_X(b))$. For every $i\in\{0,\dots,l\}$ put $n_i:=\dom(f_i)=\dom(g_i)$ and $b_i=\eta_{n_i}(a_i)\in Gn_i$. Observe that 
$\eta_X(a)=\eta_X(Ff_0(a_0))=Gf_0(\eta_{n_0}(a_0))=Gf_0(b_0)$, $\eta_X(b)=\eta_X(Fg_l(a_l))=Gg_l(\eta_{n_l}(a_l)=Gg_l(b_l)$ and for every $i\in\{1,\dots, l\}$ 
$$
Gf_i(b_i)=Gf_i(\eta_{n_i}(a_i))=\eta_X(Ff_i(a_i))=\eta_X(Fg_{i-1}(a_{i-1}))=Gg_{i-1}(\eta_{n_{i-1}}(a_{i-1}))=Gg_{i-1}(b_{i-1}),
$$
which means that $\widetilde w=\big((b_i,f_i,g_i)\big){}_{i=0}^l$ is an $(\eta_X(a),\eta_X(b))$-connecting chain and hence
$$d^p_{GX}(\eta_X(a),\eta_X(b))\le\Sigma d^p_X(\widetilde w)=\Sigma d^p_X(w)<d^p_{GX}(\eta_X(a),\eta_X(b)),$$
which is a desired contradiction, completing the proof of the inequality (\ref{eq:nat}).
\smallskip

Next, observe that 
\begin{multline*}
\check d^p_{GX}(\eta_X(a),\eta_X(b))=d^p_{GEX}(G\mathsf e_X(\eta_X(a)),G\mathsf e_X(\eta_X(b)))=\\
=d^p_{GEX}(\eta_{EX}(F\mathsf e_X(a)),\eta_{EX}(F\mathsf e_X(b))\le d^p_{FEX}(F\mathsf e_X(a),F\mathsf e_X(b))=\check d^p_{FX}(a,b),
\end{multline*}
which means that the map $\eta_X:(FX,\check d^p_{FX})\to (GX,\check d^p_{GX})$ is non-expanding.
\end{proof}

A functor $F:\Set\to\Set$ is defined to be a {\em subfunctor} of a functor $G:\Set\to\Set$ if for every set $X$ we have $FX\subseteq GX$ and for every function $f:X\to Y$ between sets, $Ff=Gf{\restriction}_{FX}$. The latter condition means that the identity inclusion $i_{FX,GX}:FX\to GX$ is a component of a natural transformation $i:F\to G$, which is a functor embedding.

A subfunctor $F$ of a functor $G:\Set\to\Set$ is defined to be {\em $1$-determined} if for any set $X$ and the constant map $\gamma:X\to 1$ we have $FX=(G\gamma)^{-1}[F1]:=\{a\in GX:G\gamma(a)\in F1\subseteq G1\}$.

\begin{theorem}\label{t:subf} Let $G:\Set\to\Set$ be a functor with finite supports and $F:\Set\to\Set$ be a subfunctor of $G$. For any distance space $(X,d_X)$, $p\in[1,\infty]$ and any $a,b\in FX\subseteq GX$ the following statements hold.
\begin{enumerate}
\item[\textup{1)}] $d^p_{GX}(a,b)\le d^p_{FX}(a,b)$ and $\check d^p_{GX}(a,b)\le \check d^p_{FX}(a,b)$.
\item[\textup{2)}] If $F$ is $1$-determined, then  $d^p_{GX}(a,b)=d^p_{FX}(a,b)$ and $\check d^p_{GX}(a,b)=\check d^p_{FX}(a,b)$.
\end{enumerate}
\end{theorem}

\begin{proof} 1. The first statement follows from Proposition~\ref{p:natsupp} and Theorem~\ref{t:nat}.
\smallskip

2. Now assuming that the subfunctor $F$  of $G$ is 1-determined, we will prove that $d^p_{GX}(a,b)=d^p_{FX}(a,b)$. By the preceding statement, it suffices to check that $d^p_{GX}(a,b)\ge d^p_{FX}(a,b)$. To derive a contradiction, assume that   $d^p_{GX}(a,b)<d^p_{FX}(a,b)$ and find  an $(a,b)$-linking chain $\big((a_i,f_i,g_i)\big){}_{i=0}^l\in L_{GX}(a,b)$ such that 
$$\sum_{i=0}^ld^p_{X^{\!<\!\w}}(f_i,g_i)<d^p_{FX}(a,b).$$
The definition of $L_{GX}(a,b)$ ensures that  
 $a=Gf_0(a_0)$, $b=Gg_l(a_l)$ and $Gg_{i-1}(a_{i-1})=Gf_i(a_i)$ for all $i\in\{1,\dots,l\}$. For every $i\in\{0,\dots,l\}$ let $n_i:=\dom(f_i)=\dom(g_i)\in\w$.

Let $\gamma:X\to 1$ and $\gamma_i:n_i\to 1$ for $i\in\{0,\dots,l\}$ be the constant maps. Taking into account that $F$ is a 1-determined subfunctor of $G$, we conclude that the equality 
$G\gamma_0(a_0)=G\gamma\circ Gf_0(a_0)=G\gamma(a)=F\gamma(a)\in F1$ implies that $a_0\in (G\gamma_0)^{-1}[F1]=Fn_0$. By induction on $i\in\{0,\dots,l\}$ we will prove that $a_i\in Fn_i$. For $i=0$ this has already been done. Assume that for some $i\in\{1,\dots,l\}$ we have proved that $a_{i-1}\in Fn_{i-1}$. Then
$$
G\gamma_i(a_i)=G\gamma\circ Gf_i(a_i)=G\gamma\circ Gg_{i-1}(a_{i-1})=G\gamma_{i-1}(a_{i-1})
=F\gamma_{i-1}(a_{i-1})\in F\gamma_{i-1}[Fn_{i-1}]\subseteq F1
$$
and hence $a_i\in (G\gamma_i)^{-1}[F1]=Fn_i$. 

After completing the inductive construction, we see that $$\big((a_i,f_i,g_i)\big){}_{i=0}^l\in L_{FX}(a,b)\subseteq L_{GX}(a,b)$$ and thus
$$d^p_{FX}(a,b)\le \sum_{i=0}^ld^p_{X^{\!<\!\w}}(f_i,g_i)<d^p_{FX}(a,b),$$
which is a desired contradiction showing that $d^p_{GX}(a,b)=d^p_{FX}(a,b)$.

Then 
$$\check d_{GX}(a,b)=d^p_{GEX}(G\mathsf e_X(a),G\mathsf e_X(b))=d^p_{FEX}(F\mathsf e_X(a),F\mathsf e_X(b))=\check d^p_{FX}(a,b).$$
\end{proof}

\section{The metrizations of the functor of $n$-th power}

Let $n\in\w$ and $F:\Set\to\Set$ be the functor assigning to each set $X$ its $n$-th power $X^n$ and to each map $f:X\to Y$ the map $Ff:FX\to FY$, $Ff:\alpha\mapsto f\circ\alpha$. It is clear that $F$ is a functor of finite degree $\deg(F)=n$, $F$ preserves supports and singletons.

\begin{theorem} For every distance $d_X$ on a set $X$ we have the inequalities
$$d^\infty_{X^n}\le \check d^p_{FX}\le d^p_{FX}\le d^p_{X^n}\le n^{\frac1p}\cdot d^\infty_{X^n}.$$
In particular, $\check d^\infty_{FX}=d^\infty_{FX}=d^\infty_{X^n}$.
If $d_X[X\times X]\ne\{0\}$, $p<\infty$ and $n\ge 2$, then $d^p_{FX}\ne d^p_{X^n}$. \end{theorem}

\begin{proof}
By Theorem~\ref{t:altm}, the pseudometric $d^p_{FX}$ coincides with the largest pseudometric on $FX$ such that for every $a\in Fn=n^n$ the map $\xi^a_X:X^n\to FX$, $\xi^a_X:f\mapsto f\circ a$, is non-expanding with respect to the $\ell^p$-metric $d^p_{X^n}$ on $X^n$. The identity map $a:n\to n$ induces the identity non-expanding map $\xi^a_X:X^n\to FX$, witnessing that $d^p_{FX}\le d^p_{X^n}$. The inequality $\check d^p_{FX}\le d^p_{FX}$ follows from Theorem~\ref{t:mainU}(3).
The inequality $d^p_{X^n}\le n^{\frac1p}\cdot d^\infty_{X^n}$ follows from Lemma~\ref{l:ineq}.

To see that $d^\infty_{X^n}\le d^p_{FX}$, observe that for every $a\in Fn=n^n$ and functions $f,g\in Fn$, we get the inequalities
$$d^\infty_{X^n}(\xi^a_X(f),\xi^a_X(g))=d^\infty_{X^n}(f\circ a,g\circ a)\le d^\infty_{X^n}(f,g)\le d^p_{X^n}(f,g),$$witnessing that the map $\xi^a_X:(X^n,d^p_{X^n})\to (FX,d^\infty_{X^n})$ is non-expanding.
By the maximality of $d^p_{FX}$, we have $d^\infty_{X^n}\le d^p_{FX}$.  Therefore, we obtain the chain of inequalities
$$d^\infty_{X^n}\le d^p_{FX}\le d^p_{X^n}\le n^{\frac1p}{\cdot}d^\infty_{X^n}$$which transforms into a chain of equalities if $p=\infty$.

Now we prove that $d^\infty_{X^n}(a,b)\le\check d^p_{FX}(a,b)$ for any $a,b\in FX=X^n$. Indeed, $$
\begin{aligned}
\check d^p_{FX}(a,b)&=d^p_{FEX}(F\mathsf e_X(a),F\mathsf e_X(b))\ge 
d^\infty_{FEX}(F\mathsf e_X(a),F\mathsf e_X(b))=d^\infty_{(EX)^n}(\mathsf e_X{\circ} a,\mathsf e_X{\circ} b)=\\
&=\max_{i\in n}d_{EX}(\mathsf e_X\circ a(i),\mathsf e_X\circ b(i))=\max_{i\in n}d_X(a(i),b(i))=d^\infty_{X^n}(a,b).
\end{aligned}
$$
\smallskip

Now assume that $d_X[X\times X]\ne\{0\}$, $p<\infty$ and $n\ge 2$. Since $d_X[X\times X]\ne\{0\}$, there exist functions $f,g\in X^n$ such that $$d_X(f(0),g(0))>0=d_X(f(1),g(1))=\dots=d_X(f(n{-}1),g(n{-}1)).$$
Let $a:n\to\{0\}\subseteq n$ be the constant map. Since the map $\xi^a_X:(X^n,d^p_{X^n})\to (FX,d^p_{FX})$, $\xi^a_X:\varphi\mapsto \varphi\circ a$, is non-expanding,
$$
d^p_{FX}(f\circ a,g\circ a)\le d^p_{X^n}(f,g)=d_X(f(0),g(0))<n^{\frac1p}\cdot d_X(f(0),g(0))=d^p_{X^n}(f\circ a,g\circ a),
$$ which implies that $d^p_{FX}\ne d^p_{X^n}$.
\end{proof}

For the functor $F$ of $n$-th power we can repeat  Problems~\ref{prob:algo} and \ref{prob:algo2}. 

\begin{problem} Is there any efficient algorithm for calculating the distances $d^p_{FX}$ and $\check d^p_{FX}$ for the functor $F$ of the $n$-th power?
\end{problem}

\section{Metrizations of the functor of  $M$-valued finitary functions}

Let $(M,+)$ be a commutative monoid whose neutral element is denoted by $0$. By $M_0$ we denote the subgroup of invertible elements of $M$.

 For every set $X$ consider the set $F(X,M)$ of all functions $\varphi:X\to M$ with finite support $\supp(\varphi):=\{x\in X:\varphi(x)\ne 0\}$. The set $F(X,M)$ carries an obvious structure of a commutative monoid, whose neutral element is the constant function $o:X\to\{0\}\subseteq M$ with empty support $\supp(o)=\emptyset$. 

For every $x\in X$ and $g\in M$ let $\delta_x^g:X\to M$ be the function with $\supp(\delta^g_x)\subseteq\{x\}$ and $\delta_x^g(x)=g$. It is easy to see that each function $\varphi\in F(X,M)$ can be written as the sum $$\varphi=\sum_{i=1}^k\delta_{x_i}^{g_i}$$ for some points $x_1,\dots,x_k\in X$ and elements $g_1,\dots,g_k\in M$.

For a subset $A\subseteq X$ and a function $\varphi\in F(X,M)$ by $\sum_{x\in A}\varphi(x)$ we understand the finite sum $\sum_{x\in A\cap\supp(\varphi)}\varphi(x)$ in the commutative monoid $M$.

The subset $F_0(X,M)=\{\varphi\in F(X,M):\sum_{x\in X}\varphi(x)=0\}$ is a subgroup in $F(X,M)$ that coincides with the subgroup $F_0(X,M_0)$.
\smallskip

If $M$ is a group, then the commutative monoids $F_0(X,M)\subseteq F(X,M)$ are groups. Moreover, if $M$ is the group $\IZ$ of integers, then $F(X,M)$ can be identified with the free Abelian group over $X$. If $M$ is a two-element group, then $F(X,M)$ is a free Boolean group over $X$.
\smallskip

For any function $f:X\to Y$ between sets let $Ff:F(X,M)\to F(Y,M)$ be the function assigning to each function $\varphi\in F(X,M)$ the (finitely supported) function $F\varphi:Y\to M$, $F\varphi:y\mapsto\sum_{x\in f^{-1}(y)}\varphi(x)$. It follows that $Ff(F_0(X,M))\subseteq F_0(Y,M)$.

This construction defines a functor $F:\Set\to\Set$ assigning to each set $X$ the set $FX:=F(X,M)$ and to each function $f:X\to Y$ between sets the function $Ff:F(X,M)\to F(Y.M)$. 
It can be shown that $Ff$ is a homomorphism between the commutative monoids $FX$ and $FY$, so actually $F$ is  a functor from the category $\Set$ to the category $\Mon$ of commutative monoids and their homomorphisms. The functor $F$ is called {\em the functor of $M$-valued finitary functions}.

Moreover, $F_0$ is a 1-determined subfunctor of the functor $F$ (for the definition of 1-determined subfunctors, see Section~\ref{s:subf}).

Now let us consider the  $\ell^p$- and $\check\ell^p$-metrizations of the functor $F$. Fix $p\in[1,\infty]$ and a distance space $(X,d_X)$. The distance $d_X$ induces the distances $d^p_{FX}$ and $\check d^p_{FX}$ on the set $FX$. These distances have all the properties described in Theorems~\ref{t:main} and \ref{t:mainU}. 
Moreover, these distances have nice algebraic properties.

\begin{proposition}\label{p:subinv} The distance $d^p_{FX}$ is subinvariant in the sense that $d^p_{FX}(a+c,b+c)\le d^p_{FX}(a,b)$ for any elements $a,b,c\in FX$ of the commutative monoid $FX$.
\end{proposition}

\begin{proof} To derive a contradiction, assume that $d^p_{FX}(a,b)<d^p_{FX}(a{+}c,b{+}c)$ and find an $(a,b)$-linking chain $w=\big((a_i,f_i,g_i)\big){}_{i=0}^l\in L_{FX}(a,b)$ such that $$\Sigma d^p_X(w)<d^p_{FX}(a{+}c,b{+}c).$$ Let $m=|\supp(c)|\in\w$. 

For every $i\in\{0,\dots,l\}$ let $n_i=\dom(f_i)=\dom(g_i)$ and $m_i:=n_i+m$. Let\break $\gamma_i:m_i\setminus n_i\to\supp(c)$ be any bijection. Let $\tilde a_i:m_i\to M$ be the function such that $\tilde a_i(x)=a_i(x)$ for $x\in n_i$ and $\tilde a_i(x)=c(\gamma_i(x))$ for $x\in m_i\setminus n_i$. Consider the functions $\tilde f_i,\tilde g_i:m_i\to X$ defined by $\tilde f_i{\restriction}_{n_i}=f_i$, $\tilde g_i{\restriction}_{n_i}=g_i$ and $\tilde f_i{\restriction}_{m_i\setminus n_i}=\gamma=\tilde g_i{\restriction}_{m_i\setminus n_i}$. 
It can be shown that $\big((\tilde a_i,\tilde f_i,\tilde g_i)\big){}_{i=0}^l\in L_{FX}(a{+}c,b{+}c)$ and 
$$d^p_{FX}(a{+}c,b{+}c)\le\sum_{i=0}^ld^p_{X^{\!<\!\w}}(\tilde f_i,\tilde g_i)=\sum_{i=0}^ld^p_{X^{\!<\!\w}}(f_i,g_i)<d^p_{FX}(a{+}c,b{+}c),$$
which is a desired contradiction, completing the proof.
\end{proof}

\begin{corollary}\label{c:subinv} The distance $\check d^p_{FX}$ is subinvariant in the sense that $$\check d^p_{FX}(a+c,b+c)\le \check d^p_{FX}(a,b)$$ for any elements $a,b,c\in FX$ of the commutative monoid $FX$.
\end{corollary}

\begin{proof} By Proposition~\ref{p:subinv},
\begin{multline*}
\check d^p_{FX}(a+c,b+c)=d^p_{FEX}(F\mathsf e_X(a+c),F\mathsf e_X(b+c))=\\
=d^p_{FEX}(F\mathsf e_X(a)+F\mathsf e_X(c),F\mathsf e_X(b)+F\mathsf e_X(c))\le d^p_{FEX}(F\mathsf e_X(a),F\mathsf e_X(b))=\check d^p_{FX}(a,b).
\end{multline*}
\end{proof}

Since $F_0$ is a 1-determined subfunctor of $F$, we can apply Theorem~\ref{t:subf}, Proposition~\ref{p:subinv} and Corollary~\ref{c:subinv} to deduce the following corollary.

\begin{corollary}\label{c:inv} $d^p_{F_0X}=d^p_{FX}{\restriction}_{F_0X{\times}F_0X}$ and $\check d^p_{F_0X}=\check d^p_{FX}{\restriction}_{F_0X{\times}F_0X}$. Moreover, the distances $d^p_{F_0X}$ and $\check d^p_{F_0X}$ are invariant in the sense that $$ d^p_{FX}(a+c,b+c)=d^p_{FX}(a,b)\mbox{ \ and \ }\check d^p_{FX}(a+c,b+c)= \check d^p_{FX}(a,b)$$ for any elements $a,b,c\in F_0X$ of the commutative group $F_0X$.
\end{corollary}

Let $\sim$ be the equivalence relation on $X$ defined by $x\sim y$ iff $d_X(x,y)<\infty$. For each point $x\in X$ its equivalence class coincides with the ball $O(x;\infty)=\{y\in X:d_X(x,y)<\infty\}$. Let $X/_\sim=\{O(x,\infty):x\in X\}$ be the quotient set of all such $\infty$-balls.

In the group $F_0X$ consider the subgroup $F_{00}X$ consisting of finitely supported functions $\varphi:X\to M_0$ such that $\sum_{x\in E}\varphi(x)=0$ for any equivalence class $E\in X/_\sim$. 

Now we are going to show that the distance $\check d^p_{FX}$ is trivial on the subgroup $F_{00}X$ whenever $p>1$ and the distance space $X$ is {\em Lipschitz-geodesic} in the sense that for any points $x,y\in X$ with $d_X(x,y)<\infty$ there exists a Lipschitz map $\gamma:[0,1]\to X$ such that $\gamma(0)=x$ and $\gamma(1)=y$. Here $[0,1]\subseteq \IR$ is the closed unit interval on the real line with the standard Euclidean distance. It is easy to see that each injective metric space is Lipschitz-geodesic.

\begin{theorem}\label{t:d-zero} If $p>1$ and $(X,d_X)$ is a Lipschitz-geodesic distance space, then\newline $d^p_{FX}(a,b)=0$ for any elements $a,b\in F_{00}X\subseteq FX$.
\end{theorem}

\begin{proof} Observe that the invariant distance $d^p_{FX}{\restriction}_{F_{0}X{\times}F_{0}X}$ on $F_0X$  is uniquely determined by the norm $\|\cdot\|:F_0X\to[0,\infty]$ assigning to each function $\varphi\in F_0X$ the distance $d^p_{FX}(\varphi,o)$ from $\varphi$ to the zero function $o:X\to\{0\}\subseteq M$, which is the neutral element of the group $F_0X$.

The equality $d^p_{FX}(a,b)=0$ for any $a,b\in F_{00}X$ will follow as soon as we prove that $\|\varphi\|=0$ for any $\varphi\in F_{00}X$. 

Choose any function $r:X\to X$ such that for each $\infty$-ball $E\in X/_\sim$ its image $r[E]$ is a singleton in $E$.
For every $E\in X/_\sim$ we have $\sum_{x\in E}\varphi(x)=0$ (by the definition of the group $F_{00}M$). This implies that $\varphi=\sum_{x\in \supp(\varphi)}(\delta_x^{\varphi(x)}-\delta_{r(x)}^{\varphi(x)})$. The triangle inequality will imply that $\|\varphi\|=0$ as soon as we prove that $\|\delta_x^g-\delta^g_{r(x)}\|=0$ for every $x\in X$ and $g\in M_0$.

Since $d_X(x,r(x))<\infty$, we can apply the Lipschitz-geodesic property of $(X,d_X)$ and find a Lipschitz map $\gamma:[0,1]\to X$ such that $\gamma(0)=x$ and $\gamma(1)=r(x)$. Now given any $\e>0$, we will prove that $\|\delta^g_x-\delta^g_{r(x)}\|<\e$. 

Find an even number $n$ so large that $\Lip(\gamma)\cdot (n-1)^{\frac{1-p}p}<\e$ (such number $n$ exists as $p>1$). Next, consider the function $a_0:n\to \{-g,g\}\subseteq M_0\subseteq M$ such that $a_0(i)=(-1)^i\cdot g$ for $i\in n$.  Now consider the function $f_0:n\to X$ such that  $f_0(i)=f_0(i+1)=\gamma(\frac{i}{n-1})$ for every even number $i\in n$. Observe that $Ff_0(a_0)=o$. Next, let $g_0:n\to X$ be the function such that $g_0(0)=x$, $g_0(n-1)=r(x)$ and  $g_0(i-1)=g_0(i)=\gamma(\frac{i-1}{n-1})$ for every non-zero even number $i\in n$. It is easy to see that $Fg_0(a_0)=\delta_x^g-\delta_{r(x)}^g$. Then $\big((a_0,f_0,g_0)\big)$ is a $(o,\delta_x^g-\delta_{r(x)}^g)$-linking chain and hence
$$
\begin{aligned}
\|\delta_x^g-\delta_{r(x)}^g\|&=d^p_{FX}(o,\delta_x^g-\delta^g_{r(x)})\le d^p_{X^{\!<\!\w}}(f_0,g_0)=\Big(\sum_{i\in n}d_X(f_0(i),g_0(i))^p\Big)^{\frac1p}\le\\
&\le \Big(0+(n-2)(\Lip(\gamma)\tfrac{1}{n-1})^p+0\Big)^{\frac1p}<\Lip(\gamma){\cdot}(n-1)^{\frac{1-p}p}<\e,
\end{aligned}
$$witnessing that $\|\delta_x^g-\delta_{r(x)}^g\|=0$.
\end{proof}

\begin{corollary}\label{c:0infty} 
If $p>1$ and the distance space $(X,d_X)$ is Lipschitz-geodesic, then for any $a,b\in F_{0}X$  
$$
d^p_{FX}(a,b)=\begin{cases}
0&\mbox{if $a-b\in F_{00}X$};\\
\infty&\mbox{otherwise}.
\end{cases}
$$
\end{corollary}

\begin{proof} Let $q:X\to Y$, $q:x\mapsto O(x,\infty)$ be the quotient map to the quotient space $Y=X/_\sim=\{O(x,\infty):x\in X\}$ of all $\infty$-balls in the distance space $(X,d_X)$. 

Take any elements $a,b$ of the group $F_0X$. If $a-b\in F_{00}X$, then by Corollary~\ref{c:inv} and Theorem~\ref{t:d-zero},
$$d^p_{FX}(a,b)=d^p_{FX}(a-b,o)=0.$$

If $a-b\notin F_{00}X$, then $\sum_{x\in E}(a(x)-b(x))\ne0$ for some $\infty$-ball $E\in X/\sim$, which implies that $\sum_{x\in E}a(x)\ne\sum_{x\in E}b(x)$ and hence $Fq(a)\ne Fq(b)$. By Theorem~\ref{t:up}, $d^p_{FX}(a,b)=\infty$.
\end{proof}

\begin{theorem}\label{t:Gp} If $p>1$, then for any $a,b\in F_{0}X$ we have  
$$
\check d^p_{FX}(a,b)=\begin{cases}
0&\mbox{if $a-b\in F_{00}X$};\\
\infty&\mbox{otherwise}.
\end{cases}
$$
\end{theorem}

\begin{proof}  By Corollay~\ref{c:Einj}, the distance space $EX$ is injective and hence Lipschitz-geodesic. Applying Corollary~\ref{c:0infty} to the Lipschitz-geodesic distance space $EX$, we conclude that
$$
\check d^p_{FX}(a,b)=d^p_{FEX}(F\mathsf e_X(a),F\mathsf e_X(b))=
\begin{cases}
0&\mbox{iff $F\mathsf e_X(a){-}F\mathsf e_X(b)\in F_{00}EX$ iff $a-b\in F_{00}X$};\\
\infty&\mbox{iff $F\mathsf e_X(a){-}F\mathsf e_X(b)\notin F_{00}EX$ iff $a-b\notin F_{00}X$}.
\end{cases}
$$
\end{proof}

Theorem~\ref{t:Gp} and Corollary~\ref{c:0infty} yield simple formulas for calculating the distances $\check d^p_{F_0X}$ and $d^p_{F_0X}$ if $p>1$. Now we present a simple formula for calculating the distance $d^1_{F_0X}$.

For an element $\varphi\in FX$ let $R_{X}(\varphi)$ be the set of all sequences  $$s=\big((g_0,x_0,y_0),\dots,(g_l,x_l,y_l)\big)\in (M_0\times X\times X)^{<\w}$$such that
$$\varphi=\sum_{i=0}^l(\delta^{g_i}_{x_i}-\delta^{g_i}_{y_i}).$$ 
Sequences $s$ that belong to the set $R_X(\varphi)$ will be called {\em $\varphi$-representing sequences}. For each sequence $s=\big((g_0,x_0,y_0),\dots,(g_l,x_l,y_l)\big)\in R_X(\varphi)$ put $$\Sigma d^1_X(s):=\sum_{i=0}^ld_X(x_i,y_i).$$

In the family $R_{X}(\varphi)$ consider two subfamilies: $$R^\subset_X(\varphi):=\big\{\big((g_i,x_i,y_i)\big){}_{i=0}^l\in R_X(\varphi):\bigcup_{i=0}^l\{x_i,y_i\}\subseteq\supp(\varphi)\big\}$$
and
$$R^\cap_X(\varphi):=\big\{\big((g_i,x_i,y_i)\big){}_{i=0}^l\in R_X(\varphi):\{x_i,y_i\}\cap \{x_j,y_j\}=\emptyset\mbox{ for any distinct $i,j$}\big\}.$$
It is easy to see that $$R^\cap_X(\varphi)\subseteq R^\subset_X(\varphi)\subseteq R_X(\varphi)$$
and the families $R^\subset_X(\varphi)\subseteq R_X(\varphi)$ are not empty if and only if $\varphi\in F_0X$ (the latter fact can be proved by induction on $|\supp(\varphi)|$). 

It will be convenient to assume that the zero function $o\in F_0X$ is represented by the empty sequence $(\;)\in R^\cap_X(o)\subseteq R^\subset_X(o)\subseteq R_X(o)$.

For any $\varphi\in F_0X$ consider the (finite or infinite) numbers
$$
\begin{aligned}
&\|\varphi\|_{FX}:=\inf(\{\Sigma d^1_X(s):s\in R_X(\varphi)\}\cup\{\infty\}),\\
&\|\varphi\|^\subset_{FX}:=\inf(\{\Sigma d^1_X(s):s\in R_X^\subset(\varphi)\}\cup\{\infty\}),\\
&\|\varphi\|^\cap_{FX}:=\inf(\{\Sigma d^1_X(s):s\in R_X^\cap(\varphi)\}\cup\{\infty\}).
\end{aligned}
$$
The inclusions $R^\cap_X(\varphi)\subseteq R^\subset_X(\varphi)\subseteq R_X(\varphi)$ imply the inequalities
$$\|\varphi\|\le\|\varphi\|^\subset\le\|\varphi\|^\cap.$$ 

Also for any $a,b\in FX$, in the family $L_{FX}(a,b)$ of $(a,b)$-connecting chains, consider the subfamily
$L^\subset_{FX}(a,b)$ consisting of $(a,b)$-connecting chains $s=\big((a_i,f_i,g_i)\big){}_{i=0}^k$ such that $\bigcup_{i=0}^lf_i[\dom(f_i)]\cup g_i[\dom(g_i)])\subseteq\supp(a)\cup \supp(b)$. For such chain $s$ we put
$$\Sigma d^1_X(s)=\sum_{i=0}^ld^1_{X^{\!<\!\w}}(f_i,g_i)=\sum_{i=0}^l\sum_{x\in \dom(f_i)}d_X(f_i(x),g_i(x)).$$
For every $a,b\in FX$ consider the number
$$d^\subset_{FX}(a,b)=\inf(\{\Sigma d^1_X(s):s\in L^\subset_{FX}(a,b)\}\cup\{\infty\})$$
and observe that $d^1_{FX}(a,b)\le d^\subset_{FX}(a,b)$ for any $a,b\in FX$.

\begin{theorem}\label{t:Gra} For any elements $a,b\in F_0X$ 
\begin{enumerate}
\item[\textup{1)}] 
$d^1_{FX}(a,b)=\|a-b\|_{FX}\le d_{FX}^\subset(a,b)\le\|a-b\|_{FX}^{\subset}$;
\item [\textup{2)}] if $|M_0|=2$, then 
$$\check d^1_{FX}(a,b)=d^1_{FX}(a,b)=d^\subset_{FX}(a,b)=\|a-b\|_{FX}=\|a-b\|_{FX}^\subset=\|a-b\|_{FX}^\cap.$$
\end{enumerate}
\end{theorem}

\begin{proof} 1. First we show that $d^1_{FX}(a,b)\le \|a-b\|_{FX}$ and $d^\subset(a,b)\le\|a-b\|_{FX}^\subset$. To derive a contradiction, assume that $\|a-b\|_{FX}<d^1_{FX}(a,b)$ and find  an $(a-b)$-representing sequence $\big((g_i,x_i,y_i)){}_{i=0}^l\in R_X(a-b)$ such that $$\sum_{i=0}^ld_X(x_i,y_i)<d^1_{FX}(a,b).$$ Let $n_0:=2l+2$, $a_0:=\sum_{i=0}^l(\delta^{g_i}_{2i}-\delta^{g_i}_{2i+1})\in F_0(n_0)$, and $f_0,g_0:2n\to X$ be two functions defined by $f_0(2i)=f_0(2i+1)=x_i=g_0(2i)$ and $g_0(2i+1)=y_i$ for every $i\in\{0,\dots,l\}$. It is easy to see that $Ff_0(a_0)=o$ and $Fg_0(a_0)=\sum_{i=0}^l(\delta^{g_i}_{x_i}-\delta^{g_i}_{y_i})=a-b$. Then $\big((a_0,f_0,g_0)\big)\in L_{FX}(o,a-b)$ is an $(o,a-b)$-linking chain, so
$$
\begin{aligned}
d^1_{FX}(a,b)&=d^1_{FX}(a-b,o)\le d^1_{X^{\!<\!\w}}(f_0,g_0)=\sum_{j\in 2l+2}d_X(f_0(j),g_0(j))=\\
&=\sum_{i=0}^l\big(d_X(f_0(2i),g_0(2i))+d_X(f_0(2i+1),g_0(2i+1))\big)=\\
&=\sum_{i=0}^l\big(d_X(x_i,x_i)+d_X(x_i,y_i)\big)<d^1_{FX}(a,b),
\end{aligned}
$$
which is a desired contradiction showing that $d^1_{FX}(a,b)\le \|a-b\|_{FX}$.
Observe that the inclusion $\big((g_i,x_i,y_i)\big){}_{i=0}^l\in R^\subset_{X}(a-b)$ implies that $\big((a_0,f_0,g_0)\big)\in L_{FX}^\subset(a,b)$, which yields the inequality $d^\subset_{FX}(a,b)\le\|a-b\|_{FX}^\subset$. The inequality $d^1_{FX}(a,b)\le d^\subset_{FX}(a,b)$ follows from the inclusion $L^\subset_{FX}(a,b)\subseteq L_{FX}(a,b)$.
\smallskip

Next, we prove that $\|a-b\|_{FX}\le d^1_{FX}(a,b)$. To derive a contradiction, assume that $d^1_{FX}(a,b)<\|a-b\|_{FX}$ and find an $(a,b)$-linking chain $\big((a_i,f_i,g_i)\big){}_{i=0}^l\in L_{FX}(a,b)$ such that $\sum_{i=1}^ld^1_{X^{\!<\!\w}}(f_i,g_i)<\|a-b\|_{FX}$. For every $i\in\{0,\dots,l\}$ let $n_i:=\dom(f_i)=\dom(g_i)$. Consider the  sequence $(m_i)_{i=0}^{l+1}$ of numbers $m_i=\sum_{j<i}n_j$ and put $m:=m_{l+1}$. For every number $j<m$, find (unique) numbers $i_j\in \{0,\dots,l\}$ and  $r_j\in n_{i_j}$ such that $j=m_{i_j}+r_j$. Then put $(g_j,x_j,y_i):=(a_{i_j}(r_j),f_{i_j}(r_j),g_{i_j}(r_j))$. We
claim that the sequence $\big((g_j,x_j,y_j)\big){}_{j=0}^{m-1}$ is $(a-b)$-representing.

Indeed, 
$$
\begin{aligned}
\sum_{j\in m}(\delta^{g_j}_{x_j}-\delta^{g_j}_{y_j})&=\sum_{i=0}^l\sum_{r\in n_i}\big(\delta^{a_i(r)}_{f_i(r)}-\delta^{a_i(r)}_{g_i(r)}\big)=\sum_{i=0}^l\big(Ff_i(a_i)-Fg_i(a_i)\big)=\\
&=Ff_0(a_0)-Fg_l(a_l)+\sum_{i=1}^l\big(Ff_i(a_i)-Fg_{i-1}(a_i)\big)=a-b+0.
\end{aligned}
$$
Then $$\|a-b\|_{FX}\le\sum_{j\in m}d_X(x_j,y_j)=\sum_{i=0}^l\sum_{r\in n_i}d_X(f_i(r),g_i(r))=\sum_{i=0}^ld^1_{X^{\!<\!\w}}(f_i,g_i)<\|a-b\|_{FX},$$which is a desired contradiction that completes the proof of $d^1_{FX}(a,b)=\|a-b\|_{FX}\le d^\subset_{FX}(a,b)\le\|a-b\|_{FX}^\subset$.
\smallskip

2. Now assume that $|M_0|=2$ and let $g$ be the unique non-zero element of the group $M_0$. For every $\varphi\in F_0X$, the inclusions $R^\cap_X(\varphi)\subseteq R^\subset_X(\varphi)\subseteq R_X(\varphi)$ imply that $\|\varphi\|_{FX}\le\|\varphi\|_{FX}^\subset\le\|\varphi\|_{FX}^\cap$. Assuming that $\|\varphi\|_{FX}<\|\varphi\|_{FX}^\cap$, find a sequence $s=\big((g_i,x_i,y_i)\big){}_{i=0}^l\in R_X(\varphi)$ such that $\sum_{i=0}^ld_X(x_i,y_i)<\|\varphi\|_{FX}^\cap$. We can assume that the length $l+1$ of this sequence is the smallest possible. It follows from $\|\varphi\|_{FX}<\|\varphi\|_{FX}^\cap$ that $\varphi\ne o$ and hence $l\ge 1$. The minimality of $l$ ensures that for every $i\le l$ the element $g_i$ is non-zero and hence is equal to the unique non-zero element $g$ of the 2-element group $M_0$.
Also, the minimality of $l$ implies that $\{x_i,y_i\}\ne \{x_j,y_j\}$ for any $i<j$. Indeed, assuming that $\{x_i,y_i\}=\{x_j,y_j\}$ for some numbers  $i<j\le l$, we can see that $\delta^{g_i}_{x_i}-\delta^{g_i}_{y_i}=\delta^{g_j}_{x_j}-\delta^{g_j}_{y_j}$ and hence $(\delta^{g_i}_{x_i}-\delta^{g_i}_{y_i})+(\delta^{g_j}_{x_j}-\delta^{g_j}_{y_j})=o$, which allows us to remove the triples $(g_i,x_i,y_i)$ and $(g_j,x_j,y_j)$ for the sequence $s$ without breaking the $\varphi$-representation property. But this contradicts the minimality of $l$. If for some distinct $i,j$ the doubletons $\{x_i,y_i\}$ and $\{x_j,y_j\}$ have a common point, then we can replace two triples $(g_i,x_i,y_i)$ and $(g_j,x_j,y_j)$ in the sequence $s$  by one triple $(g,x,y)$ where $\{x,y\}$ coincides with the symmetric difference of the sets $\{x_i,y_i\}$ and $\{x_j,y_j\}$. Here we observe that $(\delta^{g_i}_{x_i}-\delta^{g_i}_{y_i})+(\delta^{g_j}_{x_j}-\delta^{g_j}_{y_j})=\delta^g_x-\delta^g_y$ and $d_X(x,y)\le d_X(x_i,y_i)+d_X(x_j,y_j)$. So, we can  make the sequence $s$ shorter, which contradicts the minimality of $l$. Now we see that the sequence $s$ belongs to the family $R^\cap_X(\varphi)$, which implies that $$\|\varphi\|_{FX}^\cap\le\sum_{i=0}^ld_X(x_i,y_i)<\|\varphi\|_{FX}^\cap,$$
and is a desired contradiction showing that $\|\varphi\|_{FX}=\|\varphi\|_{FX}^\subset=\|\varphi\|_{FX}^\cap$. Combining these equalities with the (in)equality $d^1_{FX}(a,b)=\|a-b\|_{FX}\le d_{FX}^\subset(a,b)\le\|a-b\|_{FX}^\subset$, proved in the first part of the theorem and the inequality $\check d^1_{FX}(a,b)\le d^1_{FX}(a,b)$ proved in Theorem~\ref{t:mainU}(3), we conclude that
$$\check d^1_{FX}(a,b)\le d^1_{FX}(a,b)=d^\subset_{FX}(a,b)=\|a-b\|_{FX}=\|a-b\|_{FX}^\subset=\|a-b\|_{FX}^\cap.$$
To complete the proof of the theorem, it remains to show that $\|a-b\|_{FX}\le \check d^1_{FX}(a,b)$, which will follow from the invariantness of the distance $\check d^1_{FX}$ as soon as we check that $\|\varphi\|_{FX}\le \check d^1_{FX}(\varphi,o)$ for every $\varphi\in F_0X$.

Assuming that the latter inequality does not hold, we conclude that the set $$\Phi:=\{\varphi\in F_0X:\check d^1_{FX}(\varphi,o)<\|\varphi\|_{FX}\}$$ is not empty. Let $\varphi\in\Phi$ be an element whose support $S=\supp(\varphi)$ has the smallest possible cardinality. Since $\varphi\in F_0X$, the cardinality $|S|$  is even (and non-zero as $\|\varphi\|_{FX}>0$). We claim that the restriction $d_X{\restriction}_{S\times S}$ is an $\infty$-metric. Indeed, assuming that $S$ contains two distinct points $x,y$ with $d_X(x,y)=0$, we can consider the function $\delta^g_x-\delta^g_y\in F_0X$ and conclude that $0\le \check d^1_{FX}(\delta^g_x,\delta^g_y)\le d^1_{FX}(\delta^g_x,\delta^g_y)\le d_X(x,y)=0$. Then the function $\psi=\varphi+\delta^g_x-\delta^g_y$ has support $\supp(\psi)=S\setminus\{x,y\}$ of cardinality $|\supp(\psi)|<|S|=|\supp(\varphi)|$ and belongs to the set $\Phi$ since
$$
\check d^1(\psi,o)\le \check d^1(\delta^g_x-\delta^g_y,o)+\check d^1(\varphi,o)=\check d^1(\varphi,o)<\|\varphi\|_{FX}\le \|\psi\|_{FX}+\|\delta^g_x-\delta^g_y\|_{FX}=\|\psi\|_{FX},
$$
but this contradicts the choice of $\varphi$. This contradiction shows that the restriction $d_X{\restriction}_{S\times S}$ is an $\infty$-metric.


As we already know, 
$$\|F\mathsf e_X(\varphi)\|_{FEX}^\subset=d^1_{FEX}(F\mathsf e_X(\varphi),o)=\check d_{FX}(\varphi,o)<d^1_{FX}(\varphi,o)=\|\varphi\|_{FX}.$$ By the definition of $\|F\mathsf e_X(\varphi)\|^\subset_{FEX}$, there is a sequence $\big((g_i,x_i,y_i)\big){}_{i=0}^l\in R^\subset_{FEX}(F\mathsf e_X(\varphi))$ such that $\sum_{i=0}^l(\delta^g_{x_i}-\delta^g_{y_i})=F\mathsf e_X(\varphi)$ and $\sum_{i=0}^ld_{EX}(x_i,y_i)<\|\varphi\|_{FX}$. The definition of $R^\subset_{EX}(F\mathsf e_X(\varphi))$ ensures that $$\bigcup_{i=0}^l\{x_i,y_i\}\subseteq \supp(F\mathsf e_X(\varphi))\subseteq \mathsf e_X[\supp(\varphi)]=\mathsf e_X[S].$$
For every $i\in\{0,\dots,l\}$ choose points $x_i',y_i'\in S$ such that $\mathsf e_X(x_i')=x_i$ and $\mathsf e_X(y_i')=y_i$. The injectivity of $\mathsf e_X$ on the set $S=\supp(\varphi)$ and the equality $F\mathsf e_X(\varphi)=\sum_{i=0}^l(\delta^g_{x_i}-\delta^g_{y_i})$ imply that $\varphi=\sum_{i=0}^l(\delta^g_{x_i'}-\delta^g_{y'_i})$. Then $\big((g_i,x'_i,y'_i)\big){}_{i=0}^l\in R_{FX}(\varphi)$ and 
$$\|\varphi\|_{FX}\le\sum_{i=0}^l d_X(x_i',y_i')=\sum_{i=0}^l d_{EX}(x_i,y_i)<\|\varphi\|_{FX},$$which is a desired contradiction that completes the proof.
\end{proof}  

It should be mentioned that the assumption $|M_0|=2$ in Theorem~\ref{t:Gra} is essential and cannot be replaced by any weaker assumption.

\begin{example} Let $\{z_1,z_2,z_3\}$ be an equilateral triangle with unit sides in the Euclidean plane and $z_0$ its center. Consider the distance space $X=\{z_0,z_1,z_2,z_3\}$ endowed with the metric $d_X$, inherited from the plane.
If $|M_0|>2$, then there exists a function $\varphi\in F_0X$ such that $\varphi(z_0)=0$ and $\varphi(z_i)\ne 0$ for all $i\in\{1,2,3\}$. The function $\varphi$  has $\supp(\varphi)=\{z_1,z_2,z_3\}$. It can be shown that $$d^1_{FX}(\varphi,o)=\|\varphi\|_{FX}=\sqrt{3}<2=d^\subset_{FX}(\varphi,0)=\|\varphi\|_{FX}^\subset<\|\varphi\|_{FX}^\cap=\infty.$$\end{example}

\begin{remark} If the monoid $M$ is a cyclic group of order $n\le \w$ with generator $g$, then the group $FX=F(X,M)$ can be identified with the free Abelian group of exponent $n$ over $X$. Each distance $d_X$ on $X$ induces the {\em Graev distance} $d^g_{FX}$ on $FX$, defined by
$$d^g_{FX}(a,b)=\inf\Big(\Big\{\sum_{i=0}^ld_X(x_i,y_i):x_0,y_0,\dots,x_l,y_l\in X,\;a-b=\sum_{i=0}^l(\delta^g_{x_i}-\delta^g_{y_i})\Big\}\cup\{\infty\}\Big).$$
Theorem~\ref{t:Gra} implies that $d^1_{FX}\le d^g_{FX}$. If $n\le 3$, then  $d^1_{FX}=d^g_{FX}$. If $n>3$ and $X=\{x,y\}$ with $d_X(x,y)=1$, then $d^g_{FX}(\delta^{2g}_x,\delta^{2g}_y)=2>1=d^1_{FX}(\delta^{2g}_x,\delta^{2g}_y)$, so $d^1_{FX}\ne d^g_{FX}$. More information on Graev distances on free Abelian (Boolean) groups can be found in  \cite[\S7.2]{AT}, \cite{Graev} (and \cite[\S4]{Sipa}).   
\end{remark}

\begin{remark} Theorem~\ref{t:Gra}(2) suggests an algorithm for calculating the distances $\check d^1_{FX}=d^1_{FX}$ in case of $|G|=2$. However, for calculating the norm $\|\varphi\|_{FX}$ of an element $\varphi\in F_0X$ with support of cardinality $|\supp(\varphi)|=2n$, the brute force algorithm requires checking $\frac{(2n)!}{2^nn!}$ cases, which yields superexponential algorithmic complexity.
\end{remark}   

\begin{problem} Is there a simpler (desirably, polynomial time) algorithm for  calculating the distances $\check d^1_{FX}=d^1_{FX}$ on the group $FX$ in case of the two-element group $|M|=2$?
\end{problem}

\begin{problem} Is there any reasonable algorithm for calculating the distances $\check d^1_{FX}$ and $d^1_{FX}$ on $FX$ for arbitrary finite group $G$ and arbitrary metric space $(X,d_X)$ with computable metric?
\end{problem}

\section{The metrizations of the hyperspace functor}\label{s:hyperf}

If $M$ is a two-element semilattice, then the functor $F=F(-,M)$, considered in the preceding section, is isomorphic to the hyperspace functor $H$ assigning to each set $X$ the semilattice $[X]^{<\w}$ of finite subsets of $X$ and to each function $f:X\to Y$ between sets the function $Hf:HX\to HY$, $Hf:a\mapsto f[a]$.

We recall that a {\em semilattice} is a commutative semigroup $S$ whose any element $x\in S$ is {\em idempotent} (i.e. $x+x=x$). Each 2-element semilattice is isomorphic to the semilattice $2=\{0,1\}$ endowed with the operation of maximum. The hyperspace $HX$ is endowed with the semilattice operation of union.

Identifying each finite set $A\subseteq X$ with its characteristic function $\chi_A:X\to 2=\{0,1\}$, we can identify the semilattice $HX$ with the monoid $F(X,2)$. This determines an isomorphism of the functors $H$ and $F(-,2)$.

It is easy to see that the functor $H$ preserves supports; moreover, $\supp(a)=a$ for every set $X$ and element $a\in HX$. 

The hyperspace functor $H$ has a metrization by the Hausdorff distance $d_{HX}$, defined in Section~~\ref{s:distance}.  The interplay between the distances $d_{HX}$, $d^p_{HX}$ and $\check d^p_{HX}$ is described in the following theorem. 

From now on, we assume that $p\in[1,\infty]$ and $(X,d_X)$ is a distance space.

\begin{theorem}\label{t:hyper} 
$d_{HX}=\check d^\infty_{HX}=d^\infty_{HX}\le \check d^p_{HX}\le d^p_{HX}.$\newline  Moreover, the distances $d^p_{HX}$ and $\check d^p_{FX}$ are subinvariant in the sense that
$$d^p_{HX}(a\cup c,b\cup c)\le d^p_{HX}(a,b)\mbox{ and }\check d^p_{HX}(a\cup c,b\cup c)\le \check d^p_{HX}(a,b)$$ for every $a,b,c\in HX$.
\end{theorem}

\begin{proof} The subinvariance of the distances $d_{HX}$ and $\check d^p_{HX}$ follows from Proposition~\ref{p:subinv} and Corollary~\ref{c:subinv}.

By Theorem~\ref{t:shukel}, $d_{HX}\le d^\infty_{HX}$. Next, we prove that  $d^\infty_{HX}(a,b)\le d_{HX}(a,b)$ for any  $a,b\in HX$. This inequality is trivial if $a=b$ or $d_{HX}(a,b)=\infty$.
So, assume that $a\ne b$ and $d_{HX}(a,b)<\infty$. In this case $a,b$ are non-empty finite subsets of $X$. By the definition of $d_{HX}(a,b)$, there are two functions $p:a\to b$ and $q:b\to a$ such that 
$$\max\{\max_{x\in a}d_X(x,p(x)),\max_{y\in b}d_X(y,q(y))\}=d_{HX}(a,b).$$ 
Fix any bijective functions $\alpha:|a|\to a$ and $\beta:|b|\to b$. Let $n_0=|a|+|b|$ and consider the functions $f_0:n_0\to a$ and $g_0:n_0\to b$ defined by the formulas
$$f_0(i)=\begin{cases}
\alpha(i),&\mbox{if $0\le i<|a|$};\\
q{\circ}\beta(i-|a|),&\mbox{if $|a|\le i<n_0$};
\end{cases}
\quad\mbox{and}\quad  
g_0(i)=
\begin{cases}
p\circ \alpha(i),&\mbox{if $0\le i<|a|$};\\
\beta(i-|a|),&\mbox{if $|a|\le i<n_0$}.
\end{cases}
$$
For the element $a_0:=n_0\in Hn_0$, we get $Hf_0(a_0)=f_0[n_0]=a$ and $Hg_0(a_0)=g_0[n_0]=b$, which means that $\big((a_0,f_0,g_0)\big)\in L_{HX}(a,b)$ is an $a,b$-linking chain. Then
$$
\begin{aligned}
d^\infty_{HX}(a,b)&\le d^\infty_{X^{\!<\!\w}}(f_0,g_0)=\max_{x\in n_0}d_X(f_0(x),g_0(x))=\\
&=\max\{\max_{0\le i<|a|}d_X(\alpha(i),p{\circ}\alpha(i)),\max_{|a|\le i<n_0}d_X(q{\circ}\beta(i),\beta(i))\}\le\\
&\le\max\{\max_{x\in a}d_X(x,p(x)),\max_{y\in b}d_X(q(y),y)\}=d_{HX}(a,b).
\end{aligned}
$$
Now we see that $d_{HX}=d^\infty_{HX}$ (for any distance space $(X,d_X)$). In particular, $d_{HEX}=d^\infty_{HEX}$.
Next, we prove that $d_{HX}(a,b)=\check d^\infty_{HX}(a,b)=d^\infty_{HX}(a,b)$ for any elements $a,b\in HX$. 
Looking at the definition of the Hausdorff distance, we can see that $d_{HX}(a,b)=d_{HEX}(H\mathsf e_X(a),H\mathsf e_X(b))$. Now the definition of the distance $\check d_{HX}$ and Theorem~\ref{t:mainU}(3) ensure that
\begin{multline*}
d^\infty_{HX}(a,b)=d_{HX}(a,b)=d_{HEX}(H\mathsf e_X(a),H\mathsf e_X(b))=\\
d^\infty_{HEX}(H\mathsf e_X(a),H\mathsf e_X(b))=\check d^\infty_{HX}(a,b)\le d^\infty_{HX}(a,b),
\end{multline*}
 which implies  the desired equalities  $d_{HX}(a,b)=\check d^\infty_{HX}(a,b)=d^\infty_{HX}(a,b)$. Applying  Theorem~\ref{t:mainU}(2,3), we conclude that 
 $$d_{HX}(a,b)=\check d^\infty_{HX}(a,b)=d^\infty_{HX}(a,b)\le\check d^p_{HX}(a,b)\le d^p_{HX}(a,b).$$
\end{proof}

For every $n\in\IN$ consider the subfunctor $H_n$ of the functor $H$, assigning to each set $X$ the subfamily $H_nX:=\{a\in HX:|a|\le n\}$. It is easy to see that the functor $H_n$ has degree $\le n$ and preserves supports. The properties of its metizations are desribed in the following proposition.

\begin{proposition}\label{p:Hn} 
$d_{HX}\le \check d^p_{H_nX}\le d^p_{H_nX}\le n^{\frac1p}{\cdot}d^\infty_{H_nX}\le 3\cdot n^{\frac1p}{\cdot}d_{HX}.$
\end{proposition}

\begin{proof} The inequality $d_{HX}\le \check d^p_{H_nX}$ follows from Theorems \ref{t:hyper} and \ref{t:subf}. The inequalities $\check d^p_{H_nX}\le d^p_{H_nX}\le n^{\frac1p}{\cdot}d^\infty_{H_nX}$ follow from Theorems~\ref{t:mainU}(3) and \ref{t:main}(4). It remains to prove that $d^\infty_{H_nX}(a,b)\le 3\cdot d_{HX}(a,b)$ for any sets $a,b\in H_n X$. This inequality holds trivially if $a=b$ or $d_{HX}(a,b)=\infty$. So, assume that $a\ne b$ and $d_{HX}(a,b)<\infty$. In this case the sets $a,b$ are non-empty. Let $\gamma:a'\to b'$ be a maximal bijective function between subsets $a'\subseteq b$ and $b'\subseteq b$ such that $d_X(x,\gamma(x))\le d_{HX}(a,b)$ for every $x\in a'$. By the maximality, the function $\gamma$ has an extension $p:a\to b'$ such that $\max_{x\in a}d_X(x,p(x))\le d_{HX}(a,b)$, and $\gamma^{-1}$ has an extension $q:b\to a'$ such that  $\max_{y\in b}d_X(y,q(y))\le d_{HX}(a,b)$.

Now we define an $(a,b)$-linking chain $\big((a_0,f_0,g_0),(a_1,f_1,g_1)(a_2,f_2,g_2)\big)\in L_{H_nX}(a,b)$ such that $d^\infty_{H_n}(a,b)\le \sum_{i=0}^2d^\infty_{X^{\!<\!\w}}(f_i,g_i)\le 3{\cdot}d_{HX}(a,b)$.

Let $n_0=|a|$, $a_0=n_0\in H_nn_0$, $f_0:|a|\to a$ be any bijective function and $g_0=p\circ f_0$.

Let $n_1=|a'|=|b'|$, $a_1=n_1\in H_nn_1$, $f_1:n_1\to b'$ be any bijection and $g_1=\gamma^{-1}\circ f_1$.

Let $n_2=|b|$, $a_2=n_2\in H_nn_2$, $g_2:n_2\to b$ be any bijection and $f_2=q\circ g_2$. 

Observe that $a=f_0[n_0]$, $g_0[n_0]=p[a]=b'=f_1[n_1]$, $g_1[n_1]=a'=q[b]=f_2[n_2]$, $g_2[n_2]=b$, which means that $\big((a_i,f_i,g_i)\big){}_{i=0}^2\in L_{H_nX}(a,b)$ and then 
$$d^\infty_{H_nX}(a,b)\le\sum_{i=0}^2d^\infty_{X^{\!<\!\w}}(f_i,g_i)\le 3\cdot d_{HX}(a,b).$$
\end{proof}

\begin{problem} Can the constant $3$ in Proposition~\ref{p:Hn} be replaced by a smaller constant?
\end{problem}

The distance $d^1_{HX}$ admits an equivalent graph-theoretic description, similar to the description of the metric $d^1_{FX}$ in Theorem~\ref{t:Gra}. It is defined with the help of graphs, see \cite{MO}, \cite{MO2}.  

By a {\em graph} we understand a pair $\Gamma=(V,E)$ consisting of a finite set $V$ of vertices and a finite set $E$ of edges. Each edge $e\in E$ is a non-empty subset of $V$ of cardinality $|e|\le 2$. For a graph $\Gamma=(V,E)$, the {\em connected component} of a vertex $v\in V$ is the set $\Gamma_v$ of all vertices $u\in V$ for which there exists a sequence $v=v_0,\dots,v_n=u$ such that $\{v_{i-1},v_i\}\in E$ for every $i\in\{1,\dots,n\}$. 

By a {\em graph in the distance space} $(X,d_X)$ we understand any graph $\Gamma=(V,E)$ with $V\subseteq X$. In this case we can define the length $\length(\Gamma)$ of $\Gamma$ by the formula
$$\length(\Gamma)=\sum_{e\in E}\bar d_X(e)$$where $\bar d_X(e)=\max\{d_X(x,y):x,y\in e\}$.

For two finite sets $a,b\subseteq X$ let ${\mathbf\Gamma}(a,b)$ be the family of  graphs $\Gamma=(V,E)$ in $X$ such that $a\cup b\subseteq V$ and each connected component of $\Gamma$ intersects both sets $a$ and $b$. 
Observe that the family $\Gamma(a,b)$ contains the complete graph on the set $a\cup b$ and hence is not empty.

\begin{theorem} For any elements $a,b\in HX$ we have
$$d_{HX}(a,b)\le  d^1_{HX}(a,b)=\inf\big(\{\ell(\Gamma):\Gamma\in \mathbf\Gamma(a,b)\}\cup \{\infty\}\big).$$
\end{theorem}

\begin{proof} The inequality $d_{HX}(a,b)\le d^1_{HX}(a,b)$ was proved in Proposition~\ref{t:hyper}.

Next, we prove that $\hat d_{HX}(a,b)\le d^1_{HX}(a,b)$, where $$\hat d_{HX}(a,b):=\inf\big(\{\ell(\Gamma):\Gamma\in\mathbf \Gamma(a,b)\}\cup\{\infty\}\big).$$
 To derive a contradiction, assume that $d^1_{HX}(a,b)<\hat d_{HX}(a,b)$ and find an $(a,b)$-linking chain $\big((a_i,f_i,g_i)\big){}_{i=0}^l\in L_{HX}(a,b)$ such that $\sum_{i=0}^ld^1_{X^{\!<\!\w}}(f_i,g_i)<\hat d_{HX}(a,b)$.
For every $i\in\{0,\dots,l\}$ let $n_i=\dom(f_i)=\dom(g_i)$. Consider the graph $\Gamma=(V,E)$ with $$V=\bigcup_{i=0}^l(f_i[a_i]\cup g_i[a_i])\quad\mbox{and}\quad E=\bigcup_{i=0}^l\big\{\{f_i(x),g_i(x)\}:x\in a_i\big\}.$$
Observe that
$$\length(\Gamma)\le\sum_{i=0}^l\sum_{x\in n_i}d_X(f_i(x),g_i(x))=\sum_{i=0}^ld^1_{X^{\!<\!\w}}(f_i,g_i)<\hat d_{HX}(a,b).$$
We claim that $\Gamma\in\mathbf\Gamma(a,b)$.
Observe that $a\cup b=f_0[a_0]\cup g_l[a_l]\subseteq V$. Next, we check that for every $x\in V$ the connected component $\Gamma_x$ of the graph $\Gamma$ intersects the sets $a$ and $b$. Find $i\in \{0,\dots,l\}$ such that $x\in f_i[a_i]\cup g_i[a_i]$. Then $x\in \{f_i(t_i),g_i(t_i)\}$ for some $t_i\in a_i$. For the point $x_i:=g_i(t_i)\in g_i[a_i]=f_{i+1}[a_{i+1}]$ we can find a point $t_{i+1}\in a_{i+1}$ with $x_i=f_{i+1}(t_{i+1})$ and consider the point $x_{i+1}=g_{i+1}(t_{i+1})$. Continuing by induction, we can construct sequences of points $(x_j)_{j=i}^{l+1}$ and $(t_j)_{j=i}^l\in \prod_{j=i}^l a_j$ such that $\{x,x_i\}\in E$, $x_i=f_i(t_i)$, $x_{l+1}=g_l(t_l)\in g_l[a_l]=b$ and $g_{j-1}(t_{j-1})=x_{j}=f_{j}(t_j)$ for every $j\in\{i+1,\dots,l\}$. By the definition of the set $E$, for every $j\in\{i,\dots,l\}$ the doubleton $\{x_{j},x_{j+1}\}=\{f_j(t_j),g_j(t_j)\}$ is an egde of the graph $\Gamma$. Then the chain $x,x_i,\dots,x_{l+1}$ witnesses that $x_{l+1}\in \Gamma_x\cap b$. By analogy we can prove that the intersection $\Gamma_x\cap a$ is not empty.
Therefore, $\Gamma\in\mathbf \Gamma(a,b)$ and $\hat d_{HX}(a,b)\le\length(\Gamma)<\hat d_{HX}(a,b)$. This contradiction shows that $\hat d_{HX}(a,b)\le d^1_{HX}(a,b)$.
\vskip3pt

Next, we prove that $d^1_{HX}(a,b)\le \hat d_{HX}(a,b)$. Assuming that $\hat d_{HX}(a,b)<d^1_{HX}(a,b)$, we can find a graph $\Gamma\in\mathbf \Gamma(a,b)$ such that $\ell(\Gamma)<d^1_{HX}(a,b)$. By Lemma~\ref{l:graph} (proved below), there exists an $(a,b)$-linking chain $s\in L_{FX}(a,b)$ such that $\Sigma d^1_X(s)\le\ell(\Gamma)$. Then $d^1_{HX}(a,b)\le\Sigma d^1_X(s)\le\ell(\Gamma)<d^1_{HX}(a,b)$, which is a desired contradiction showing that $d^1_{HX}=\hat d_{HX}$.
\end{proof}

\begin{lemma}\label{l:graph} For any finite sets $a,b\subseteq X$ and graph $\Gamma\in\mathbf \Gamma(a,b)$, there exists an $(a,b)$-linking chain $s\in L_{FX}(a,b)$ such that $\Sigma d^1_X(s)\le\ell(\Gamma)$.
\end{lemma}

\begin{proof}  It suffices to prove that for every $n\in\IN$ the following statement $(*_n)$ holds:
\begin{itemize}
\item[$(*_n)$]  {\em for any finite sets $a,b\subseteq X$ and graph $\Gamma=(V,E)\in\mathbf \Gamma(a,b)$ with $|V|\le n$, there exists an $(a,b)$-linking chain $s\in L_{FX}(a,b)$ such that $\Sigma d^1_X(s)\le\ell(\Gamma)$.}
\end{itemize}

To see that $(*_1)$ is true, take any finite sets $a,b\subseteq X$ and any graph $\Gamma=(V,E)\in\mathbf \Gamma(a,b)$ with $|V|=1$. In this case, the definition of $\mathbf \Gamma(a,b)\ne\emptyset$ implies that either $a,b$ are both empty or $a=b=V$. In both cases we get $a=b$. So, we can find an $(a,b)$-linking chain $s\in L_{FX}(a,b)$ such that $\Sigma d^1_X(s)=0\le\ell(\Gamma)$.

Now assume that for some $n\in\IN$ the statement $(*_n)$ is true. Take any finite sets $a,b\in \IN$ and a graph $\tilde \Gamma=(\tilde V,\tilde E)\in\mathbf \Gamma(a,b)$ with  $|V|=n+1$. If $a=b$, then we can find an $(a,b)$-linking chain $s\in L_{FX}(a,b)$ with $\Sigma d^1_X(s)=0\le\ell(\Gamma)$ and finish the proof. So, we assume that $a\ne b$.

Let $\Gamma=(V,E)$ be any minimal subgraph of the graph $\tilde \Gamma$ such that $\Gamma\in\mathbf \Gamma(a,b)$. If $|V|\le n$, then by the inductive assumption, there exists an $(a,b)$-linking chain $s\in L_{FX}(a,b)$ such that $\Sigma d^1_X(s)\le\ell(\Gamma)\le\ell(\tilde\Gamma)$ and we are done. So, we assume that $|V|=n+1$. By the minimality of $\Gamma$, each connected component of $\Gamma$ is a tree whose pendant vertices belong to the set $a\cup b$. Since $\Gamma\in\mathbf \Gamma(a,b)$ and $a\ne b$, the graph $\Gamma$ has some non-trivial connected component. Take any pendant vertex $v$ of this non-trivial connected component and find a (unique) vertex $u\in V\setminus\{v\}$ such that $\{v,u\}\in E$. Consider the subgraph $\Gamma'=(V',E')$ of $\Gamma$ with $V':=V\setminus\{v\}$ and $E'=\{e\in E:v\notin e\}$. It follows that $|V'|=n$.
\smallskip

Three cases are possible.

1) If $v\in b\setminus a$, then let $b'=(b\setminus \{v\})\cup \{u\}$. Observe that  every path connecting the vertex $v$ with another vertex of $\Gamma$ includes the vertex $u$. Consequently, $\Gamma'\in\mathbf \Gamma(a,b')$.  By the inductive assumption, there exists an $(a,b')$-linking chain $s'=\big((a_i,f_i,g_i)\big){}_{i=0}^{k-1}\in L_{FX}(a,b')$ such that $\Sigma d^1_X(s')\le\ell(\Gamma')$. For this chain we have $f_i[a_i]=a$ and $g_{k-1}[a_i]=b'$. Let $n_{k}=|b|$. It follows that $b=b'\cup\{v\}$ or $b=(b'\setminus \{u\})\cup\{v\}$. 
\smallskip

1a) If $b=b'\cup\{v\}$, then choose any bijective map $g_{k}:n_{k}\to b$ such that $g_{k}(n_{k}-2)=u$ and $g_k(n_k-1)=v$. Next, consider the surjective map 
$f_k:n_k\to b'$ such that $f_k(i)=g_k(i)$ for all $i\le n_k-2$ and $f_k(n_k-1)=u$.
Let $a_k=n_k\in Hn_k$. Observe that $Hf_k(a_k)=f_k[n_k]=b'$ and $Hg_k(a_k)=g_k[n_k]=b$. Then $s=\big((a_i,f_i,g_i)\big){}_{i=0}^k$ is an $(a,b)$-linking chain such that $\Sigma d^1_X(s)=\Sigma d^1(s')+d_X(u,v)\le\ell(\Gamma')+d_X(u,v)=\ell(\Gamma)$.
\smallskip

1b) If $b=(b'\setminus\{u\})\cup\{v\}$, then choose any bijective map $g_{k}:n_{k}\to b$ such that $g_{k}(n_{k}-1)=v$. Next, consider the bijective map 
$f_k:n_k\to b'$ such that $f_k(i)=f_k(i)$ for all $i\le n_k-2$ and $f_k(n_k-1)=u$.
Let $a_k=n_k\in Hn_k$. Observe that $Hf_k(a_k)=f_k[n_k]=b'$ and $Hg_k(a_k)=g_k[n_k]=b$. Then $s=\big((a_i,f_i,g_i)\big){}_{i=0}^k$ is an $(a,b)$-linking chain such that $\Sigma d^1_X(s)=\Sigma d^1(s')+d_X(u,v)\le\ell(\Gamma')+d_X(u,v)=\ell(\Gamma)$.
\smallskip

2) The case $v\in a\setminus b$ can be considered by analogy with the case 1). 
\smallskip

3) It remains to consider the third case $v\in a\cap b$. Let $\Gamma_v=(V_v,E_v)$ be the connected component of the graph $\Gamma$ that contains the point $v$.
The third case has 4 subcases.
\smallskip

3a)  The sets $V_v\cap a\setminus\{v\}$ and $V_v\cap b\setminus\{v\}$ are not empty. In this case put $a'=a\setminus\{v\}$ and $b'=b\setminus\{v\}$ and observe that the  graph $\Gamma'$ belongs to the family $\mathbf \Gamma(a',b')$. Applying the statement $(*_n)$ to the graph $\Gamma'$, we can find an $(a',b')$-linking chain $\big((a_i',f'_i,g'_i)\big){}_{i\in k}\in L_{FX}(a',b')$ such that $\Sigma d^1_X(s)\le\ell(\Gamma')$. For every $i\in k$ let $n'_i=\dom(f_i')=\dom(g_i')$,  $n_i:=n_i'+1$, $a_i:=a_i'\cup\{n_i'\}\subseteq n_i$ and $f_i,g_i:n_i\to X$ be two maps such that $f_i{\restriction}_{n_i'}=f_i'$, $g_i{\restriction}_{n_i'}=g_i'$, and $f_i(n'_i)=g_i'(n_i)=v$. Then $s=\big((a_i,f_i,g_i)\big){}_{i\in k}\in L_{FX}(a,b)$ is an $(a,b)$-linking chain such that $\Sigma d^1_X(s)=\Sigma d^1_X(s')\le\ell(\Gamma')\le \ell(\Gamma)$.
\smallskip

3b) The sets $V_v\cap a\setminus\{v\}$ and $V_v\cap b\setminus\{v\}$ are empty.
In this case $a\cap V_v=\{v\}=b\cap V_v$. Consider the subgraph $\Gamma''=(V'',E'')$ of $\Gamma$ such that $V''=V\setminus V_v$ and $E''=\{e\in E:e\subseteq V'')$ and let $a'=a\setminus\{v\}$, $b'=b\setminus \{v\}$. It follows from $a\ne b$ that $a'\cup b'$ is not empty and so is the set $V''\supseteq a'\cup b'$. Observe that each connected component of the graph $\Gamma''$ is a connected component $C$ of the graph $\Gamma$ and hence it intersects both sets $a$ and $b$. Taking into account that $C\cap V_v=\emptyset$, we conclude that $C\cap a'\ne\emptyset\ne C\cap b'$, which means that $\Gamma''\in\mathbf \Gamma(a',b')$. 
Applying the inductive assumption to the graph $\Gamma''$, we can find an $(a',b')$-linking chain $s'=\big((a_i',f'_i,g_i')\big){}_{i\in k}\in L_{FX}(a',b')$ such that $\Sigma d^1_X(s')\le \ell(\Gamma'')$. For  $i\in k$ let $n'_i=\dom(f_i')=\dom(g_i')$,  $n_i:=n_i'+1$, $a_i:=a_i'\cup\{n_i'\}\subseteq n_i$ and $f_i,g_i:n_i\to X$ be two maps such that $f_i{\restriction}_{n_i'}=f_i'$, $g_i{\restriction}_{n_i'}=g_i'$, and $f_i(n'_i)=g_i'(n_i)=v$. Then $s=\big((a_i,f_i,g_i)\big){}_{i\in k}\in L_{FX}(a,b)$ is an $(a,b)$-linking chain such that $\Sigma d^1_X(s)=\Sigma d^1_X(s')\le\ell(\Gamma'')\le \ell(\Gamma)$.
 \smallskip
 
3c) The set $V_v\cap a\setminus\{v\}$ is not empty and $V_v\cap b\setminus\{v\}$ is  empty. Then for the sets $a'=a\setminus\{v\}$ and $b'=(b\setminus \{v\})\cup\{u\}$, the connected graph $\Gamma'$ belongs to the family $\mathbf\Gamma(a',b')$. By the inductive assumption, there exists an $(a',b')$-linking chain $s'=\big((a_i',f_i',g_i')\big){}_{i\in k}\in L_{FX}(a',b')$ such that $\Sigma d^1_X(s')\le\ell(\Gamma')$. For every $i\in k$ let $n_i'=\dom(f_i')=\dom(g'_i)$ and put $n_i:=n_i'+1$, $a_i:=a_i'\cup\{n_i'\}\in H n_i$. Let $f_i,g_i:n_i\to X$ be the functions such that $f_i{\restriction}_{n_i'}=f_i'$, $g_i{\restriction}_{n_i'}=g_i$, and $f_i(n_i')=v$, $g_i(n_i')=v$. Let $n_k=|b'|=|b|$, $a_k=n_k$, and $f_k:n_k\to b'$ be any bijection such that $f_k(n_k-1)=u$.
Let $g_k:n_k\to b$ be the bijective function such that $g_k(i)=f_k(i)$ for $i<n_k-1$ and $g_k(n_k-1)=v$. 
It is easy to see that $s:=\big((a_i,f_i,g_i)\big){}_{i=0}^k$ is an $(a,b)$-linking chain such that $\Sigma d^1_{FX}(s)=\Sigma d^1_{FX}(s')+d_X(u,v)\le \ell(\Gamma')+d_X(u,v)=\ell(\Gamma)$.
\smallskip

3d) The case $V_v\cap a\setminus\{v\}=\emptyset\ne V_v\cap b\setminus\{v\}$ can be considered by analogy with the case 3c).

Thus we have handled all possible cases and proved the statement $(*)_{n+1}$.
\end{proof}

\begin{remark}Let $M$ be a commutative monoid, $F:\Set\to\Set$ be the functor of $M$-valued finitely supported functions, and $(X,d_X)$ be a distance space.
\begin{enumerate}
\item  If $M$ is a two-element group, then $\check d^1_{FX}=d^1_{FX}$ on the group $FX=F(X,M)$, see Theorem~\ref{t:Gra}(2).
\item  If $M$ is a two-element semilattice, then $\check d^\infty_{FX}=d^\infty _{FX}$ on the semilattice $FX=F(X,M)$, see Theorem~\ref{t:hyper}(2).
\end{enumerate}
\end{remark}

\begin{problem} Find more functors $F:\Set\to\Set$ such that $\check d^p_{FX}=d^p_{FX}$ for some $p\in[1,\infty]$ and all distance spaces $(X,d_X)$.
\end{problem}

\begin{problem}\label{prob:AH} Are there (efficient) algorithms for calculation of the distances  $d^p_{HX}$ and $\check d^p_{HX}$.
\end{problem}

\begin{remark} The answer to Problem~\ref{prob:AH} strongly depends on the geometry of the distance space $(X,d_X)$. For subsets of the Euclidean plane, the problem of calculating the distance $\hat d_{HX}$ is related to the classical Steiner's problem \cite{Steiner} of finding a tree of the smallest length that contains a given finite set. This problem is known \cite{Holby} to be computationally very difficult. On the other hand, for finite subsets of the real line,  there exists an algorithm \cite{MO2} of complexity $O(n\ln n)$, calculating the distance $d^1_{HX}(a,b)$ between two sets $a,b\subset\mathbb R$ of cardinality $|a|+|b|\le n$. Also there exists an algorithm of the same complexity $O(n\ln n)$ calculating the Hausdorff distance $d_{HX}(a,b)$ between two subsets of the real line. Finally, let us remark that the evident brute force algorithm for calculating the Hausdorff distance $d_{HX}(a,b)$ between finite subsets of an arbitrary distance space $(X,d_X)$  has complexity $O(|a|{\cdot}|b|)$.
\end{remark}
\newpage


\end{document}